%% file: Isotopies-46.tex
\documentclass[11pt]{article}
\usepackage{amsmath}
\usepackage{amsthm,amsfonts,amssymb,epsfig,graphics}
\usepackage{color}
\usepackage[utf8]{inputenc}
\usepackage[english,french]{babel} 
\usepackage{hyperref}
\usepackage{array}
\usepackage{tikz}\usetikzlibrary{arrows}

\graphicspath{{}}

\newcommand{\abs}[1]{|#1|}

\def\diam{\mathrm{Diam}}

\def\inte{\mathrm{Int}}
\def\fill{\mathrm{Fill}}
\def\dist{\mathrm{dist}}
\def\clos{\mathrm{Clos}}

\def\fix{\mathrm{Fix}}
\def\id{\mathrm{Id}}

\def\homeo{\mathrm{Homeo}}

\newcommand{\bbR}{{\mathbb{R}}}
\newcommand{\bbN}{{\mathbb{N}}}

\newcommand{\bbZ}{{\mathbb{Z}}}
\newcommand{\bbD}{{\mathbb{D}}}

\newcommand{\bbS}{{\mathbb{S}}}

\def\bbH{{\mathbb{H}}}

\def\cA{{\cal A}}  \def\cG{{\cal G}}  \def\cS{{\cal S}}
    \def\cT{{\cal T}}
  \def\cI{{\cal I}}  \def\cU{{\cal U}}
   \def\cP{{\cal P}} 
    
\def\cF{{\cal F}}

\def\?{$^{***}$\marginpar{?}}

\newtheorem{theo}{Theorem}

\newtheorem*{ques*}{Question}
\newtheorem*{prop*}{Proposition}
\newtheorem*{conj*}{Conjecture}
\newtheorem*{theo*}{Theorem}
\newtheorem{coro}{Corollary}[section]
\newtheorem{theo2}[coro]{Theorem}

\newtheorem{claim}[coro]{Claim}
\newtheorem{fact}[coro]{Fact}
\newtheorem{claim*}{Claim}
\newtheorem{affi*}{Affirmation}
\newtheorem{prop}[coro]{Proposition}
\newtheorem{lemm}[coro]{Lemma}
\newtheorem{sublemm}[coro]{Sub-Lemma}
\newtheorem*{lemm*}{Lemma}
\def\?{\footnote{?}}

\newlength{\espaceavantspecialthm}
\newlength{\espaceapresspecialthm}
\newenvironment{defi}[1][]{\refstepcounter{coro} 
\vskip \espaceavantspecialthm \noindent \textbf{Definition~\thecoro
#1.} }%
{\vskip \espaceapresspecialthm}

\newenvironment{rema}[1][]{\refstepcounter{coro} 
\vskip \espaceavantspecialthm \noindent \textbf{Remark~\thecoro
#1.} }%
{\vskip \espaceapresspecialthm}

\title{Fixed point sets of isotopies on surfaces}
\author{Fran\c cois B\'eguin, Sylvain Crovisier et Fr\'ed\'eric Le Roux}

\begin{document}
\maketitle
\sloppy

\selectlanguage{english}
\begin{abstract}
We consider a self-homeomorphism $h$ of some surface $S$.
A subset $F$ of the fixed point set of $h$ is said to be \emph{unlinked} if there is an isotopy from the identity to $h$ that fixes every point of $F$. With Le Calvez' transverse foliations theory in mind, we prove the existence of unlinked sets that are maximal with respect to inclusion. As a byproduct, we prove the arcwise connectedness of the space of homeomorphisms of the $2$-sphere that preserves the orientation and pointwise fix some given closed connected set $F$.
\end{abstract}

\selectlanguage{french}
\begin{abstract}
Considérons un homéomorphisme $h$ d'une surface $S$. Nous dirons qu'un ensemble $F$ de points fixes de $h$ est non enlacé s'il existe une isotopie de l'identité à $h$ qui fixe tous les points de $F$. Motivés par la théorie des feuilletages transverses de Patrice Le Calvez, nous montrons l'existence d'ensemble non-enlacés qui sont maximaux au sens de l'inclusion. Notre critère de non-enlacement conduit aussi au corollaire suivant : pour toute partie $F$ fermée et connexe de la sphère, l'espace des homéomorphismes de la sphère qui préservent l'orientation et qui fixent tous les points de $F$ est connexe par arcs.
\end{abstract}
\selectlanguage{english}

\tableofcontents


\newpage
\section{Introduction}

\subsection{Main result}

Throughout this text, we consider a connected surface $S$ without boundary, not necessarily compact nor orientable.
Let $\homeo(S)$ be the space of homeomorphisms of $S$, equipped with the topology of uniform convergence on compact subsets of $S$.
An \emph{isotopy} is a continuous path $t \mapsto f_{t}$ from a compact interval $I\subset \mathbb{R}$ to $\homeo(S)$; this isotopy is denoted by $(f_{t})_{t\in I}$ or, more simply, by $(f_t)$. 

Let $F$ be a closed subset of $S$. We denote by $\homeo(S,F)$ the subgroup of $\homeo(S)$ consisting of elements $f$ that fix every point of $F$.
A homeomorphism $f \in \homeo(S,F)$ is  \emph{isotopic to the identity relatively to $F$} if it belongs to the arcwise connected component of the identity in $\homeo(S,F)$, 
in other words if there is an isotopy $(f_{t})_{t\in [0,1]}$ such that $f_{0}=\mathrm{Id}$, $f_{1}=f$, and $f_{t}(x) = x$ for every $x$ in $F$ and $t\in [0,1]$. In this case, the set $F$ is said to be \emph{unlinked} for $f$. Our main technical result is the following criterium of unlinkedness.

\begin{theo}\label{theo.finitely-isotopic}
Let $f \in \homeo(S,F)$. Assume  there exists a dense subset $F_{0}$ of $F$ such that 
every finite subset of $F_{0}$ is unlinked for $f$. Then $F$ is unlinked for~$f$.
\end{theo}

This criterium has applications to the existence of transverse foliations and to the topology of spaces of homeomorphisms. We explain this in the next two subsections.

\subsection{Unlinked continua}

Let $\bbD^2$ be the closed unit disk in the plane, and let $f$ be a homeomorphism of the plane that fixes every point of $\bbD^2$. The celebrated Alexander trick (Proposition~\ref{p.alexander}) provides an isotopy from the identity to $f$ relative to $\bbD^2$ ; more generally, the space $\homeo(\bbR^2,\bbD^2)$ of homeomorphisms that fix every point of $\bbD^2$ is easily seen to be contractible. What happens if we replace $\bbD^2$ by any closed connected  non-empty subset $F$ of the plane? Suppose for simplicity that the complement of $F$ is also connected. When $F$ is locally connected, one may use the Rieman-Caratheodory conformal mapping theorem to transport the Alexander trick, by conjugacy, from the outside of the unit disk to the outside of $F$, and again get an isotopy that fixes every point of $F$. Our first corollary solves the general (non locally connected) case.

\begin{coro}\label{coro.unlinkedness1}
Let $F$ be a closed connected subset of the plane. Then the space 
$$
\homeo^+(\bbR^2, F)
$$
of homeomorphisms that preserve the orientation and fix every point of $F$ is arcwise connected.
In other words, every closed connected subset of the fixed point set of an orientation preserving homeomorphism of the plane is unlinked.
\end{coro}

We conjecture that the space $\homeo^+(\bbR^2, F)$ is always contractible, as far as $F$ contains more than one point. 
The generalization of Corollary~\ref{coro.unlinkedness1} to arbitrary surfaces will be discussed in section~\ref{sec.unlinked-continua}. 

\subsection{Maximal unlinked  sets and transverse foliations}

We describe now the application of Theorem~\ref{theo.finitely-isotopic} that was our main motivation.

\begin{coro}\label{coro.maximal-unlinked-sets}
Let $f \in \homeo(S)$.  Assume $F$ is a closed unlinked set for $f$. Then there exists a closed unlinked set $F'$, containing $F$, and maximal for the inclusion among unlinked sets.
\end{coro}

The maximality amounts to saying that for every isotopy $(f_{t})$ from the identity to $f$, relative to $F'$, and for every fixed point $x_{0}$ of $f$ which is not in  $F'$, the loop $t \mapsto f_{t}(x_{0})$ is not contractible in $S \setminus F'$ (see~Lemma~\ref{l.fibration} below).

\begin{proof}[Proof of Corollary~\ref{coro.maximal-unlinked-sets}]
In order to apply Zorn's lemma, we consider a totally ordered family of closed unlinked sets $(F_{i})_{i \in I}$ containing $F$. Let  $F_{0}$ be the union of the $F_{i}$'s, and $F'$ be the closure of $F_{0}$. A subset of an unlinked set is clearly unlinked, thus every finite subset of $F_{0}$, being included in some $F_{i}$, is unlinked. Theorem~\ref{theo.finitely-isotopic} entails that $F'$ is unlinked. Thus every totally ordered family of closed unlinked sets is included in a closed unlinked set: we may apply Zorn's lemma, and we get some closed maximally unlinked set $F'$ containing $F$, as required by the corollary.
\end{proof}

To explain the interest of the maximal unlinked sets provided by Corollary~\ref{coro.maximal-unlinked-sets}, let us recall Jaulent's and Le Calvez's results (\cite{jaulent,lecalvez05}). Let $F$ be a maximally unlinked set for $f$, and $(f_{t})$ be an isotopy from the identity to $f$ that fixes every point of $F$. In this context, Le Calvez has proved that there exists an oriented foliation $\cF$ on $S \setminus F$ which is \emph{transverse to the isotopy $(f_{t})$}: this means that the trajectory of every point $x$ outside $F$ is homotopic in $S \setminus F$, relative to its endpoints, to a curve which crosses every oriented leaf of the foliation from left to right. Le Calvez's theorem is a extremely powerful tool to study surface dynamics (see for example the introduction of~\cite{lecalvez-tal} and the references within). The general existence of maximally unlinked sets is a crucial auxiliary tool for Le Calvez's theorem: it extends its range of action from some specific cases, when the existence of maximally unlinked sets is obvious (for instance when the fixed point set is assumed to be finite, as in Le Calvez's proof of Arnol'd's conjecture in~\cite{lecalvez05}) or easy (for instance for diffeomorphisms, see the appendix of~\cite{humiliere}), to the case of every homeomorphism which is isotopic to the identity.

We have to point out that  Jaulent has already provided an auxiliary tool which is enough for most applications of Le Calvez's theorem, if less satisfactory (\cite{jaulent}).
Namely, he has proved the existence of ``weakly maximally unlinked sets'' for $f$, which are closed subsets $F$ with the following property:
on $S \setminus F$, there is an isotopy from the identity to the restriction of $f$, and no trajectory for this isotopy is a contractible loop. Corollary~\ref{coro.maximal-unlinked-sets} above is a strengthening of Jaulent's result, since we get an isotopy which extends continuously to the identity on $F$. It answers Question 0.2 of Jaulent's paper.
This stronger version is needed for example in the work of Yan on torsion-low isotopie (\cite{yan}).
Another motivation for the present work was to clarify the links between various natural notions of ``unlinkedness''; we will discuss this in the next subsection.

\bigskip

We end the discussion of maximal unlinked sets with a version of Corollary~\ref{coro.maximal-unlinked-sets} that takes into account a given isotopy. Let $f \in \homeo(S)$.  We consider couples $(F,(f_{t}))$ where $F$ is a closed set of fixed points of $f$, and $(f_{t})$ is an isotopy from the identity to $f$ relatively to $F$. The set $\cS$ of such couples is endowed with the order relation defined by $(F,(f_{t})) \leq (F', (f'_{t}))$ when $F \subset F'$ and for every point $z \in S \setminus F$, the trajectories of $z$ under both isotopies $(f_{t})$ and $(f'_{t})$ are homotopic in $S \setminus F$. Note that it suffices to check the compatibility between $(f_{t})$ and $(f'_{t})$ on one point in each connected component of $S\setminus F$ (for example by requiring that the trajectory under $(f_{t})$ of some point of $F' \setminus F$ is a contractible loop in $S \setminus F$). Actually, in most cases, the compatibility between isotopies is automatic: for instance,  for every couples $(F,(f_{t}))$ and $(F',(f'_{t}))$ of $\cS$ such that $F$ has at least three points, if $F$ is included in $F'$ then $(F,(f_{t})) \leq (F', (f'_{t}))$. This ``uniqueness of homotopies'' will be a key ingredient in the construction of isotopies. It also allows the following improvement on Corollary~\ref{coro.maximal-unlinked-sets}  (see the end of section~\ref{ss.uniqueness} for the proof).

\begin{coro}\label{coro.maximal-isotopies}
For every  $(F,(f_{t})) \in \cS$ there exists a maximal element $(F', (f'_{t})) \in \cS$ such that 
$(F,(f_{t})) \leq (F', (f'_{t}))$.
\end{coro}

\subsection{Fifty ways to be unlinked}

Consider as before a homeomorphism $f$ on a connected surface $S$, and some closed set $F$ of fixed points of $f$. If $F$ is unlinked for $f$, then $h = f_{\mid S \setminus F}$ is isotopic to the identity as a homeomorphism of $S \setminus F$. The converse is not true, as shown by the example of a Dehn twist in the closed annulus with $F$ equal to the boundary. We will say that $F$ is \emph{weakly unlinked} if $h$ is isotopic to the identity on $S \setminus F$.  Both weak and strong unlinkedness turn out to have many equivalent formulations. Our main task in this text will be to prove the most difficult implications between these formulations. We begin by defining the various formulations of weak unlinkedness in paragraph (a), and of strong unlinkedness in paragraph (b) below. In paragraph (a) we will think of $M$ as $S \setminus F$.

\paragraph{(a) Homotopies and isotopies on the complement of $F$}
\label{subsection.weakly-unlinked}
Let  $M$ be a connected surface without boundary, and $h$ be a homeomorphism of $M$. Remember that $h$ is said to be \emph{homotopic to the identity} if there exists a \emph{homotopy}, \emph{i. e.} a continuous map $H : M \times [0,1] \to M$, such that $H(.,0) = \mathrm{Id}$ and $H(.,1)=h$. If the map $H$ can be chosen to be proper (inverse images of compact subsets are compact), then $h$ is said to be \emph{properly homotopic to the identity}. 
A loop is a continuous map from the circle $\bbS^1$ to $M$; two loops $\alpha,\beta$ are said to be \emph{freely homotopic} if there exists a map $H : \bbS^1 \times [0,1] \to M$ such that $H(.,0)=\alpha$ and $H(.,1) = \beta$.
We consider the following properties on $h$.

\begin{enumerate}
\item[(W1)] \label{weakly-unlinked.isotopic} \emph{The map $h$ is isotopic to the identity.}
\item[(W2)] \label{proper.homotopic}
\emph{The map $h$ is properly homotopic to the identity.}
\item[(W3)] \label{weakly-unlinked.homotopic} \emph{The map $h$ is homotopic to the identity.} 
\item[(W4)] \label{weakly-unlinked.pi1} 
\emph{The map $h$ lifts to a homeomorphism $\tilde h$ of the universal cover of $M$  that commutes with the automorphisms of the universal cover. Equivalently, $h$ acts trivially on $\pi_{1}(M)$.\footnote{More formally, the outer automorphism of $\pi_{1}(M,*)$ induced by $h$ is trivial, see subsection~\ref{ss.w5w4}.}}
\item[(W5)] \label{weakly-unlinked.curves} \emph{Every loop is freely homotopic to its image under $h$}.
\end{enumerate}

\begin{theo}[mostly due to D. Epstein]
\label{theo.equivalences-weak}~
 Let $M$ be a connected surface without boundary, and $h$ be a homeomorphism of $M$.
 Then properties (W1), (W2), (W3), (W4), (W5) are equivalent, with the following exceptions  (see Remark~\ref{r.exceptions}):
\begin{itemize}
\item  (W3) does not implies (W2) when $M$ is the plane or the open annulus.
\item  (W4) does not implies (W3) when $M$ is the sphere.
\end{itemize}
\end{theo}

The implication (W5) $\Longrightarrow$ (W4) is purely algebraic, and due to Grossman and Allenby, Kim, and  Tang (\cite{grossman,allenby}). The implication  (W4) $\Longrightarrow$ (W3) is standard in homotopy theory,  see for example \cite[Proposition 1.B.9]{hatcher}. The most difficult part of the Theorem, namely that property (W2) implies property (W1), is proved by Epstein in his classical work~\cite{epstein}. We will provide a complete proof (spread in Propositions~\ref{p.W5-W4}, \ref{p.w4Tow3}, \ref{p.homotopic-implies-proper}, \ref{p.projective-plane} and Section~\ref{ss.empty-F}).

\medskip

\begin{rema}\label{r.exceptions}
The exceptional cases may be further analyzed by considering the ends of $M$ and the orientation (see sections~\ref{ss.ends} and~\ref{ss.orientation}).
Propositions~\ref{p.w4Tow3} and \ref{p.sphere} will imply:
\begin{itemize}
\item \emph{When $M$ is the plane or the sphere, any orientation-preserving homeomorphism of $M$ is properly homotopic to the identity (hence, (W4), (W3) and (W2) are equivalent  when one restricts to orientation-preserving homeomorphisms).}
\item \emph{When $M$ is the open annulus, (W3) implies (W2) if one restricts to homeomorphisms fixing each of the two ends of $M$.}
\end{itemize}
An orientation-reversing  homeomorphism of the plane is homotopic to the identity but is not properly homotopic to the identity. Similarly, a homeomorphism of the open annulus $\mathbb{A}=\mathbb{S}^1\times\mathbb{R}$ which interchanges the two ends $\mathbb{A}$ is homotopic, but not properly homotopic, to the identity. Finally, an orientation-reversing homeomorphism of the sphere is not homotopic to the identity, although it acts trivially on the (trivial!) fundamental group. 
\end{rema}

\paragraph{(b) Homotopies and isotopies relative to $F$}
\label{subsection.strongly-unlinked}

Let $S$ be a surface with a homeomorphism $f$, and  $F$ be a closed set of fixed points of $f$.
We will say that $f$ is \emph{strongly homotopic to the identity relatively to $F$} 
if there exists a \emph{stron homotopy relative to $F$}, \emph{i. e.} a map
$H : S \times [0,1] \to S$ such that 
\begin{itemize}
\item $H(.,0) = \mathrm{Id}$, $H(.,1) = f$,
\item  $H(x,t) \not \in F$ for every $x \not \in F$ and every $t$,
\item and $H(x,t) = x$ for every $x \in F$ and every $t$.
\end{itemize}
In other words, 
 the restriction $f_{\mid S \setminus F} : S \setminus F \to S \setminus F$ is homotopic to the identity, and the homotopy extends continuously to $F$ as a map that fixes every point of $F$. 
We consider the following properties on $f$ and $F$.

\begin{enumerate}
\item[(S1)] \label{strongly-unlinked.isotopic} \emph{$f$ is isotopic to the identity relatively to $F$.}
\item[(S2)] \label{strongly-unlinked.homotopic} \emph{$f$ is strongly homotopic to the identity relatively to $F$.}
\item[(S3)] \textbf{(the ``finitely homotopic'' criterion)} \label{strongly-unlinked.finitely-homotopic} \emph{There exists a dense subset $F_0$ of $F$ such that for every finite subset $F'$ of $F_0$, $f$ is  strongly  homotopic to the identity relatively to $F'$.}
\end{enumerate}

 Note that, clearly, Property (S1) implies Property (S2), and Property (S2) implies Property (S3). The main task of the paper is to prove the converse implications. The following statement clearly implies Theorem~\ref{theo.finitely-isotopic}.

\bigskip

\begin{theo}\label{theo.equivalences-strong}
Let $S$ be a connected surface without boundary, $F$ a closed non-empty subset of $S$, and $f \in \homeo(S,F)$. Then Properties (S1), (S2), (S3) are equivalent.
\end{theo}

\paragraph{(c) The totally discontinuous case}
In the case when $F$ is totally discontinuous, every isotopy on the complement of $F$ extends to an isotopy of $S$ fixing every point of $F$. Thus in this case weak and strong unlinkedness are equivalent, as expressed by the following proposition.
\begin{prop}
Let $S$ be a connected surface, $F$ a closed non-empty subset of $S$, and $f \in \homeo(S,F)$. Let  $M= S \setminus F$, and $h$ be the restriction of $f$ to $M$. 

Assume that  $F$ is totally discontinuous. Then property (S1) for  $f$ is equivalent to property (W1) for $h$.
\end{prop}

It is not true in general that property (W1) for $h$ implies property (S1) for $f$. The most elementary counterexample is given by a twist map on an annulus, with $F$ equal to the boundary of the annulus: the restriction of $f$ to the interior of the annulus is isotopic to the identity, but $f$ is not isotopic to the identity relative to its boundary. It is plausible that every counter-example is a variation on this one.

\paragraph{(d) The braid viewpoint}

Property (S3), which appears in Theorem~\ref{theo.finitely-isotopic}, reduces the problem of unlinkedness to the case of finite sets. Unlinkedness of finite sets can be characterized using braids, as follows (this will not be used anywhere else in the text). Let $F$ be a closed subset of a surface $S$. A \emph{geometric pure braid based on $F$} is a continuous map $b : F \times [0,1] \to S$ with $b(x,0)=b(x,1)=x$ for every $x$ in $F$, and such that $x \mapsto b(x,t)$ is injective for every $t$. A geometric pure braid $b$ \emph{represents the trivial braid} if there exists a continuous map $B : F \times [0,1] \times [0,1] \to S$ such that $B(.,.,0)= b$, $B(.,.,1)$ is the constant braid $(x,t) \mapsto x$, and $B(.,.,s)$ is a geometric pure braid for every $s$. Now consider a homeomorphism $f$ of $S$ that fixes every point of $F$. Assume $f$ is isotopic to the identity, and choose any isotopy $(f_{t})$ from the identity to $f$. This isotopy generates the geometric pure braid $b_{F,(f_{t})} : (x,t) \mapsto f_{t}(x)$. The following criterium is an easy consequence of the fact that for any given point $x_{0}$, the map that takes a homeomorphism $h$ to the image of $x_{0}$ under $h$ is a fiber bundle (see Lemma~\ref{l.fibration}).

\begin{prop}
Assume that $F$ is a finite set. If the geometric pure braid $b_{F,(f_{t})}$ represents the trivial braid, then $f$ is isotopic to the identity relatively to $F$.
\end{prop}

\subsection{Strategy for Theorem~\ref{theo.equivalences-strong}}\label{ss.strategy}

Our main task is to prove the non-trivial implications of Theorem~\ref{theo.equivalences-strong}, namely that ``finitely homotopic'' implies strongly homotopic, \emph{i. e.}  (S3) implies (S2), and that trongly homotopic implies isotopic, \emph{i. e.}  (S2) implies (S1).

To prove (S3) implies (S2), we need to build homotopies. In order to build a homotopy between $f$ and the identity, one chooses continuously for each point $z$ in the surface a path $\gamma_z$ between $z$ and $f(z)$. This is done in two steps: first identify the homotopy class $C_z$ of paths in $S \setminus F$ which will contain $\gamma_z$ and then select continuously the path $\gamma_z$ in the class $C_z$. The hypothesis (S3) provides an increasing family $(F_{n})_{n \geq 0}$ of finite subsets of $F$, whose union is dense in $F$, and for each $n$ a homotopy $I_{n}$ relative to $F_{n}$. For the first step, fix a point $z$ in the complement of $F$. Then we prove that up to homotopy, there exists a unique paths from $z$ to $f(z)$ in the complement of $F$ which for each $n$ is homotopic in the complement of $F_{n}$ to the trajectory of $z$ under the homotopy $I_{n}$ (Lemma~\ref{l.uniqueness-Cz}); for this the essential ingredient is the ``uniqueness of homotopies'' which is discussed in section~\ref{ss.uniqueness}.  For the second step, we apply a selection technique due to E. Michael, which provides a continuous selection for a multivalued map under some very general assumptions. 

\bigskip

To prove (S2) implies (S1), we need to build isotopies. More precisely, we have a non-empty closed set $F$ of fixed points of $f$ and we want to build an isotopy relative to $F$ between the identity and $f$, under the hypothesis that there is a strong homotopy. To explain the strategy, let us assume for simplicity that $S$ is the closed unit square, and $F$ is any closed subset of $S$. We take the equivalent viewpoint of constructing an isotopy, relative to $F$,  from the identity to the inverse of $f$. Then the strategy will go as follows (see Figure~\ref{f.strategy}). Take the middle vertical segment $\alpha_{1}$, and we find a first isotopy $I_{1}$, that does not move the points of $F$, from the identity to a homeomorphism $f_{1}$ that brings $\alpha_{1}$ back into place, namely such that  $f_{1} f $ fixes $\alpha_{1}$. Now take the middle horizontal segment $\alpha_{2}$ and find a second isotopy $I_{2}$ from $f_{1}$ to a homeomorphism $f_{2}$, that does not move the points of $ F_1:=F \cup \alpha_{1}$, and brings $\alpha_{2}$ back into place. Go on this way, bringing successively back into place  a sequence of segments $\alpha_{n}$ that cut the square into pieces with smaller and smaller diameters. The infinite concatenation of all these isotopies is easily seen to converge to the inverse of $f$, as wanted. 

\begin{figure}[ht]
\def\svgwidth{\textwidth}
\begin{center}
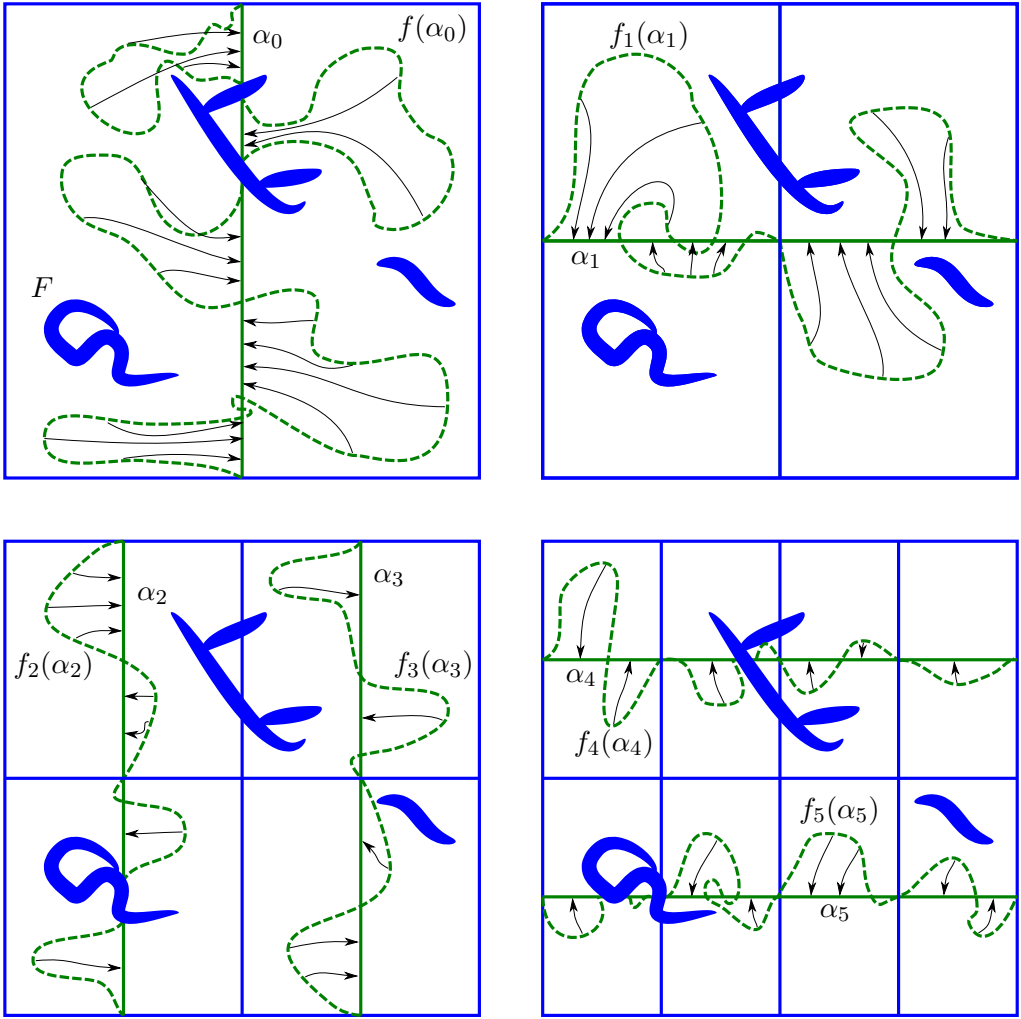
\caption{\label{f.strategy} Proof of (S2)$\Rightarrow$(S1): the first steps in the construction of the isotopy.}
\end{center}
\end{figure}

\bigskip

With this strategy in mind, it is clear that the key step in the above explanation is the existence of each isotopy $I_{n}$. More precisely, we need the two following properties:

\begin{itemize}
\item  Given that $f$ is strongly homotopic to the identity, and given any arc $\alpha$, there is an isotopy $(f_{t})$ relative to $F$  such that $f_{1} f(\alpha) = \alpha$. We call this the ``Arc Staightening Lemma'' (Proposition~\ref{l.multicoupure}). The proof of this lemma is probably the novelty of this paper in terms of techniques.
\item With the same notations, the map $f_{1} f$ is strongly homotopic to the identity relatively to $F \cup \alpha_{1}$. This will be the content of Lemma~\ref{l.still-homotopic}.
\end{itemize}

Finally, let us give a hint about the ``Arc Straightening Lemma''. Again the required isotopy will be obtained as an infinite concatenation of more elementary isotopies. As a first rough approximation, let us pretend that each elementary isotopy consists in bringing back a subarc of $\alpha$ which is (the closure of) a connected component $\beta$ of $\alpha \setminus F$. Before doing this there is  a preliminary technical step to get the additional property that $\beta$ coincides with $f(\beta)$ in a neighborhood of its endpoints, and is transverse everywhere else. With this additional property, the isotopy that sends $f(\beta)$ back to $\beta$ is constructed by  ``pushing bigons'' to successively remove all the transverse intersection points between $f(\beta)$ and $\beta$. This technique goes back at least to Epstein's paper. We have some constraints on  the support of the isotopy, in order to ensure the convergence of the infinite concatenation of our elementary isotopies (the precise statement is Proposition~\ref{p.straightening-arc-easy}). One key issue is to ensure the convergence of the infinite concatenation of all these elementary isotopies, each bringing back into place one connected component of $\alpha \setminus F$. If no precaution is taken, then the trajectory of a point outside $F$ could be pushed closer and closer to $F$ by each elementary isotopy, thus converging to a point in $F$, and then the concatenation would converge to a non-injective map. In the proof the elementary isotopies are concatenated in a carefully chosen order (using the above additional property), so that each point travels only under a finite number of elementary isotopies, thus circumventing this pitfall.

\subsection{Organization of the paper}

The diagram on Figure~\ref{fig.implications} shows where the various implications are proven.

\begin{figure}[p]
\def\longueur{5cm}
\def\intercol{3cm}
\def\interligne{2.5cm}
\def\hauteurTitre{0cm}
\vspace{-1.5cm}
\hspace{-4cm}
\begin{center}
\begin{tikzpicture}[>=latex]
\node[text width=\longueur,text badly centered] at(0,\hauteurTitre){\bf \Large Strong unlinkedness \normalsize \\ $f:S \to S$,  \ \  $F\subset \mathrm{Fix}(f)$};
\node[text width=\longueur,text badly centered] at(\longueur+\intercol,\hauteurTitre){\bf \Large Weak unlinkedness \normalsize \\ $M = S \setminus F, \ \ h = f_{\mid M}$};
\node[draw,text width=\longueur,text badly centered] (S1)  at(0,-\interligne){(S1)  $f$ is isotopic to the identity relatively to $F$};
\node[draw,text width=\longueur,text badly centered] (W1) at(\longueur+\intercol,-\interligne){(W1) $h$ is isotopic to the identity};
\node[draw,text width=\longueur,text badly centered] (S2) at(0,-3*\interligne){(S2) $f$ is strongly homotopic to the identity relatively to $F$ };
\node[draw,text width=\longueur,text badly centered] (W2)  at(\longueur+\intercol,-2*\interligne){(W2) $h$ is properly homotopic to the identity};
\node[draw,text width=\longueur,text badly centered] (S3)  at(0,-5*\interligne){(S3)  $f$ is strongly homotopic to the identity relatively to every finite subset of a dense subset};
\node[draw,text width=\longueur,text badly centered] (W3)  at(\longueur+\intercol,-3*\interligne){(W3) $h$ is homotopic to the identity};
\node[draw,text width=\longueur,text badly centered] (W4)  at(\longueur+\intercol,-4*\interligne){(W4) $h$ acts trivially on $\pi_{1}(M)$};
\node[draw,text width=\longueur,text badly centered] (W5)  at(\longueur+\intercol,-5*\interligne){(W5) every loop is freely homotopic to its image};

\tikzstyle{implique}=[thick,-implies,double equal sign distance,shorten >=0.05cm,shorten <=0.05cm]
\tikzstyle{impliqueTrivial}=[-implies,thin,double equal sign distance,shorten >=0.05cm,shorten <=0.05cm]
\draw[impliqueTrivial]  (S1) -- (W1) ; 
\draw[impliqueTrivial]  (S1.340) -- (S2.20) ; 
\draw[impliqueTrivial]  (S2.340) -- (S3.33) ; 
\draw[impliqueTrivial]  (W1.210) -- (W2.150) ;
\draw[impliqueTrivial]  (W2.210) -- (W3.150) ;
\draw[impliqueTrivial]  (W3.210) -- (W4.150) ;
\draw[impliqueTrivial]  (W4.210) -- (W5.150) ;
\draw[implique]  (S3) -- (S2) node[midway,left,text width=2cm] {\footnotesize selection of trajectories (sections 3,4)} ;
\draw[implique]  (S2) -- (S1) node[midway,left,text width=3cm] {\footnotesize arcs straightening (sections 5,6,7,8)} ;
\draw[implique]  (W5) -- (W4) node[midway,right,text width=4cm] {\footnotesize algebra of surface groups (section 2)} ;
\draw[implique]  (W4) -- (W3) node[midway,right,text width=4cm] {\footnotesize selection of trajectories (section 3)} ;
\draw[implique]  (W3) -- (W2) node[midway,right,text width=4cm] {\footnotesize selection of trajectories (section 3)} ;
\draw[implique]  (W2) -- (W1) node[midway,right,text width=4cm] {\footnotesize Epstein, arc straightening (section 8)} ;

\draw[thick,-implies,double equal sign distance,dashed,shorten >=0.2cm,shorten <=0.2cm] 
(W1) to[bend right=20] 
node[midway,above,text width=3cm] {\footnotesize when $F$ is  totally disconnected} (S1) ;
\end{tikzpicture}
\end{center}
\caption{Organization of the paper. The uncommented arrows denote the trivial implications}
\label{fig.implications}
\end{figure}
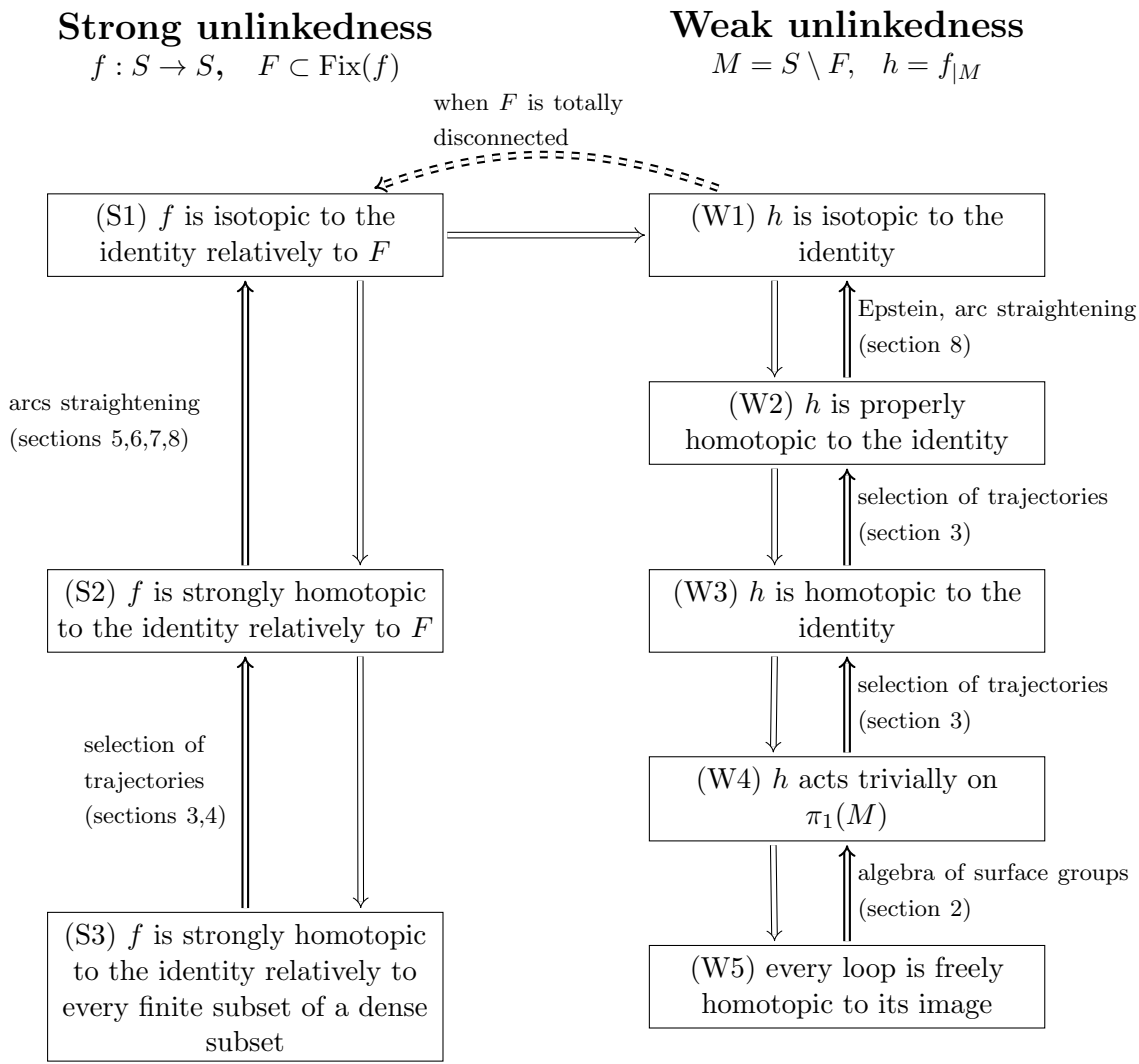

In section~\ref{s.algebra},  we collect the needed algebraic properties of fundamental groups of surfaces. In particular in~\ref{ss.w5w4} we recall the main arguments for the proof that (W5) implies (W4) (Proposition~\ref{p.W5-W4}). These algebraic properties are also used in~\ref{ss.uniqueness} to prove the key observation that homotopies are essentially unique.

\bigskip

Then we start building homotopies. 

\bigskip

In section~\ref{sec.selection}, we prove a general criterion to select a curve continuously in a given homotopy class. We first use this in~\ref{ss.constructions-homotopies} to construct a homotopy, proving that (W4) implies (W3) (Proposition~\ref{p.w4Tow3}). Then again in~\ref{ss.proper-homotopies} to transform a homotopy into a proper homotopy, proving that (W3) implies (W2)  (Proposition~\ref{p.homotopic-implies-proper}). Some alternative proofs using hyperbolic geometry are provided in~\ref{sub.hyperbolic-geometry}.

In section~\ref{sec.fini-homo-implique-homo}, assuming property (S3), we first build a homotopy on the complement of $F$. Then we modify this homotopy so that it extends to $F$, thus getting a strong homotopy and proving that (S3) implies (S2) (Proposition~\ref{prop.finitely-homotopic}). Here the selection criterion of the previous section is essential.

\bigskip

From then on, we undertake to build isotopies. 

\bigskip

In section~\ref{s.isotopy-arc}, we show how to ``push bigons'' in order to construct an isotopy relative to $F$, sending an arc $\beta$ to another arc $\beta'$ which is assumed to be homotopic to $\beta$, in the easy case where $\beta$ is essentially disjoint from $F$. In section~\ref{sec.redresser-les-courbes} we do the same in the case where $\beta$ is a simple closed curve.

In section~\ref{sec.arc-straightening} we prove the Arc Straightening Lemma (see subsection~\ref{ss.strategy} for a brief description of the content of this Lemma).

In section~\ref{sec.homo-implique-iso} we use the Arc Straightening Lemma to  build an isotopy from a homeomorphism to the identity, thus proving that  (S2) implies (S1) (Proposition~\ref{p.strong-homotopy-implies-strong-isotopy}). We also use it to give a proof that (W2) implies (W1) in~\ref{ss.empty-F}.

In section~\ref{sec.unlinked-continua} we prove Corollary~\ref{coro.unlinkedness1} concerning unlinked continua and its generalization to every surface.

\newpage
\part{Homotopies}

\section{Algebra of surface groups and applications}\label{s.algebra}
Throughout the paper, we consider curves as maps from an interval or from the circle to the surface. A (free) homotopy from a curve $\alpha: I \to S$ to a curve $\beta: I \to S$ is a continuous maps $H : I \times [0,1] \to S$ such that $H(.,0)=\alpha$ and $H(.,1)=\beta$.

\subsection{Surface groups}\label{ss.surface-groups}
We collect some useful facts about fundamental groups of (compact and non compact, orientable and non orientable) surfaces. Fundamental groups of all surfaces are computed for instance in ~\cite[section 4.2]{stillwell2012classical}, including  a proof that the fundamental group of a non compact surface is always a free group.

%

\begin{prop}\label{prop.algebra-fundamental-group}
Let $G$ be the fundamental group of some connected surface $S$ without boundary. Then
\begin{enumerate}
\item if $S$ is not the torus nor the Klein bottle, then every abelian subgroup of $G$ is cyclic.
\item if $S$ is not the projective space, the torus, the Klein bottle, the annulus or the M\"obius band, then the center of $G$ is trivial.
\end{enumerate}
\end{prop}

\begin{proof}
Assume $G$ contains some abelian subgroup $H$ of rank greater than 1. By the covering spaces classification theorem, there is a surface $S'$ which is a covering of $S$ and has fundamental group isomorphic to $H$. By the classification of fundamental groups of surfaces, $S'$ is a torus. In particular $S$ is a compact surface. Furthermore, since a covering maps multiplies the Euler characteristic by some integer, the Euler characteristic of $S$ is zero, hence $S$ is a torus or a Klein bottle.

Now assume the center of $G$, say $Z$, is non trivial. Let $S'$ be a covering of $S$ with fundamental group isomorphic to $Z$. Then $S'$ is a surface with non trivial abelian fundamental group, thus a torus, a projective space, an annulus or a M\"obius band. The surface $S$ is a quotient of one of these, thus it belongs to the list of the five exceptional cases of the second point of the proposition.
\end{proof}

An element $\alpha$ in a group $G$ is called \emph{primitive} if it has no root, \emph{i. e. } there is no element $\beta$ and integer $n\geq 2$ such that $\beta^n = \alpha$.

\begin{prop}\label{prop.algebra-fundamental-group2}
 If $\alpha$ is a primitive element in the fundamental group $G$ of a non compact surface, and if $\beta$ commutes with $\alpha$, then $\beta$ is a power of $\alpha$.
\end{prop}

\begin{proof}
The subgroup generated by $\alpha$ and $\beta$ is abelian, hence it is a cyclic group according to the first point of the previous proposition. Thus $\alpha$ and $\beta$ are powers of a commun element, but since $\alpha$ is primitive, $\beta$ is a power of $\alpha$.
\end{proof}

The following proposition says that an essential simple closed curve on an orientable surface is always primitive in the fundamental group, and describes the only exceptions in the non-orientable case.

\begin{prop}\label{prop.primitive}
Let $\alpha$ be an essential simple closed curve on a surface $S$. Assume that there exists a closed curve  $\beta$ and an integer $n \geq 2$ such that $\beta^n$ is freely  homotopic to $\alpha$. Then $\alpha$ bounds a compact M\"obius strip in $S$, $n=2$, and $\beta$ is freely homotopic to the core of the M\"obius strip.
\end{prop}

\begin{proof}(See also~\cite{farb-margalit} for a proof in the orientable hyperbolic case.)
The existence of $\alpha$ implies that $S$ is not the sphere, and the existence of a root $\beta$ of $\alpha$ implies that $S$ is not the projective plane. Thus the universal cover $X$ of $S$ is homeomorphic to the plane. Let $T_{\beta}$ be an automorphism corresponding to $\beta$: there is a lift of $\beta$ which is invariant under $T_{\beta}$. This automorphism acts freely and properly discontinuously on the plane, thus it is conjugate either (1) to a translation or (2) to the composition of a translation and a symmetry. We may lift the free homotopy between $\beta^n$ and $\alpha$, thus getting a lift $\tilde \alpha$ of $\alpha$ joining a point and its image by $T_{\beta}^n$. If $T_{\beta}$ is a translation, then the quotient $X/T_{\beta}$ is an open annulus, and the projection of $\tilde \alpha$ is a simple closed curve which makes $n$ turns around this annulus. This is impossible. Thus $T_{\beta}$ is the composition of a translation and a symmetry, and the quotient $X/T_{\beta}$ is an open M\"obius band.

We claim that $n=2$. Assume by contradiction that $n > 2$. Consider any point on $\tilde \alpha$, its image under $T_{\beta}^n$ is another point on $\tilde \alpha$; both points bound a closed subarc $A$ of $\tilde \alpha$. Since $\alpha$ is a simple closed curve homotopic to $\beta^n$, this arc $A$ is disjoint from all its iterates $T_{\beta}^p A$ except for $p=-n,0,n$. In particular, $A$  $T_{\beta}A$ and $T_{\beta}^2A$ are pairwise disjoint, but $T_{\beta}^n A$ meets $A$. This contradicts Bonino's Corollary 3.11 in~\cite{bonino}\footnote{One can also get a contradiction by using the uniformization theorem, the classification of the isometries of the eulidean (resp. hyperbolic) plane.}. Hence $n=2$

Now since $T_{\beta}^2$ is an automorphism corresponding to $\alpha$, we deduce that $\alpha$ is two-sided (\emph{i. e.} it has a tubular neighborhood homeomorphic to an annulus, not a M\"obius band).

Consider the orientation covering $Y$ of $S$, which is 2 to 1.
Since $\alpha$ is two-sided, its inverse image in $Y$ has two components.
Since $\beta$ is one-sided,  its inverse image has a single component. We may lift the free homotopy between $\alpha$ and $\beta^2$, starting with any of the two lifts of $\alpha$, and this provides a homotopy between the given lift of $\alpha$ and the lift of $\beta$. Thus both lifts of $\alpha$ in $Y$ are essential simple closed curves which are homotopic and disjoint. Thus they bound an annulus (Lemma~\ref{l.bound-annulus}). The projection of this annulus in $X$ provides a M\"obius band bounded by $\alpha$, whose core is freely homotopic to $\beta$.
\end{proof}

\begin{lemm}\label{lemm.groupe-fonda-sous-surface}
Let $S$ be a surface, $S'$ be a subset of $S$ which is a surface with boundary, bounded by curves none of which is contractible in $S$.
Then the natural map $\pi_{1}(S') \to \pi_{1}(S)$ induced by the inclusion $S' \hookrightarrow S$ is injective.
\end{lemm}

\begin{proof}
Let $X$ be the universal cover of $S$, and $\tilde S' \subset X$ a connected component of the inverse image of $S'$ in $X$. The space $\tilde S'$ is a surface bounded by lifts of the boundary curves of $S'$. Since each boundary curve is essential, its lifts are topological proper lines in $X$. Thus $\tilde S'$, being an intersection of topological half-planes, is simply connected. This implies that $\tilde S'$ is a universal cover of $S'$. The lemma follows.
\end{proof}

\subsection{Free homotopies of loops and outer automorphisms}\label{ss.w5w4}

In this section, $h$ denotes a homeomorphism of a connected surface $M$ without boundary. We will prove that Property (W5) implies Property (W4). Property (W5) says that every loop  is freely homotopic to its image. 

Let us first explain more precisely the meaning of the phrase ``$h$ acts trivially on $\pi_1(M)$" in (W4).
We choose a point $x_{0}$ in $M$. The homeomorphism $h$ induces an isomorphism $h_{*}$ between $\pi_{1}(M,x_{0})$ and $\pi_{1}(M,h(x_{0}))$. Similarly, any curve $\alpha$ joining $x_{0}$ to $h(x_{0})$ induces isomorphism $\alpha_{*}: \gamma \mapsto \alpha \star \gamma \star \alpha^{-1}$ between $\pi_{1}(M,x_{0})$ and $\pi_{1}(M,h(x_{0}))$. We say that \emph{$h$ acts trivially in $\pi_{1}(M)$} if we can choose a curve $\alpha$ such that $h_{*} = \alpha_{*}$. Note that given any curve $\alpha$ from $x_{0}$ to $h(x_{0})$, we get the automorphism $\alpha_{*}^{-1} h_{*}$ of the group $\pi_{1}(M,x_{0})$. Changing $\alpha$ amounts to composing this automorphism by an \emph{inner automorphism} (a conjugacy). Thus the action of $h$ on $\pi_{1}(M)$ is only defined up to a composition with a conjugacy. Formally, $h$ induces an \emph{outer automorphism} of $\pi_{1}(M)$, an element of the quotient of the group of automorphisms by inner automorphisms. 

Now, let us explain the equivalence between the two formulations of (W4). If $\tilde h$ is a lift of $h$ to the universal cover $\tilde M$ that commutes with the covering automorphisms, then we get a curve $\alpha$ so that $\alpha_*=h_*$ by projecting any curve joining a point $\tilde x_{0}$ to its image $\tilde h (\tilde x_{0})$. Conversely, if $\alpha$ is a curve joining $x_{0}$ to $h(x_{0})$ so that $h_{*} = \alpha_{*}$, then we consider a  $\tilde\alpha$ is a lift of $\alpha$. It joins a lift $\tilde x_0$ of $x_0$ to a lift $\tilde y_0$ of $y_0:=h(x_0)$. Then the lift $\tilde h$ of $h$ mapping $\tilde x_0$ to $\tilde y_0$ commutes with the covering automorphisms. 

\begin{prop}
\label{p.W5-W4}
Let $h$ be a homeomorphism of a connected surface $M$.
If every loop in $M$ is freely homotopic to its image under $h$
then  $h$ acts trivially on $\pi_{1}(M)$.
\end{prop}

\begin{proof}
Assume every loop is freely homotopic to its image. Choose some curve $\alpha$ from a point to its image. Let $\alpha_{*}, h_{*} : \pi_{1}(M,x_{0}) \to \pi_{1}(M,h(x_{0}))$ be the isomorphism induced by $\alpha$ and $h$ respectively. Consider the automorphism $\Phi = \alpha_{*}^{-1} h_{*}$ of $G=\pi_{1}(M,x_{0})$.  The hypothesis entails that every element $g$ of $G$ is conjugate to its image $\Phi(g)$: there is an element $c_g$ of $G$ such that $\Phi (g) = c_g g c_g^{-1}$. We want to prove that $\Phi$ is an inner automorphism, \emph{i. e. } that there exists $c \in G$ such that $\Phi (g) = c g c^{-1}$ for every $g\in G$ (the same $c$ for all $g$'s). 
 
 This purely algebraic statement may be found in the following papers. If $M$ is a non compact surface, then $\pi_{1}(M)$ is a free group, and the result is Lemma 1 of~\cite{grossman}. If $M$ is a compact orientable surface, then the result is again proved in~\cite{grossman}. (In this paper a group is said to have property A if every automorphims that sends each element to a conjugate is an inner automorphism; the surface groups are denoted $H_{k}$, and the proof that $H_{k}$ has property A is exactly the proof of Theorem 3 of the paper.)  If $M$ is a compact non orientable surface, then the result is Theorem 3.2 of~\cite{allenby}.

\bigskip

Since non compact surfaces are the interesting case for the present paper, and the argument is short, let us recall it. 
In this case $G$ is isomorphic to a free group. If $G$ is generated by a single element then the result is obvious. Thus in what follows we assume that $G$ is a free group of rank at least $2$.

\begin{lemm} 
In the free group $G=<a,b>$, every element which is conjugate to $ab$ may be written   $wabw^{-1}$ or $wbaw^{-1}$ as a reduced word in the letters $a,b$, for some word $w$.
\end{lemm}

\begin{proof}
Let $c \in G$, and write $c$ as the reduced word in $a,b$. If $cabc^{-1}$ is not a reduced word then one of $ca$ and $bc^{-1}$ is not a reduced word either. We deal with the first case, the other being similar. If $ca$ is not a reduced word then $c$ ends by the letter $a^{-1}$. Write  $c = c_{1} c_{2}$,  where $c_{2}$  a maximal ending subword of $c$ with the property that the letters of $c_{2}$ are alternatively $b^{-1}$ and $a^{-1}$. 
Then $cabc^{-1}$  equals either $c_{1} ab c_{1}^{-1}$ or $c_{1} ba c_{1}^{-1}$, according to the parity of the length of $c_{2}$.
\end{proof}

\begin{lemm}\label{lemma.free-group-conjugate}
Let $c$ be an element in the free group $G=<a,b>$, and assume that
\[
acbc^{-1} \mbox{ is conjugated to } ab.
\]
Then there exists integers $k,\ell$ such that $c=a^{k}b^{\ell}$.
\end{lemm}

\begin{proof}
We write $c$ as the reduced word in $a,b$. Up to eliminating the $b$'s at the end of $c$, which does not change the value of $cbc^{-1}$, we may assume that $cbc^{-1}$ is reduced. The problem now boils down to proving that $c$ is a power of $a$. We first examine the case when $acbc^{-1}$ is reduced. Using the hypothesis and the first lemma, we may write $acbc^{-1} = wabw^{-1}$ or $acbc^{-1} = wbaw^{-1}$. Note that in these equalities all the words are reduced. The word $acbc^{-1}$ has a '$b$' in the position just after the middle, thus the second equality is impossible.  Looking separately at the beginning and the end of the equality $acbc^{-1} = wabw^{-1}$, we get $ac=wa$ and $c=w$. Thus $c$ commutes with $a$; since the group is free, $c$ is a power of $a$, as wanted. In the case when the word $acbc^{-1}$ is not reduced, 
$c=a^{-1}c_{1}$, then $acbc^{-1}$ is equal to the reduced word $c_{1}bc_{1}^{-1}a$, and we conclude using the same method as in the first case.
\end{proof}

Now consider an automorphism $\Phi$ of the free group $G=<a_{1}, ..., a_{n} >$ such that for every $g\in G$, $\Phi(g) = c_{g} g c_{g}^{-1}$, and write $c_{i}$ for $c_{a_{i}}$. Up to composing $\Phi$ with the inner automorphism $g \mapsto c_{1}^{-1} g c_{1}$, we may assume that $c_{1} = 1$. Fix some $i>1$ and let  
$a=a_{1}$, $b=a_{i}$, $c=c_{i}$. Since $\Phi(ab) = \Phi(a) \phi(b) = acbc^{-1}$ is conjugate to $ab$, 
we are in the situation of Lemma~\ref{lemma.free-group-conjugate}.
We get $c_{i} = a_{1}^{k_{i}} a_{i}^{\ell_{i}}$. Up to composing each $c_{i}$ with $a_{i}^{-\ell_{i}}$ we may assume that $c_{i} = a_{1}^{k_{i}}$. It remains to check that $k_{i}$ does not depend on $i$. To see this we apply again lemma~\ref{lemma.free-group-conjugate} with $a=a_{i}$,  $b=a_{j}$ with $1<i<j$, and $c = a_{1}^{k_{j}-k_{i}}$. We get that $a_{1}^{k_{j}-k_{i}}$ equals some power of $a_{i}$, and thus $k_{j}=k_{i}$ since $G$ is free.
\end{proof}

\subsection{Uniqueness of homotopies}\label{ss.uniqueness}

Algebraic properties of fundamental groups have some important consequences on uniqueness of homotopies that we explain now.
We begin by the following easy but fundamental fact (which is already in \cite{gramain} and \cite{jaulent}).

\begin{fact}
\label{fact.commute}
In a  topological space $X$, let $\alpha$ be a loop based at $z$, and $H: \bbS^1 \times [0,1] \to X$ be a homotopy of loops such that $H(.,0) = H(.,1) = \alpha$. Let $\beta$ be the loop given by the trajectory of $z = \alpha(0)$ under $H$, namely $\beta(t) = H(0,t)$. Then in the fondamental group $\pi_{1}(X,z)$, $[\alpha]$ commutes with $[\beta]$.
\end{fact}

\begin{proof}
Since $H(.,0) = H(.,1)$, $H$ induces a continuous map from the quotient torus 
$$
T = \bbS^1 \times [0,1]\ \ / \ \  (s,0) \sim (s,1).
$$
Then the result follows from the fact that the fundamental group of the torus is abelian: indeed in $T$ the curves $s \mapsto (s,0)$ and $t \mapsto (0,t)$ induces two commuting elements of $\pi_{1}(T,(0,0)$; and their images under $H$ are nothing but the curves $\alpha$ and $\beta$.
\end{proof}

\begin{coro}\label{coro.uniqueness-homotopy}
Let $M$ be a surface that is not the projective plane, the torus, the Klein bottle, the annulus, the M\"obius band. Let $f$ be a homeomorphism of $M$ which is homotopic to the identity.
\begin{enumerate}
\item If $H,H'$ are two homotopies from the identity to $f$, then for every point $z$, the trajectories of $z$ under $H$ and $H'$ are homotopic.
\item There is  a unique lift of $f$ to the universal cover of $M$ which is obtained by lifting a homotopy from the identity to $f$. 
\item If $M$ is not the sphere, then the space of homotopies from the identity to $f$ is contractible. 
\end{enumerate}
\end{coro}

\begin{proof}
Let $H,H'$ be homotopies from the identity to $f$. Consider the homotopy  from the identity to the identity obtained by concatenating $H$ and the inverse of $H'$, and let $\beta$ be the loop given by the trajectory of $z$ under this homotopy. Proving Property 1 amounts to showing that $\beta$ is homotopically trivial. Fact~\ref{fact.commute} implies that $\beta$ commutes with every other loop $\alpha$. In other words $\beta$ is in the center of $\pi_{1}(M,z)$. But Proposition~\ref{prop.algebra-fundamental-group} states that the center of $\pi_{1}(M,z)$ is trivial. Hence, $\beta$ must be  homotopically trivial, and Property 1 is proved.  

To prove Property 2, we consider again two homotopies $H,H'$ from the identity to $f$. The homotopy $H$ (resp. $H'$) can be lifted to the universal cover $\tilde M$ of $M$ as a homotopy from the identity to some homeomorphism $\tilde f$ (resp. $\tilde f'$) of $ \tilde M$ which is a lift of $f$. We have to prove that the lifts $\tilde f$ and $\tilde f'$ coincide. Given some point $\tilde z$ in $\tilde M$, the point $\tilde f (\tilde z)$ is the endpoint of the lift of the trajectory under $H$ of the projection $z$ of $\tilde z$ in $M$. Likewise, $\tilde f '(\tilde z)$ is the endpoint of the lift of the trajectory of $z$ under $H'$. According to Property~1, these two trajectories are homotopic, and thus $\tilde f (\tilde z) = \tilde f' (\tilde z)$. Since $\tilde z$ is an arbitrary point, this shows that the lifts $\tilde f$ and $\tilde f'$ coincide, and Property 2 is proved.

Finally we prove Property~3. According to the Uniformization Theorem (see \emph{e.g.} \cite{reyssat}), and given the restrictions on its topology, the surface $M$ admits a hyperbolic structure, the universal cover $\tilde M$ identifies with the hyperbolic plane $\bbH^2$, and the automorphisms are isometries of $\bbH^2$. For every $x,y \in \bbH^2$, denote by $s \mapsto (1-s) x +  s y$ the arclength parametrization of the unique geodesic arc from $x$ to $y$. Let $\tilde f$ be the lift of $f$ which is given by Property 2. Then the formula $\tilde H_{0}(t,x) = t \tilde f(x) + (1-t) x$
provides a canonical homotopy from the identity to $\tilde f$ on $\bbH^2$, which induces a homotopy $H_{0}$ from the identity to $f$ on $M$. Given a homotopy $H$ from the identity to $f$, with lift $\tilde H$, the formula
$$
s \tilde H_{0} + (1-s) \tilde H
$$
provides a continuous deformation from $\tilde H$ to $\tilde H_{0}$ which induces a deformation from $H$ to $H_{0}$. This deformation depens continuously on $H$. This shows that the space of homotopies from the identity to $f$ is contractible. An alternative proof can be cooked up using the selections techniques presented in Section~\ref{sec.selection}.
\end{proof}

\bigskip

To end this section, we apply the previous results to show how Corollary~\ref{coro.maximal-isotopies} on maximal isotopies follows from Corollary~\ref{coro.maximal-unlinked-sets}.

\begin{proof}[Proof of Corollary~\ref{coro.maximal-isotopies}]
We consider a closed set $F$ of a surface $S$, a homeomorphism $f$ of $S$ that fixes every point of $F$, and an isotopy $I$ from the identity to $f$ relative to $F$. Let $M = S \setminus F$. First assume that $M$ is not 
one of the five exceptional cases of the first point of Corollary~\ref{coro.uniqueness-homotopy}.
Corollary~\ref{coro.maximal-unlinked-sets} provides a maximal unlinked closed set $F'$ containing $F$. In particular, there exists an isotopy $I'$ from the identity to $f$ relative to $F'$, and the couple  $(F', I')$ is a maximal element of $\cal S$. Corollary~\ref{coro.uniqueness-homotopy} says that the trajectories of a point $z \not \in F$ under both isotopies are homotopic in $M$, which means that $(F, I) \leq (F', I')$, as wanted.

Now assume $M$ is one of the five remaining cases.
If  $(F, I)$ is a maximal element of $\cal S$, then we are done. If not, we find a couple 
$(F_{1}, I_{1})$ with $F$ strictly included in $F_{1}$.
Now $M_{1} = S \setminus F_{1}$ does not belong to the excluded list anymore, and we may apply the first case to get a maximal couple $(F', I')$ such that $(F, I) \leq (F_1, I_{1}) \leq (F', I')$, as wanted.
\end{proof}


\section{Continuous selection of paths}\label{sec.selection}

In order to build a homotopy between a surface homeomorphism $f\in \homeo(S)$ and the identity,
we must choose continuously for each point $z$ in the surface a path $\gamma_z$ between $z$ and $f(z)$.
This is done in two steps: we first identify the homotopy class $C_z$ of paths in $S$ which will contain $\gamma_z$
and then select continuously the path $\gamma_z$ in the class $C_z$.
This section explains the second step. As an application we will prove the implications (W4) $\Rightarrow$ (W3) and (W3) $\Rightarrow$ (W2). These techniques will also be crucial in the next section.

One may obtain the family $(\gamma_z)_{z\in S}$ by using the general Michael's selection theorems~\cite{michael}. For a two-dimensional topological space $S$, the main assumptions for the existence of a continuous selection are:
the map $z\mapsto C_z$ is lower semi-continuous; for each $z$, the class $C_z$ is path-connected
and simply-connected; the map $z\mapsto C_z$ is
equi-locally path-connected and equi-locally simply connected.
We will not use Michael's theorem, we rather prove a selection result which is adapted to our simpler setting (Proposition~\ref{prop.selection}). The greatest simplification comes from a strong version of lower semi-continuity.

\subsection{An easy selection theorem}

In this section we consider two metric spaces $(S,d)$ and $(\cP,\delta)$. In the applications $S$ will be a surface (or more generally will contain a dense open set 
homeomorphic to a surface) and $\cP$ will be a space of paths on $S$. We use the notation $B_{d}(x,r)$ and $B_{\delta}(\gamma,\alpha)$ for the balls in $S$ and $\cP$.

We consider a family $(C_{z})_{z \in S}$ of subsets of $\cP$ indexed by points of $S$. We will state some topological condition that guarantees the existence of a \emph{continuous selection}, that is, a continuous map $z \mapsto \gamma_{z}$ from $S$ to $\cP$ such that for every $z$, $\gamma_{z} \in C_{z}$. In this section we will often denote applications as families, \emph{e. g.} the map $z \mapsto \gamma_{z}$ is denoted by $(\gamma_{z})_{z \in S}$.

\paragraph{Local triviality.}
  We will say that the family $z \mapsto C_{z}$  is \emph{locally trivial}
if for every $w \in S$, there exists a neighborhood $U_w$ of $w$
and a continuous map
$$\Theta \colon U_w\times C_{w}\to \cP$$
such that $\Theta(w, .) = \operatorname{Id}$ and
$\Theta(z,.)$ is a homeomorphism from $C_{w}$ to $C_z$
for each $z\in U_w$  (see Figure~\ref{f.sc-local-triviality-B}).  

\begin{figure}[ht]
\begin{center}
\def\svgwidth{0.45\textwidth}
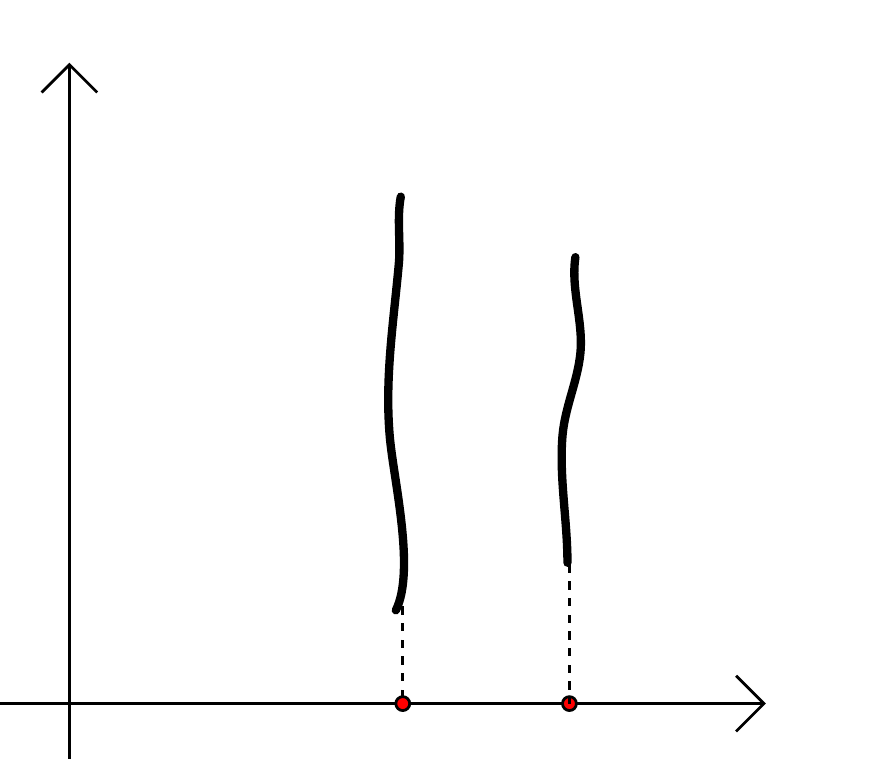
\def\svgwidth{0.45\textwidth}
\hspace{1cm}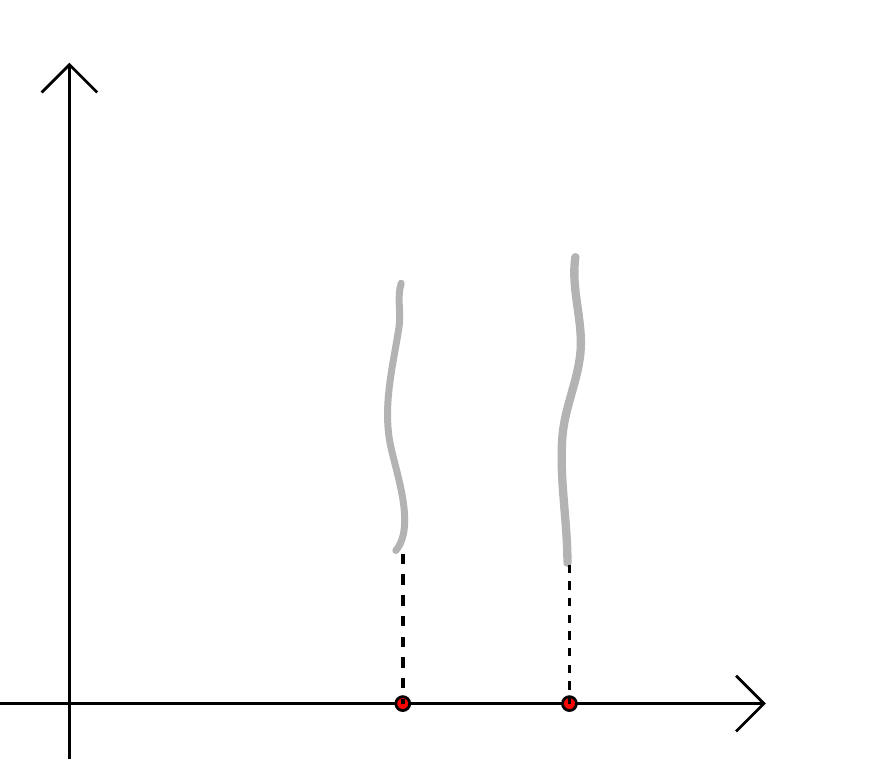
\caption{\label{f.sc-local-triviality-B}	Local triviality and extension of continuous families}
\end{center}
\end{figure}


Assuming this property and these notations, let $(z_{t})_{t\in \Delta_{1}}$ be a continuous family of points of $U_{w}$ indexed by the unit interval $\Delta_{1} = [0,1]$. Let $\gamma_{0} \in C_{z_{0}}$, $\gamma_{1} \in C_{z_{1}}$  and denote by $\gamma'_{0}, \gamma'_{1}$ the two elements of $C_{w}$ such that $\gamma_{i} = \Theta(z_{i},\gamma'_{i})$ for $i=0,1$. Let us assume furthermore that $C_{w}$ is arcwise connected. In this case there is a continuous family $(\gamma'_{t})_{t \in \Delta_{1}}$, included in $C_{w}$, that extends the family $\gamma'_{0}, \gamma'_{1}$.
Then the formula
$$
\gamma_{t} = \Theta(z_{t}, \gamma'_{t}), \ \ \ t \in \Delta_{1}
$$
produces a continuous family that extends $\gamma_{0}, \gamma_{1}$ and such that $\gamma_{t} \in C_{z_{t}}$ for each $t\in \Delta_{1}$. Likewise, let us assume that $C_{w}$ is simply connected. Let $(z_{t})_{t \in \Delta_{2}}$ be a continuous family of points of $U_{w}$ indexed by the standard $2$-simplex $\Delta_{2}$ (a triangle), and $(\gamma_{t})_{t \in \partial \Delta_{2}}$ a continuous family of paths, indexed by the boundary of $\Delta_{2}$, such that $\gamma_{t} \in C_{z_{t}}$ for every $t \in \partial \Delta_{2}$. Then as above this family extends to a continuous family $(\gamma_{t})_{t \in \Delta_{2}}$ such that $\gamma_{t} \in C_{z_{t}}$ for every $t \in \Delta_{2}$.

\paragraph{Local properties at singletons.}
Let $y_{0}$ be some point for which $C_{y_{0}}=\{\gamma_{y_{0}}\}$ is a singleton.
 We will say that the map $z \mapsto C_{z}$ is:
 \begin{itemize}
 \item[--] \emph{lower semi-continuous at $y_0$}
if for every neighborhood $\cU$ of $\gamma_{y_{0}}$ there is a neighborhood $V$ of $y_{0}$ such that, for every $z\in V$, the set $C_{z}$ meets $\cU$,
\item[--] \emph{equi-locally arcwise connected at $y_{0}$} 
if for every neighborhood $\cU$ of $\gamma_{y_{0}}$ there is a neighborhood $V$ of $y_{0}$ 
and a neighborhood $\cU'$ of $\gamma_{y_{0}}$
such that, for every $z \in V$ and every $\gamma,\gamma' \in C_{z} \cap \cU'$, there exists a path from $\gamma$ to $\gamma'$ in  $C_{z} \cap \cU$,
\item[--] \emph{equi-locally simply connected at $y_{0}$} 
if, for every neighborhood $\cU$ of $\gamma_{y_{0}}$, there is a neighborhood $V$ of $y_{0}$ 
and a neighborhood $\cU'$ of $\gamma_{y_{0}}$
such that, for every $z \in V$, every continuous family $(\gamma_{t})_{t \in \partial \Delta_{2}}$ included in $C_{z} \cap \cU'$, extends to a continuous family $(\gamma_{t})_{t \in \Delta_{2}}$ included in $C_{z} \cap \cU$.
\end{itemize}

\paragraph{Selection.}
Here is the selection result that we shall need.

\begin{prop}\label{prop.selection}
Let  $F$ be  a closed subset of $S$, and assume $M:=S \setminus F$ is a surface.\footnote{Not that here $M$ is not supposed to be connected. In section~\ref{ss.proper-homotopies}, $S$ will be the end-compactification of a connected surface $M$, and $F$ will be the set of ends of $M$; but in section~\ref{ss.strong-homotopies} $M$ will be the complement of a general closed set $F$ in a surface $S$.} Assume the following hypotheses on the family $(C_{z})_{z \in S}$:
\begin{enumerate}
\item Each $C_{z}$ is arcwise connected and simply connected.
\item On $S\setminus F$, the map $z\mapsto C_z$ is locally trivial.
\item\label{i.singleton} For every $y\in F$, the set $C_y$ is a singleton. 
\item The map $z \mapsto C_{z}$ is lower semi-continuous, equi-locally arcwise connected and equi-locally simply connected at every point of $F$.
\end{enumerate}
Then there exists a continuous selection $z \mapsto \gamma_{z}$.
\end{prop}

\begin{proof}
Since $S\setminus F$ is a surface, it admits a triangulation.
Any compact set of $S\setminus F$ meets only finitely many triangles.
By subdividing the triangulation if necessary,
we may assume that each closed triangle is included in one of the open set $U_{w}$ given by the local triviality (hypothesis 2).

\medskip

When $F=\emptyset$, the proof is straightforward: first select for any vertex $z$ of the triangulation any element $\gamma_{z}$ in $C_{z}$; then extend this selection continuously to the 1-skeleton using the local triviality; then extend the selection to the whole of $S$ using again the local triviality.

\medskip

 From now on we assume that $F$ is non empty. 
For every point $y$ in $F$, according to hypothesis~\ref{i.singleton}, $C_y$ is a singleton; we denote by $\gamma_y$ the unique element $C_{y}$.

\medskip

For each vertex $z$ of the triangulation, we choose a point $y_{z}$ in $F$ such that $d(z,y_{z}) \leq 2 d(z,F)$, and
an element $\gamma_{z}$ of $C_{z}$ not too far from $\gamma_{y_{z}}$ (remember that $y_{z_{0}}$ is a point in $F$), namely in such a way that
$$
(\star) \ \ 	\delta(\gamma_{z}, \gamma_{y_{z}}) \leq 2 \inf_{\gamma^*} \delta(\gamma^*,\gamma_{y_{z}})
$$
where the infimum is taken among all $\gamma^*$ in $C_{z}$.
\medskip

For each edge $\Delta_{1} = [z_{0} z_{1}]$ of the triangulation, whose endpoints are vertices $z_{0},z_{1}$,
we choose an open set $U_{w}$ given by the local triviality and containing $\Delta_1$.
Since $C_{w}$ is arcwise connected, the family $\{\gamma_{z_{0}}, \gamma_{z_{1}}\}$ extends to  a family $(\gamma_{z})_{z \in \Delta_{1}}$, such that $\gamma_{z}$ is in $C_{z}$ for each $z \in [z_{0}z_{1}]$ (as explained above after the definition of the local triviality). Again we choose this family not too far from $\gamma_{y_{z_{0}}}$, namely in such a way that 
$$
(\star \star) \ \ \sup_{z \in \Delta_{1}} \delta(\gamma_{z}, \gamma_{y_{z_{0}}}) \leq 2 \inf_{(\gamma^*_{z})} \sup_{z \in \Delta_{1}} \delta(\gamma^*_{z},\gamma_{y_{z_{0}}}) 
$$
where the infimum is taken among all continuous families  $(\gamma^*_{z})_{z\in \Delta_{1}}$ with $\gamma^*_{z} \in C_{z}$ for each $z$ in $\Delta_{1}$ and
$\gamma^*_{z_0}=\gamma_{{z_{0}}}$, $\gamma^*_{z_1}=\gamma_{{z_{1}}}$.

For each triangle $\Delta_{2}$ of the triangulation, whose vertices are $z_0,z_1,z_2$,
we choose an open set $U_w$ containing $\Delta_{2}$.
We can extend the family $(\gamma_{z})_{z \in \partial \Delta_{2}}$ to a family defined for every $z \in \Delta_2$. Again we choose the extension so that 
$$
(\star \star \star) \ \ \sup_{z \in \Delta_{2}} \delta(\gamma_{z}, \gamma_{y_{z_{0}}}) \leq 2 \inf_{(\gamma^*_{z})} \sup_{z \in \Delta_{2}} \delta(\gamma^*_{z},\gamma_{y_{z_{0}}}) 
$$
where the infimum is taken among all continuous families  $(\gamma^*_{z})_{ z \in \Delta_2}$ with $\gamma^*_{z} \in C_{z}$ for each $z \in \Delta_2$ and $\gamma^*_z=\gamma_z$
when $z\in \partial \Delta_2$.
To complete the proof it remains to show the following claim.

\begin{claim}
The map $z \mapsto \gamma_{z}$ satisfying properties $(\star), (\star \star), (\star \star \star)$ is continuous.
\end{claim}

 To prove the claim, we first note that by construction the map is continuous at every point of $S \setminus F$. Given  $y_{0} \in F$, let us prove the continuity at $y_{0}$. From the lower semi-continuity at $y_{0}$, we immediately get that the restriction of the map $z \mapsto \gamma_{z}$ to the set $F$ is continuous.

\medskip

Let $\varepsilon>0$. We first apply the equi-local simple connectedness to
the ball $B_\delta(\gamma_{y_{0}},\frac{\varepsilon}{10})$, and get some balls $B_d(y_{0},a_{1})$ and $B_\delta(\gamma_{y_{0}},\alpha_{1})$.
We then apply the equi-local arcwise connectedness to the ball $B_\delta(\gamma_{y_{0}},\frac{\alpha_{1}}{10})$
and get some balls $B_d(y_{0},a_{0})$ and $B_\delta(\gamma_{y_{0}},\alpha_{0})$.
The lower semi-continuity applied to the ball  $B_\delta(\gamma_{y_{0}},\frac{\alpha_{0}}{10})$
gives a ball $B_d(y_{0},b)$.
Finally, we apply the continuity of $y \mapsto \gamma_{y}$ on $F$ to get $c>0$
such that for every $y \in F \cap B_d(y_{0},c)$,  the distance $\delta(\gamma_{y},\gamma_{y_{0}})$ is less than the minimum $\beta$ of $\varepsilon/10, \alpha_{0}/10, \alpha_{1}/10$.

\medskip

Let $d$ be a positive real number, smaller than $a_{1}, a_{0}, b, c$. Let $\Delta_{2}$ be a $2$-simplex of the triangulation $\cT$ included in $B(y_{0},\frac{d}{3})$. To get the continuity we prove that for $z \in \Delta_{2}$, the distance $\delta(\gamma_{z}, \gamma_{y_{0}})$ is less than $\varepsilon$. Denote by $z_{0}, z_{1}, z_{2}$ the vertices and by $[z_{i}, z_{j}]$, the edges of the $2$-simplex $\Delta_{2}$.

\medskip

By lower semi-continuity,
there exists $\gamma'_{z_{i}} \in C_{z_{i}}$ at distance less than $\alpha_{0}/10$ from $\gamma_{y_{0}}$.
The point $y_{z_{0}}$ used in the definition of $\gamma_{z_{0}}$ satisfies
$d(z_{0}, y_{z_{0}}) < 2d(z_{0}, y_{0})$
and thus  $d(y_{z_{0}}, y_{0}) < 3d(z_{0}, y_{0}) < d$. In particular the continuity of $y\mapsto \gamma_y$
on $F$ gives $\delta(\gamma_{y_{z_{0}}}, \gamma_{y_{0}}) < \beta \leq \alpha_{0}/10$. The definition of $\gamma_{z_{0}}$ now implies
\begin{eqnarray*}
\delta(\gamma_{z_{0}}, \gamma_{y_{0}}) &\leq & \delta(\gamma_{y_{0}}, \gamma_{y_{z_{0}}}) +  \delta(\gamma_{z_{0}}, \gamma_{y_{z_{0}}}) \\
& < & \delta(\gamma_{y_{0}}, \gamma_{y_{z_{0}}})   + 2 \delta(\gamma'_{z_{0}}, \gamma_{y_{z_{0}}}) \quad\quad\quad\quad\quad\quad\quad \mbox{by $(\star)$}\\
& \leq & \delta(\gamma_{y_{0}}, \gamma_{y_{z_{0}}})   + 2 \left( \delta(\gamma'_{z_{0}}, \gamma_{y_0})+ \delta(\gamma_{y_0}, \gamma_{y_{z_{0}}})\right)\\
&< & 2 \delta(\gamma'_{z_{0}}, \gamma_{y_{0}}) + 3 \delta(\gamma_{y_{0}}, \gamma_{y_{z_{0}}})\\
& <&  5 \frac{\alpha_{0}}{10} = \frac{\alpha_{0}}{2}.
\end{eqnarray*}
Of course we get the same estimate for $z_{1}$ and $z_{2}$.

\medskip

Now we can use equi-local arcwise connectedness:
there exists $(\gamma'_{z})_{z \in [z_{0}z_{1}]}$, with $\gamma'_{z} \in C_{z}$ for every $z$ in $[z_{0} z_{1}]$, at distance less than $\alpha_{1}/10$ from $\gamma_{y_{0}}$.
The same type of inequalities as above (using $(\star\star)$ instead of $(\star)$) gives, for every $z \in [z_{0}, z_{1}]$,
$$
\delta(\gamma_{z}, \gamma_{y_{0}}) \leq \alpha_{1}/2.
$$
Of course we get the same estimate for every $z$  in $[z_1,z_2]$ and for $z$ in $[z_2,z_0]$.

\medskip

Finally use the equi-local simple connectedness:
there exists $(\gamma'_{z})_{z \in \Delta_{2}}$, with $\gamma_{z} \in C'_{z}$ for every $z$, at distance less than $\varepsilon/10$ from $\gamma_{y_{0}}$.
Again, the same type of inequalities as above (using this time $(\star\star\star)$ instead of $(\star)$) gives, for every $z \in \Delta_{2}$,
$$
\delta(\gamma_{z}, \gamma_{y_{0}}) \leq \varepsilon/2.
$$
This completes the proof of Proposition~\ref{prop.selection}.
\end{proof}

\subsection{Construction of homotopies}
\label{ss.constructions-homotopies}

With the selection tool we show that Property (W4) implies Property (W3)
(except in the case of the sphere, where this does not hold, and in the case of the projective plane, which will be a consequence of Proposition~\ref{p.projective-plane}).

\begin{prop}\label{p.w4Tow3}
\label{p.W4-W3}
Let $f$ be a homeomorphism of a connected surface $M$ without boundary,
which is not the sphere nor the projective plane.
If $f$ acts trivially on $\pi_{1}(M)$ then $f$ is homotopic to the identity.
\end{prop}

We begin by characterizing the triviality of the action on the fundamental group in terms of families of paths.

\begin{claim}\label{claim.pi-1-Cz}
Let $f$ be a homeomorphism of a connected surface $M$. Then $f$ acts trivially on $\pi_{1}(M)$ if and only if there exists a family $(C_{z})_{z \in M}$ of sets of paths with the following properties:
\begin{enumerate}
\item\label{i.homotopy-class} Each $C_{z}$ is a homotopy class of paths from $z$ to $f(z)$,
\item\label{i.continuity} (Continuity.) For every $w$, for every $z$ close enough to $w$, the family $C_{z}$ contains a path obtained by
concatenating a small path from $z$ to $w$, a path in $C_{w}$, and a small path from $f(w)$ to $f(z)$.
\end{enumerate}
\end{claim}

\begin{proof}
Assume that $f$ acts trivially on $\pi_1(M)$. Let $\tilde f:X\to X$ be a lift that commutes with the deck transformations. For every $z\in M$, we pick a lift $\tilde z\in X$ of $z$ and define $C_z$ to be the set of the projections in $M$ of the paths in $X$ going from $\tilde z$ to $\tilde f(\tilde z)$. This set $C_z$ is obviously a homotopy class of paths from $z$ to $f(z)$. Property~\ref{i.continuity} follows from the continuity of $\tilde f$.

Conversely, assume the existence of a family $(C_{z})_{z \in M}$ satisfying Properties~\ref{i.homotopy-class} and~\ref{i.continuity}. For every point $z\in M$, we pick a lift $\tilde z\in X$ of $z$, we consider the lift $F(z)$ of the point $f(z)$ so that every path in $C_z$ lifts to a path from $\tilde z$ to $F(z)$. Then the map $\tilde f:\tilde z\mapsto F(z)$ is a lift of $f$ which commutes with the covering automorphisms. Hence $f$ acts trivially on $\pi_1(M)$.
%
\end{proof}

\begin{proof}[Proof of Proposition~\ref{p.w4Tow3}]
To construct a homotopy, we want to apply Proposition~\ref{prop.selection} to the family $z \mapsto C_{z}$ provided by Claim~\ref{claim.pi-1-Cz}, with $S=M$ and $F=\emptyset$.
This will yield a continuous family of paths $z\mapsto \gamma_z$ connecting
$z$ to $f(z)$. Then the map $(t,z)\mapsto \gamma_z(t)$
will be a homotopy between the identity and $f$. Thus the proof will be complete with the following claim (note that items 3 and 4 of Proposition~\ref{prop.selection} are automatically satisfied when $F=\emptyset$).

\begin{claim}\label{c.fibration}
The map $z\mapsto C_z$ as in Claim~\ref{claim.pi-1-Cz} satisfies items 1 and 2 of Proposition~\ref{prop.selection}. Namely, every $C_{z}$ is arcwise connected and simply connected, and the map $z\mapsto C_z$ is locally trivial.
\end{claim}

Let us prove the claim. Since $M$ is not the sphere, nor the projective plane,
its universal cover $X$ is a plane. Let $\tilde z$ be a lift of some point $z$ of $M$, and $\tilde z'$ be the endpoint of the lift of any element of $C_{z}$ starting at $\tilde z$. The set $\tilde C_{z}$ of paths in $X$ joining $\tilde z$ to $\tilde z'$ is contractible. Furthermore, the covering map induces a homeomorphism between $\tilde C_{z}$ and $C_{z}$. Thus $C_{z}$ is arcwise connected and simply connected.

 To prove the local triviality, let us fix a point $w\in M$ (see Figure~\ref{f.sc-local-triviality-C}). One can find a continuous family of 
 homeomorphisms $\varphi_{z,0}$ of $M$, supported on a small disk $U_{w}$ around $w$,
defined for all $z$ in $U_{w}$, such that $\varphi_{w,0} = \operatorname{Id}$
 and such that for every such $z$, $\varphi_{z,0}(w)=z$ (see the proof of Lemma~\ref{l.fibration}).
 Since the space of compactly supported homeomorphisms of the disk is contractible (Alexander's trick), we may extend this family to a continuous family $\varphi_{z,t}$ supported in $U_{w}$, defined for $z$ in $U_{w}$ and $t$ in $[0,\frac{1}{2}]$, such that 
$\varphi_{z,\frac{1}{2}}=\operatorname{Id}$ for every such $z$, and $\varphi_{w,t}= \operatorname{Id}$ for every $t$.


Likewise, there is 
a continuous family of 
 homeomorphisms $\psi_{z,t}$ of $M$, supported on a small disk around $f(w)$,
 defined for all $t\in [\frac{1}{2},1]$ and all $z$ in $U_{w}$, 
 and such that for every such $z$,
$$
\psi_{z,\frac{1}{2}}=\operatorname{Id}, \ \ 
\psi_{z,1}(f(w))=f(z)
\text{ and } 
\psi_{w,t}=\id \text{ for every } t.
$$
 
We then consider the map $\Theta$ on $U_w\times C_{w}$
which associates to $(z,\gamma)$ the path $\gamma'$ between $z$ and $f(z)$ defined by
$$
\gamma'(t)= \left\{
\begin{array}{ll}
\varphi_{z,t}\circ\gamma(t) & \text{ for } t \in [0,\frac{1}{2}] \\
\psi_{z,t}\circ \gamma(t) & \text{ for } t \in [\frac{1}{2},1] 
\end{array}
\right.
$$
\begin{figure}[ht]
\begin{center}
\def\svgwidth{0.3\textwidth}
\def\ZZ{$\varphi_{z,\frac{1}{2}}=\psi_{z,\frac{1}{2}}=\operatorname{Id}$}
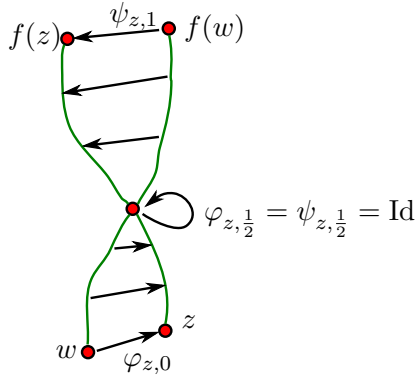
\caption{\label{f.sc-local-triviality-C}Proof of local triviality}
\end{center}
\end{figure}
This path joins $z$ to $f(z)$. By continuity of the family $z \mapsto C_{z}$, it belongs to $C_{z}$.
Since $\varphi_{w,t}=\psi_{w,t}=\id$ for every $t$, we have $\Theta(w,.)=\id$.
Furthermore the formulas defining $\Theta(z,.)$ can be reversed, thus $\Theta(z,.)$ is a homeomorphism between $C_{w}$ and $C_z$, as required by the local triviality.
This proves the claim, and completes the proof of Proposition~\ref{p.w4Tow3}.
\end{proof}

\subsection{Construction of proper homotopies}\label{ss.proper-homotopies}
We will now use the selection result to construct proper homotopies, proving that Property (W3) implies Property (W2).

\begin{prop}\label{p.homotopic-implies-proper}
Let $M$ be a connected surface, which is neither the plane nor the open annulus, and  $f$ be a homeomorphism of $M$ homotopic to the identity. Then there is a proper homotopy from the identity to $f$.

The same result holds:
\begin{itemize}
\item[--] if $M$ is the open annulus, and $f$ fixes each of the two ends of $M$,
\item[--] if $M$ is the plane, and $f$ preserves the orientation.
\end{itemize}
\end{prop}

\begin{rema}
\label{r.characterization-proper-homotopies}
Let $M$ be a surface and $f$ be a homeomorphism of $M$. Let $S$ be the end-compactification of $M$, and let $F = S \setminus M$ denote the set of ends (see section~\ref{ss.ends}). The homeomorphism $f$ extends to a homeomorphism of $S$ that we denote by $\hat f$. Moreover, the end compactification provides the following nice characterization of proper homotopies:  a homotopy $H$ from the identity to $f$ is proper if and only if it extends to a homotopy $\hat H$ of $S$ such that $\hat H(x,t) = x$ for every $x \in F$ and every $t \in [0,1]$. The proof is easy and left to the reader. In particular, a necessary condition for the existence of a proper homotopy from the identity to $f$ is that $f$ fixes each end of $M$ (\emph{i.e.} $\hat f(x)=x$ for every $x\in F$).  
\end{rema}

\begin{rema}
The homeomorphism $f$ of the open annulus $\mathbb{S}^1\times (-1,1)$ defined by $f(\theta,s)=(\theta,-s)$ is homotopic to the identity. To see this, one may consider for example the homotopy $(f_t)_{t\in [0,1]}$ defined by $f_t(\theta,s)=(\theta,(1-2t)s)$. Yet, there is no proper homotopic joining the identity to $f$, since $f$ exchanges the two ends of $\mathbb{S}^1\times (-1,1)$. Similarly, the homeomorphism $g$ of the plane $\mathbb{R}^2$ defined by $g(x,y)=(x,-y)$ is homotopic to the identity, but there is no proper homotopy joining the identity to $g$ since $g$ reverses the orientation. These two examples show that Proposition~\ref{p.homotopic-implies-proper} is optimal.
\end{rema}

\begin{proof}[Proof of Proposition~\ref{p.homotopic-implies-proper}]
Let $M$ be a connected surface and $f$ be a homeomorphism of $M$ which is homotopic to the identity. Assume that $M$ is not compact (otherwise the Proposition is trivial). In particular, $\pi_1(M)$ is a free group. We denote by $S$ the end-compactification of $M$, by $F = S \setminus M$ the set of ends of $M$, and by $\hat f$ the extension of $f$ to $S$. 

We exclude the two following cases: (1) $M$ is an annulus and $\hat f$ exchanges both ends of $M$, and (2) $M$ is a plane (the case when $M$ is a plane and $f$ preserves the orientation wil be treated separately at the end of the proof).
As explained in Remark~\ref{r.characterization-proper-homotopies}, we want to construct a homotopy of the end-compactification $S$, from the identity to $\hat f$, which fixes every end of $M$.

By hypothesis $f$ is homotopic to the identity, thus there is a family $(C_{z})_{z \in M}$ of paths in $M$ as in Claim~\ref{claim.pi-1-Cz}. Claim~\ref{c.fibration} applies: every $C_{z}$ is arcwise connected and simply connected, and the map $z \mapsto C_{z}$ is locally trivial.

 We extend the map $z \mapsto C_{z}$ to $S$ by defining $C_{z}$ to be the singleton containing the constant path $\{t \mapsto z\}$ when $z$ is an end of $M$.
We want to apply Proposition~\ref{prop.selection}. Points 1,2,3 are satisfied, it remains to check that the map $z \mapsto C_{z}$ is lower semi-continuous, equi-locally arcwise connected and equi-locally simply connected at every point $z$ of $F$.\footnote{Actually, in the case of a M\"obius band, the semi-continuity may fail; but in this case we will recover it by making another choice for the $C_{z}$'s, see below.}
 The following lemma is crucial (this is where we have to exclude the annulus case).

\begin{lemm}\label{lemm.end-1}
The homeomorphism $\hat f$ fixes each end of $M$. 
\end{lemm}

\begin{proof}
Let $x$ be an end of $M$. We assume by contradiction that $\hat f(x) \neq x$, and we will prove that $M$ is an open annulus (the first excluded case).
We first prove that $x$ is planar and isolated in $F$. Let $U$ be a small neighborhood of $x$, so that $\hat f(U)$ is disjoint from $U$.

Assume that $x$ is not planar (see Figure~\ref{f.properhomotopy}, left). So $U\cap M$ contains a surface $V$ with boundary, which is either a handle (homeomorphic to the torus minus an open disc) or a closed M\"obius band (homeomorphic to the projective plane minus an open disc). We can choose a simple loop $\gamma$ in $V$, and a simple loop $\alpha$ in $U$, so that the intersection number of $\gamma$ and $\alpha$ equals $1$.
The intersection number is invariant by homotopy, hence every curve freely homotopic to $\alpha$ must meet $\gamma$.  In particular, every curve freely homotopic to $\alpha$ must meet $U$.  But the curve $f(\alpha)$  is homotopic to $\alpha$ since $f$ is homotopic to the identity, and it is included in $\hat f(U)$ which was assumed to be disjoint from $U$. This is a contradiction.

\begin{figure}[ht]
\begin{center}
\def\svgwidth{0.45\textwidth}
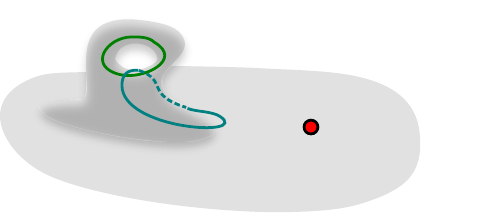
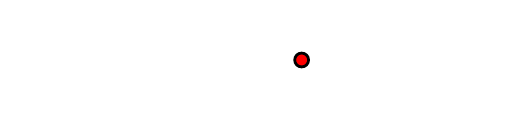
\caption{The ends are planar and isolated \label{f.properhomotopy}}
\end{center}
\end{figure}

Assume now that the end $x$ is planar but not isolated  (see Figure~\ref{f.properhomotopy}, right). Since $x$ is planar, we can find a neighbourhood $V$ of $x$, homeomorphic to a disc, included in $U$. Since $x$ is not isolated, we can find another end $x'$ of $M$ in $V$. We consider a simple closed curve $\alpha$ in $V$ which separates $x$ from $x'$: its intersection number with an arc from $x$ to $x'$ is $\pm 1$. The curve $f(\alpha)$ is freely homotopic to $\alpha$, hence it must also separate $x$ from $x'$. But $f(\alpha)$ is included in $\hat f(U)$ which is disjoint from $V$. This is again a contradiction.

Thus $x$ is isolated in $F$ and planar. So we can find a simple closed curve $\alpha$, surrounding a disk $D$ in $S$ which does not contain any end but $x$. If $\alpha$ is contractible in $M$, it bounds a disc (Lemma~\ref{lemm.bounding-disk}) and $M$ has only one end, contrary to the assumption that $x \neq \hat f(x)$. Assume $\alpha$ is sufficiently near $x$. Then $\alpha$ and  $f(\alpha)$ are two disjoint essential simple closed curves which are homotopic. Thus they bound a closed annulus (Lemma~\ref{l.bound-annulus}), and we see that $M$ is an open annulus with two ends $x,\hat  f(x)$.
\end{proof}

The next lemma is where we have to exclude the plane case. See section~\ref{ss.orientation} for the definitions.
\begin{lemm}\label{lemm.end-2}
The homeomorphism $\hat f$ locally preserves the orientation near every isolated planar end of $M$.
\end{lemm}

\begin{proof}
Let $x$ be an isolated planar end of $M$. According to the previous lemma, $x$ is fixed by $\hat f$. Assume, by contradiction, that $\hat f$ does not preserves the orientation near $x$. If  $\alpha$ is a simple closed curve surrounding $x$, then  $f(\alpha)$ surrounds $x$ with the opposite orientation, and is freely homotopic to $\alpha$ since $f$ is homotopic to the identity. If $\alpha$ is contractible in $M$ then it bounds a disc, $M$ is a plane and $f$ reverses the orientation: this is our second excluded case. If $\alpha$ is not contractible in $M$ then the corresponding element $[\alpha]$ in the fundamental group of $M$ is conjugate to its inverse. But $[\alpha]$ is a primitive element in a free group, or the square of a primitive element in case of a M\"obius band (see Proposition~\ref{prop.primitive}), and such an element is not conjugate to its inverse. This is the desired contradiction.
\end{proof}

\begin{lemm}\label{lemm.semi-continuity-proper}
The map $z \mapsto C_{z}$ is lower semi-continuous at every point of $F$ (up to making another choice for the $C_{z}$'s in the case of a M\"obius band).
\end{lemm}

\begin{proof}
Let $x \in F$. Let us prove the lower semi-continuity at $x$. We first examine the case when $x$ is isolated in $F$ and planar.
Let $U$ be a (disk) neighborhood of $x$ which contains no other point of $F$. According to Lemma~\ref{lemm.end-1}, $\hat f$ fixes $x$, so we can choose a smaller disk neighborhood $V$ of $x$ such that $\hat f(V)$ is included in $U$.  Let $z$ be a point of $V \setminus \{x\}$. We are looking for a curve belonging to $C_{z}$ and included in $U$.
Let $\alpha$ be a simple closed curve around $x$ in $V \setminus \{x\}$ based at $z$, $\beta$ be a curve from $z$ to $f(z)$ in $U$, and $\gamma$ be any curve in $C_{z}$. According to Lemma~\ref{lemm.end-2}, $\hat f$ preserves the orientation near $x$. Consequently, the curve $\alpha$ is freely homotopic to $f(\alpha)$ in $U \setminus \{x\}$. 
In the fundamental group $\pi_{1}(S \setminus F, z)$, we get
 $$
 [\alpha] = [\beta \star f(\alpha) \star \beta^{-1}].
 $$  
On the other hand, recall that there is homotopy from the identity to $f$, and that the set $C_z$ is the homotopy class of the trajectory of $z$ for this homotopy. As a consequence, we get 
$$
 [\alpha] = [\gamma \star f(\alpha) \star \gamma^{-1}].
$$
Thus we see that $[\beta^{-1}\gamma]$ commutes with $\alpha$. Now we examine two sub-cases. If $[\alpha]$ is a primitive element in the free group $\pi_{1}(M, x)$, then we apply Proposition~\ref{prop.algebra-fundamental-group2} and we conclude that there exists some integer $n$ such that 
$$
[\beta^{-1}\gamma] = [\alpha]^n.
$$
Thus $C_{z}$ contains the curve $\beta \alpha^{n}$ which is included in $U$. This proves the semi-continuity in this sub-case. If $[\alpha]$ is not a primitive element, then Proposition~\ref{prop.primitive} says that $\alpha$ bounds a compact M\"obius band. Then $M$ is an open M\"obius band, whose only end is $x$, and
whose fundamental group is cyclic and generated by the core curve $\alpha'$ with $[\alpha] = [\alpha']^2$. In this sub-case, either $[\beta^{-1}\gamma]$ is a power of $[\alpha]$ and we conclude as before, or $[\beta^{-1}\gamma \alpha']$ is a power of $\alpha$. There is an isotopy of the M\"obius band which ``makes a full turn'', that is, which goes from the identity to the identity and such that every trajectory is freely homotopic to the core $\alpha'$ (in the model where $M$ is the quotient of the plane by the map $(x,y) \mapsto (x+1, -y)$, take the homotopy from the identity to this map and project down to $M$). Composing the homotopy $H$ with this full turn isotopy, we get a new homotopy $H'$ from the identity to $f$, under which the trajectory of $z$ is homotopic to  $\gamma' = \gamma\alpha'$. For each $z$, define $C'_{z}$ to be the homotopy class of the trajectory of $z$ under $H'$.
Now $\beta^{-1}\gamma'$ is a power of $\alpha$, and we conclude as before that $C'_{z}$ contains a curve included in $U$. Thus the new family $(C'_{z})$ is lower semi-continuous.

Let us prove the lower semi-continuity in the case when $x$ is accumulated in $F$ or non planar.
Let $U$ be a neighborhood of $x$, and again consider a smaller neighborhood $V$, bounded by an essential simple closed curve disjoint from $F$, such that $V \cup \hat f(V) \subset U$.   Let $z$ be a point of $V$ distinct from $x$. As in the proof of Lemma~\ref{lemm.end-1}, there is a simple closed curve $\alpha$, included in $V$, which is not freely homotopic to any curve disjoint from $V$, and in particular it is essential and not homotopic to the boundary of $V$.
Let $V'=V \setminus F$.
We consider the universal cover $X$ of $M=S \setminus F$, which is a topological plane, a lift $\tilde V'$ of $V'$ in $X$, a lift $\tilde z$ of $z$ included in $\tilde V'$. Consider a lift $\tilde \alpha$ of $\alpha$ included in $\tilde V'$, and let $T$ be the corresponding (non trivial) automorphism of $X$.
Since $\alpha$ is included in $V'$ this automorphism leaves $\tilde V'$ invariant.
Let $\tilde f$ be the lift of $f$ obtained by lifting the homotopy from the identity to $f$.
Let us show that $\tilde f \tilde V'$ meets $\tilde V'$.
Assume by contradiction that $\tilde f \tilde V'$ is disjoint from $\tilde V'$. Since $\tilde f$ commutes with $T$, the set $\tilde f \tilde V'$ is also invariant under $T$. The set $\tilde V'$ is a submanifold, its boundary is a union of topological lines, since it is disjoint from $\tilde f \tilde V'$ one of these lines, say $L$, separates $\tilde V'$ from $\tilde f \tilde V'$. Since both $\tilde V'$ and $\tilde f \tilde V'$ are invariant under $T$  we get $TL = L$. This indicates that $\alpha$ is homotopic to the projection of $L$, which is the boundary of $V$ in $S \setminus F$, a contradiction.
Thus we have proved that $\tilde f \tilde V'$ meets $\tilde V'$. Since both sets are connected we may find a curve $\beta$ in their union which joins $\tilde z$ to $\tilde f (\tilde z)$. The projection of $\beta$ is a curve included in $V'\cup f(V')$ and homotopic to the trajectory of $z$ under the homotopy. In other words, it belongs to $C_{z}$. This proves the semi-continuity of $z \mapsto C_{z}$ at $x$.
\end{proof}

\begin{lemm}\label{lemm.ELC-ELSC}
 The map $z \mapsto C_{z}$ is equi-locally connected and equi-locally  simply connected at every point of $F$.
 \end{lemm}
 
 \begin{proof}
Let $x$ be a point of $F$, and $U$ be a neighborhood of $x$ in $S$. Since $F$ is totally disconnected and since we have assumed that $M=S\setminus F$ is not a plane, one can find a closed neighbourhood $V$ of $x$ in $S$, included in $U$, bounded by a simple closed curve disjoint from $F$, and such that $M\setminus V$ is not a disc. Note that $V':=V\setminus F$ is connected since $V$ is connected and $F$ is totally disconnected. Let $z \in V\setminus F$, and $\gamma,\gamma'$ be two elements of $C_{z}$ that are included in $V'$. Consider the universal cover $X$ of $M$, a lift $\tilde V'$ of $V'$ in $X$. Since $M\setminus V$ is not a disk, the map induced by the inclusion $V \setminus F \hookrightarrow M$ between the corresponding fundamental groups is injective (Lemma~\ref{lemm.groupe-fonda-sous-surface}). Thus $\tilde V'$ is a plane. 
Since $C_{z}$ is a homotopy class of paths from $z$ to $f(z)$, there exist some lifts $\tilde \gamma, \tilde \gamma'$ of $\gamma,\gamma'$ having the same endpoints. Since $\tilde V'$ is simply connected, those two lifts are homotopic in $\tilde V'$ relatively to the endpoints. This homotopy induces a homotopy from $\gamma$ to $\gamma'$, in $V \setminus F$, relatively to the endpoints. This proves the equi-local connectedness.

\bigskip

It remains to check that $z\mapsto C_{z}$ is equi-locally simply connected at $x$. We pursue with the notations of the preceding paragraph. Given a continuous family $(\gamma_{t})_{t \in \partial \Delta_{2}}$ included in $C_{z}$, we lift it to a continuous family of arcs $(\tilde \gamma_{t})_{t \in \partial \Delta_{2}}$ included in $\tilde V'$ and sharing  the same end-points. Since the homotopy group $\pi_2(\tilde V')$ is trivial, we may extend this family to a continuous family $(\tilde \gamma_{t})_{t \in \Delta_{2}}$ included in $\tilde V'$, which in turns provides a continuous family $(\gamma_{t})_{t \in \Delta_{2}}$ included in $V$ (hence in $U$). This proves the equi-local connectedness, and completes the proof of the lemma.
\end{proof}

The family $(C_z)_{z\in S}$ satisfies all the hypotheses of Proposition~\ref{prop.selection}. Therefore, we can apply this proposition. It provides a continuous selection $z\mapsto \gamma_z$ so that $\gamma_z\in C_z$ for every $z\in S$. This continuous selection defines a homotopy $\hat H$ from the identity to the homeomorphism $\hat f$. Moreover, for every $z\in F$, this homotopy $\hat H$ satisfies $\hat H(t,z)=z$, since $C_z$ is the singleton consisting of the constant path $\{t\mapsto z\}$. According to Remark~\ref{r.characterization-proper-homotopies}, the restriction $H$ of $\hat H$ to the surface $M$ is a proper homotopy. The proof of Proposition~\ref{p.homotopic-implies-proper} is complete, but for the case $M$ is a plane and $f$ preserves the orientation.

Consider finally a homeomorphism $f$ of the plane that preserves the orientation. Choose any point $x_{0}$, and let $g$ be a compactly supported homeomorphism of $M$ such that $g(x_{0}) = f(x_{0})$. The map $f' = g^{-1}f$ induces a homeomorphism of the annulus $M' = M \setminus \{x_{0}\}$ that preserves both ends. Let $\alpha$ be a loop around infinity outside the support of $g$.
Since $f$ preserves the orientation,  the loop $f'(\alpha) = f(\alpha)$ is homotopic to $\alpha$. Thus $f'$ acts trivially on the fundamental group of $M'$, and thus is homotopic to the identity in $M'$(according to Proposition~\ref{p.W4-W3}). According to the annulus case of Proposition~\ref{p.homotopic-implies-proper}, $f'$ is also properly homotopic to the identity. Thus $f$ is properly homotopic to $g$ in the plane, and by Alexander's trick $g$ is isotopic to the identity, so $f$ is properly homotopic to the identity.
\end{proof}

\subsection{Alternative proofs using hyperbolic geometry}
\label{sub.hyperbolic-geometry}

In this section we give alternative proofs for Propositions~\ref{p.w4Tow3} and~\ref{p.homotopic-implies-proper}, using surface geometrization. For sake of simplicity, we will only consider  the hyperbolic case, \emph{i.e.} the case where the surface $M$ has negative Euler characteristic\footnote{Similar techniques, using euclidean geometry instead of hypebolic geometry, can be used to cook up some proofs for Propositions~\ref{p.w4Tow3} and~\ref{p.homotopic-implies-proper} in the case where the surface $M$ has zero Euler characteristic.}. Note that the techniques presented below, which are based on hyperbolic geometry, cannot be used to construct strong homotopies. In other words, selection tools cannot be superseded in the proof of Proposition~\ref{prop.finitely-homotopic} in the next section. 

\begin{proof}[Alternative proof of Proposition~\ref{p.w4Tow3} and~\ref{p.homotopic-implies-proper} in the hyperbolic case]
We assume that $M$ has negative Euler characteristic. By the Uniformization Theorem, $M$ admits a riemannian metric of constant curvature $-1$ and finite volume. As a further consequence, $M$ is homeomorphic to the quotient of the hyperbolic plane $\mathbb{H}^2$ by a discrete group $\Gamma$ of hyperbolic isometries. Since $f$ acts trivially on $\pi_1(M)$, it admits a lift $\tilde f:\mathbb{H}^2\to\mathbb{H}^2$ which commutes with the elements of $\Gamma$. We define a homotopy $(F_t)_{t\in [0,1]}$ joining the identity to $\tilde f$ as follows. For every $w\in\mathbb{H}^2$, we consider the unique geodesic arc $\gamma_{w}:[0,1]\to \mathbb{H}^2$ joining $\gamma_w(0)=w$ to $\tilde \gamma_w(1)=\tilde f(w)$, and we set $F_t(w):=\gamma_{w}(t)$. This homotopy commutes with the elements of $\Gamma$ (because $\tilde f$ commutes with the elements of $\Gamma$, and because the elements of $\Gamma$ map geodesics on geodesics). Hence, it induces a homotopy $(f_t)_{t\in [0,1]}$ joining the identity to $f$ on the surface $M$. This completes the proof of Proposition~\ref{p.w4Tow3}.

Now we want to show that the homotopy $(f_t)_{t\in [0,1]}$ is proper. We consider the compactification of $\mathbb{H}^2$ by the circle at infinity. On the one hand, since the hyperbolic metric on $M$ has finite volume, the set of all the fixed points at infinity of the elements of $\Gamma$ is dense in the whole circle at infinity. On the other, since  $\tilde f$ commutes with the elements of $\Gamma$, it must preserves the fixed points at infinity of every element of $\Gamma$. Consequently, $\tilde f$ extends by the identity on the circle at infinity. Let us consider an end $x$ of $M$. It admits a basis of neighborhoods whose lifts in $\mathbb{H}^2$ have geodesically convex components (a set $E$ is geodesically convex if it contains all the geodesics whose endpoints are in $E$). More precisely,
two cases may be considered:
\begin{itemize}
\item[--] either $x$ is a cusp and is a limit of simple closed horocycles,
\item[--] or $x$ is a limit of simple closed geodesics.
\end{itemize}
Let $V\subset U$ be two such neighborhoods of $x$, such that $V$ is much smaller than $U$, and $\tilde V\subset \tilde U$ be some lifts of the neighbourhoods in $\mathbb{H}^2$. Since $V$ is much smaller than $U$, its image $\tilde f(\tilde V)$ must be contained in a lift $\tilde U'$ of $U$. Since $\tilde f$ acts as the identity at infinity, $\tilde U'$ must coincide with $\tilde U$. Hence $\tilde f(\tilde V)$ is included in $\tilde U$. Since $\tilde U$ is geodesically convex, it follows that the set $F_t(\tilde V)$ is included in $\tilde U$  for every $t\in [0,1]$ 
As a further consequence, the inclusion $f_t(V)$ is included in $U$ for every $t$. This shows that the homotopy $(f_t)_{t\in [0,1]}$ is proper.
\end{proof}


\section{Homotopies relative to finite sets}\label{sec.fini-homo-implique-homo}



In this section we prove that Property (S3) implies Property (S2).
As in the previous section we consider a boundaryless connected surface $S$ (not necessarily compact  nor orientable), and some closed subset $F$. 

\begin{prop}[The finite criterion]
\label{prop.finitely-homotopic}
Let $f \in \homeo(S,F)$. Assume there exists a dense subset $F_{0}$ of $F$ such that  $f$ is strongly homotopic to the identity relatively to any finite subset of $F_{0}$.
Then $f$ is strongly homotopic to the identity relatively to $F$.
\end{prop}

The proposition is obtained in three steps:
\begin{itemize}
\item[--] We first build a homotopy $(I_t)$ on $S \setminus F$ between the restriction $f_{|S\setminus F}$ and the identity.
\item[--] Then we prove that the homotopy classes of the trajectories $t\mapsto I_t(z)$
in $S\setminus F$ are lower semi-continuous at points of $F$.
\item[--] Finally, we use the selection tools of Section~\ref{sec.selection} to show that $(I_t)$ can be chosen to extend continuously
by the identity on the set $F$.
\end{itemize}

Proposition~\ref{prop.finitely-homotopic} is obvious when $F$ is finite. Thus in the remaining of this section we will assume that $F$ is infinite.
Under the hypotheses of Proposition~\ref{prop.finitely-homotopic}, there exists an increasing sequence $(F_{n})$ of finite subsets of $S$ whose union is dense in $F$, and for each 
$n$ there is a strong homotopy $I_{n}$ from the identity to $f$ relatively to $F_{n}$.
The trajectory of a point $z$ under this homotopy is denoted by $I_{n}.z$, and likewise the induced homotopy on a closed path $\gamma$ is denoted by $I_{n}.\gamma$.
We may, and we will, assume that each $F_n$ contains at least $3$ points for every $n$.
A first important remark is that by uniqueness of the homotopies in $S\setminus F_n$
(Corollary~\ref{coro.uniqueness-homotopy}),
the trajectories $I_{n}.z$  and $I_{p}.z$ are homotopic in $S\setminus F_n$ when $n < p$.
Here is another important remark for the case of non-orientable surfaces (see the appendix, section~\ref{ss.orientation} for the definition).

\begin{claim}\label{c.preserves-orientation}
Let $f$ be as in Proposition~\ref{prop.finitely-homotopic}. Then $f$ locally preserves the orientation near each point $x$ of $F$.
\end{claim}

\begin{proof} By density it suffices to check the case when $x$ belongs to $F_{n}$ for some $n$. Let $U$ be a disk neighborhood of $x$.
Consider an oriented curve $\alpha$ bounding a very small disk neighborhood of $x$. Using the strong homotopy relative to $F_{n}$, we see that the oriented curve $f(\alpha)$ is homotopic to $\alpha$ in $U \setminus x$. Thus $f$ preserves the orientation at $x$.
\end{proof}

\subsection{Homotopy in $S \setminus F$}\label{sec-finitely-homotopic}

\begin{prop}\label{p.class}
Under the assumptions of Proposition~\ref{prop.finitely-homotopic},
there exists a homotopy $I$ in $S \setminus F$ between the restriction $f_{|S\setminus F}$ and the identity.
\end{prop}

We will use several times the following construction. Let us call \emph{compact approximation of  $S \setminus F$} a subset $M$ of $S \setminus F$ which is a compact and connected surface with boundary such that every connected component $U$ of $S \setminus M$, satisfies at least one of the two following properties:
\begin{enumerate}
\item $U$ is not relatively compact (that is, the closure of $U$ is not compact),
\item or $U$ meets $F$.
\end{enumerate}

\begin{lemm}
Let $K$ be a compact connected subset of $S$ which is disjoint from $F$. Then there exists a compact approximation of $S \setminus F$ whose interior contains $K$.
\end{lemm}

\begin{proof}
We consider a compact connected surface $M'$ included in $S \setminus F$ containing $K$ in its interior. It has a finite number of complementary components. Then $M$ is obtained as the union of $M'$ with those of the relatively compact connected components of $S \setminus M'$ which are disjoint from $F$.
\end{proof}

Note that by density of $\cup F_{n}$ in $F$, if $M$ is a compact approximation of $S \setminus F$, for every large enough $n$, every relatively compact connected component of $S \setminus M$ meets $F_{n}$. A first application of this construction is given by the following lemma, which does not involve the map $f$.

\begin{lemm}\label{l.3}
Any loop $\gamma\subset S\setminus F$ which is contractible in $S\setminus F_n$
for each $n$ is also contractible in $S\setminus F$.
\end{lemm}

\begin{figure}[ht]
\def\svgwidth{0.7\textwidth}
\begin{center}
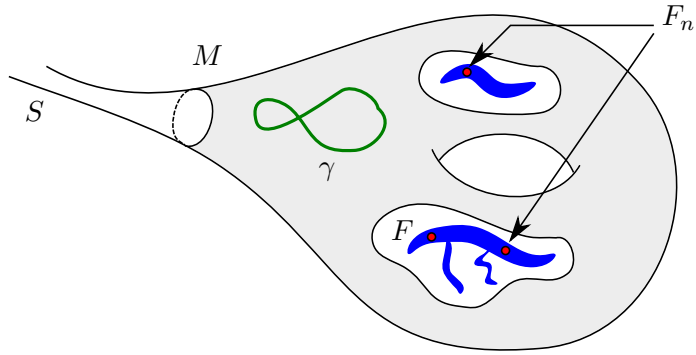
\caption{\label{f.compact-approx} Contractible loops in $S \setminus F$}
\end{center}
\end{figure}

\begin{proof}
Let $\gamma$ be as in the lemma, and $M$ be a compact approximation of $S \setminus F$ that contains $\gamma$. Let $n$ be such that every relatively compact connected component of $S \setminus M$ meets $F_{n}$ (see Figure~\ref{f.compact-approx}).

If there exists a connected component $\delta$ of the boundary of $M$
which is contractible in $S \setminus F_{n}$,
then $\delta$ bounds a topological disk $\Delta\subset S\setminus F_{n}$
(by Lemma~\ref{lemm.bounding-disk}).
The only possibility in this case is that $M = \Delta$. Since $\gamma$ is included in a disk, it is contractible.

Now we assume that no boundary component is contractible in $S \setminus F_{n}$.  We consider the universal cover $X$ of $S \setminus F_{n}$, which is homeomorphic to the plane. Let $\tilde M$ be a lift of $M$, \emph{i. e. } a connected component of the inverse image of $M$ in $X$. Since $\gamma$ is contractible in $S \setminus F_{n}$, it lifts to a closed path $\tilde \gamma$ in $\tilde M$. There is a topological disk $\Delta$ in $X$ that contains $\tilde \gamma$, and whose boundary is arbitrarily close to $\tilde \gamma$
(see Lemma~\ref{lemma.approximation-disc}); in particular we may assume that the boundary of $\Delta$ is disjoint from the boundary of $\tilde M$. On the other hand every boundary component of $M$, being non contractible in $S \setminus F_{n}$, lifts to a proper embedding of $\bbR$ into $X$; thus it is disjoint from $\Delta$. This implies that $\Delta$ is included in $\tilde M$. Since $\tilde \gamma$ is contractible in $\Delta$, the loop $\gamma$ is contractible in $M$.
\end{proof}

The next lemma says that, given a loop $\gamma$ outside $F$, the path followed by $\gamma$ under a homotopy $I_{n}$ can be continuously pushed off $F$ (see Figure~\ref{f.l.1}).

\begin{figure}[ht]
\def\svgwidth{\textwidth}
\begin{center}
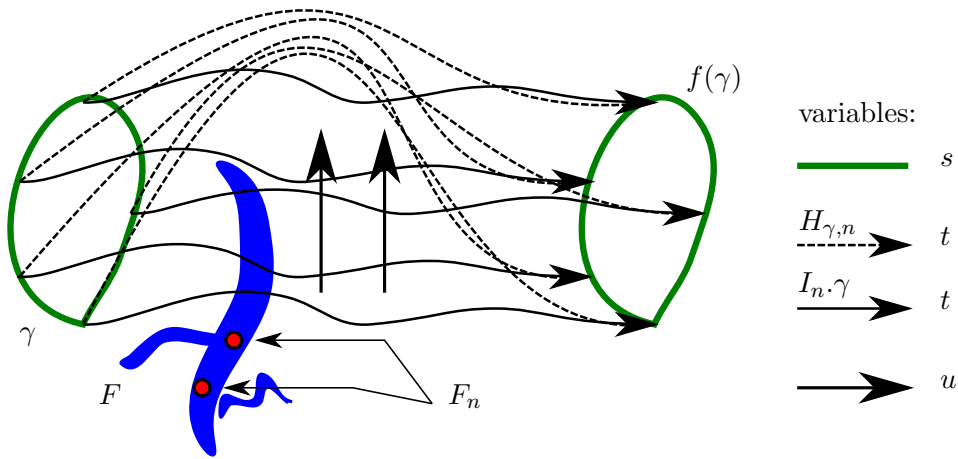
\caption{\label{f.l.1} Pushing the homotopy off $F$.}
\end{center}
\end{figure}

\begin{lemm}\label{l.1}
For every loop $\gamma\subset S\setminus F$ and every $n$,
there exists a homotopy $H_{\gamma,n}\colon \bbS^1 \times [0,1]\to S\setminus F$ between $\gamma$ and $f(\gamma)$
which is homotopic in $S\setminus F_n$ to the homotopy $I_n.\gamma$.
\end{lemm}

\begin{proof}
Let $n$ be an integer. We  denote by $\Pi: X \to S\setminus F_n$  the universal cover. 
Each homotopy $I_p$, $p\geq n$, lifts to $X$, hence defines a lift of $f$.
By compatibility of the homotopies on $S\setminus F_n$
(Corollary~\ref{coro.uniqueness-homotopy}), all these lifts coincide, defining a homeomorphism $\tilde f$
on $X$. 

Let $\tilde F = \Pi^{-1}(F \setminus F_{n})$.
Since, for $p>n$, $I_p$ is a strong homotopy relative to $F_p$,
the lift $\tilde F_p =  \Pi^{-1} (F_p\setminus F_n)$ to $X$ is fixed by $\tilde f$.
The union of the $\tilde F_{p}$'s for $p >n$ is dense in $\tilde F$, and thus
 $\tilde f$ fixes every point of $\tilde F$. Since $f$ preserves the orientation at every point of $F$ (Claim~\ref{c.preserves-orientation}), so does $\tilde f$ at every point of $\tilde F$, and thus $f$ preserves the orientation of the topological plane $X$.
Now Brown-Kister Theorem~\ref{t.brown-kister} applies and shows that
every connected component of $X\setminus \tilde F$ is invariant under $\tilde f$.

Let $\gamma$ be as in Lemma~\ref{l.1}. We will use the variable $s\in \bbS^1$ for the parametrization of $\gamma$, the variable $t\in [0,1]$ for the time of the homotopies $I_{n}$ between loops,  and $u$ for the deformation from $H_{\gamma,n}$ to $I_{n}.\gamma$ that we are about to construct (see Figure~\ref{f.l.1}).
Let us consider a lift $\tilde z$ of the point $z= \gamma(0)$ and the lift $\tilde \gamma$ of $\gamma$ starting at $\tilde z$. 
Since the connected component of $X \setminus \tilde F$ containing $\tilde z$ is invariant under $\tilde f$,
 there exists a curve joining $\tilde z$ to $\tilde f( \tilde z)$ in $X\setminus\tilde F$. Let $\alpha$ be the projection of this curve under $\Pi$. We consider a compact approximation $M$ of $S\setminus F$ containing the connected set $\gamma \cup \alpha \cup f(\gamma)$. We also consider integers $p>n$ large enough so that $F_{p}$ meets every relatively compact connected component of $S \setminus M$. The connected component $\tilde M$  of $\Pi^{-1}(M)$ that contains $\tilde z$ is a connected surface  containing $\tilde \gamma$ and $\tilde f(\tilde \gamma)$.

\begin{figure}[ht]
\def\svgwidth{\textwidth}
\begin{center}
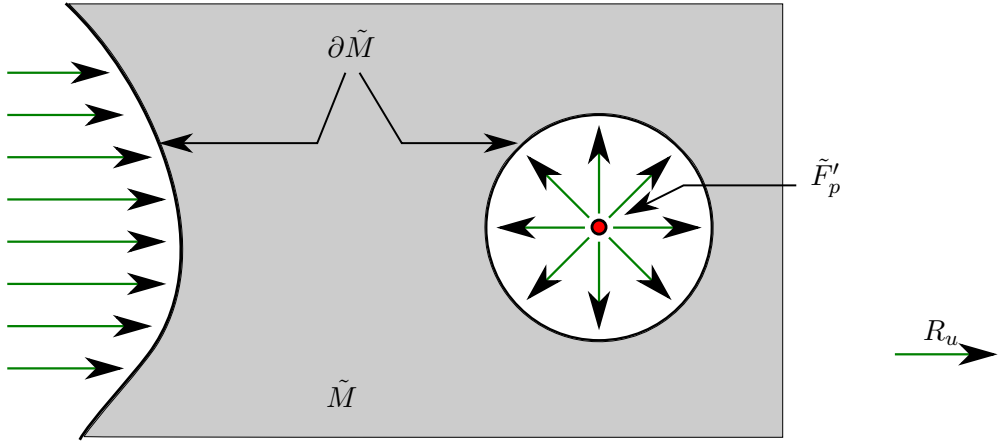
\caption{\label{f.pushing-lift} Pushing the homotopy off $F$...}
\end{center}
\end{figure}

Let us first consider the case where the loop $\gamma$ is contractible in $S \setminus F$. The lift $\tilde \gamma$ is a closed path in $X$.
Since $M$ is a compact surface, every connected component $\Delta$ of $X \setminus \inte(\tilde M)$ is of the following type: either $\Delta$ is unbounded, and then it is homeomorphic to the closed half-plane $\bbR \times [0,+\infty)$; or $\Delta$ is compact, and then it is homeomorphic to the closed unit disk (see Figure~\ref{f.pushing-lift}). In the second case $\Delta$ meets the lift $\tilde F_{p}$ of $F_{p}$; we select one point in $\tilde F_{p} \cap \Delta$ for every such $\Delta$ and denote the set of these points by $\tilde F'_{p}$. Note that for each (bounded or unbounded) such $\Delta$, the set $\Delta \setminus \tilde F'_{p}$ can be continuously deformed onto its boundary. Since the family of such $\Delta$'s is locally finite in $X$, we may put all these deformations together we get a \emph{deformation retract} $(R_{u})_{u \in [0,1]}$ from $X \setminus \tilde F'_{p}$ to $\tilde M$, \emph{i. e.} a continuous map $(z,u) \mapsto R_{u}(z)$ from $(X \setminus \tilde F'_{p}) \times [0,1]$ to $X \setminus \tilde  F'_{p}$  such that $R_{0}$ is the identity and $R_{1}$ is a \emph{retraction} from $X \setminus \tilde F'_{p}$ to $\tilde M$ (this means that $R_{1}$ maps $X \setminus \tilde F'_{p}$ to $\tilde M$ and its restriction to $\tilde M$ is the identity). We lift the strong homotopy $I_{p}$ to a homotopy $\tilde I_{p}$ in $X$. A crucial observation is that $\tilde I_{p}$ induces a homotopy in $X  \setminus \tilde F'_{p}$; in particular, the homotopy $\tilde I_{p}. \tilde \gamma$ does not meet $\tilde F'_{p}$.
Now we get the wanted homotopy by pushing  $\tilde I_{p}. \tilde \gamma$ by the deformation retract $(R_{u})$, and then projecting down to $S$. In other words the wanted homotopy is given by the formula
$$
H : (s,t,u) \mapsto \Pi \circ R_{u} \circ \tilde I_{p}(\tilde \gamma(s),t).
$$
Indeed for $u=0$ the formula boils down to $I_{p}.\gamma$. For $u=1$ we get a homotopy $H_{\gamma,n}$ between $\gamma$ and $f(\gamma)$ outside $F$. Moving $u$ from $0$ to $1$ we get a homotopy between  $I_{p}.\gamma$ and $H_{\gamma,n}$ outside $F_{n}$. Since the homotopies $I_{p}$ and $I_{n}$ are homotopic on $S \setminus F_{n}$ (Corollary~\ref{coro.uniqueness-homotopy}), the lemma is proved in this case.

Now let us consider the case where the loop $\gamma$ is not contractible in $S \setminus F$. Then $\gamma$ lifts to a non closed path $\tilde \gamma$ joining $\tilde z$ to the point $\tau(\tilde z)$ for some automorphism $\tau$ of the universal cover. We consider the intermediate covering
$$X \overset{\Pi_0}\longrightarrow X/\tau\overset{\Pi_1}\longrightarrow S\setminus F_n.$$
Since the action of any deck transformation is conjugate to a translation in the plane or to a translation composed with a symetry (in case $\gamma$ is a one-sided curve), $X/\tau$ is homeomorphic to the annulus $\bbS^{1} \times \bbR$ or to the open M\"obius band.
Let $\Delta$ be a connected component of  the complement of $\inte(\Pi_{0}(\tilde M))$ in $X/\tau$. Now $\Delta$ is a closed disk, a closed half-plane  or a closed half-annulus $\bbS^{1} \times [0,+\infty)$. Such an annulus retracts by deformation on its boundary. Thus, as before, we may construct a deformation retract $R_{u}$ from $(X/\tau) \setminus F''_{p}$ to $\tilde M$, where $F''_{p}$ is a subset of $\Pi_{0}(\tilde F_{p})$. Then we set
$$
H (s,t,u) = \Pi_{1} \circ R_{u} \circ \Pi_{0} \circ \tilde I_{p}(\tilde \gamma(s),t).
$$
Note that  $\tilde \gamma(1) = \tau \tilde \gamma(0)$,  $\tilde I_{p}(.,t)$ commutes with $\tau$ for every $t$, and $\Pi_{0} \tau = \Pi_{0}$. Thus $H(1,t,u) = H(0,t,u)$ for every $t,u$, and this formula defines a map from $\bbS^{1} \times [0,1]^{2}$ to $F \setminus F_{n}$. We conclude as in the first case.
\end{proof}

In the context of the previous lemma, the path $t \mapsto H_{\gamma,n}(0,t)$ joins the point $z = \gamma(0)$ to $f(z)$ in $S \setminus F$. We will denote it by $H_{\gamma,n}.z$  and call it the trajectory of $z$ under the homotopy $H_{\gamma,n}$.
\begin{lemm}\label{l.2}
For any loop $\gamma\subset S\setminus F$ based at $z$ and any $n\leq p$,
such that $\gamma$ is not contractible in $S\setminus F_n$,
the trajectories  $H_{\gamma,p}.z$ and $H_{\gamma,n}.z$ are homotopic in $S\setminus F_p$, and thus also homotopic to the trajectory $I_{p}.z$.
\end{lemm}
\begin{proof}
Let $\alpha$ denote the trajectory  $H_{\gamma,n}.z$ concatenated with the backward trajectory  $H_{\gamma,p}^{-1}.z$.
Since both $H_{\gamma,n}$ and $H_{\gamma,p}$ are homotopies from $\gamma$ to $f(\gamma)$ disjoint from $F$,  and thus from $F_{p}$, we may apply Fact~\ref{fact.commute} in the space $X =S \setminus F_{p}$, and we get that $[\alpha]$ and $[\gamma]$ commute in the group
 $\pi_1(S\setminus F_p, z)$. Since this group is free and $\gamma$ is non-trivial, this proves that $[\alpha]^k=[\gamma]^\ell$ for some $\ell\in \bbZ
$ and $k\in \bbZ \setminus \{0\}$ (Proposition~\ref{prop.algebra-fundamental-group}). The inclusion $S \setminus F_{p} \subset S \setminus F_{n}$ induces a morphism between fundamental groups, thus the same equality holds in $\pi_1(S\setminus F_n, z)$.
But in $S\setminus F_n$, the trajectory $H_{\gamma,n}.z$ is homotopic to $I_{n}.z$,
the trajectory $H_{\gamma,p}.z$ is homotopic to $I_{p}.z$, and  $I_{n}.z$ is homotopic to $I_{p}.z$.
Thus  in $\pi_1(S\setminus F_n, z)$ we have $[\alpha]=0$, hence also $[\gamma]^\ell=0$.
Since $\gamma$ is non-contractible in $S\setminus F_n$, this yields $\ell=0$. Going back to $\pi_1(S\setminus F_p, z)$ we get  $[\alpha]^k=0$.
Since $k\neq 0$ and there is no torsion element in $\pi_1(S\setminus F_p, z)$, it follows that $[\alpha]=0$  as wanted.
\end{proof}

\begin{lemm}\label{l.uniqueness-Cz}
For every $z \in S \setminus F$, there exists a unique homotopy class $C_{z}$ of paths from $z$ to $f(z)$ in $S\setminus F$ with the following property: for every $n$ every path in $C_{z}$ is homotopic in $S\setminus F_{n}$ to the trajectory $I_{n}.z$.
\end{lemm}

\begin{proof}
By Claim~\ref{c.preserves-orientation}, $f$ locally preserves the orientation near every point of $F$.
If the component of $S\setminus F$ which contains $z$ is a disc, then this disk is invariant under $f$ according to the local version of the Brown-Kister Theorem (Proposition~\ref{p.Brown-Kister}). Then we (obviously)
define $C_z$ as the unique homotopy class of paths from $z$ to $f(z)$ in $S\setminus F$.
From Lemma~\ref{l.1} applied to the constant loop, these paths are homotopic to $I_n.z$ in $S\setminus F_n$ for each $n$:
indeed $C_z$ contains the path $H_{\gamma,n}(z)$ which is homotopic to $I_{n}.z$.

If the component of $S\setminus F$ which contains $z$ is not a disc, there exists a loop $\gamma$ based at $z$ which is not
contractible in $S\setminus F$, hence not contractible in $S\setminus F_n$ for some $n$ (Lemma~\ref{l.3}).
Let us apply Lemma~\ref{l.1} to this loop, and define $C_z$ as the homotopy class of the trajectory $H_{\gamma,n}.z$.
By Lemma~\ref{l.2}, the paths in $C_{z}$ are homotopic in $S\setminus F_p$ to the trajectory $I_p.z$  for every $p\geq n$.
It is also the case for $p\leq n$ by compatibility of the homotopies.

Finally we prove the uniqueness. Let us consider two homotopy classes $C_z$ and $C'_z$ satisfying the conclusion of the lemma, and two paths $\alpha,\alpha'$
in each of them. By assumption, the loop $\alpha^{-1}\times \alpha'$ is contractible
in each $S\setminus F_n$, hence in $S\setminus F$ by Lemma~\ref{l.3}.
This shows that $C_z=C'_z$.
\end{proof}


\bigskip

\begin{proof}[End of the proof of Proposition~\ref{p.class}]
We want to construct a homotopy from the identity to $f$. The construction is performed independently on each connected component $M_{0}$ of $S \setminus F$. Let us check that the map $z \mapsto C_{z}$ provided by Lemma~\ref{l.uniqueness-Cz}, and restricted to $M_{0}$, satisfies the continuity property required by Claim~\ref{claim.pi-1-Cz}. Let $w$ be a point of $S \setminus F$, and $z$ be another point of $S \setminus F$ close to $w$. Let $\gamma$ be the concatenation of a small path from $z$ to $w$, a path in $C_{w}$ and a small path from $f(w)$ to $f(z)$. We have to check that $\gamma$ satisfies the characterization of $C_{z}$ given by Lemma~\ref{l.uniqueness-Cz}. Let us fix an integer $n$. Every path in $C_{w}$ is homotopic in $S \setminus F_{n}$ to the trajectory $I_{n}.w$; thus $\gamma$ is homotopic in  $S \setminus F_{n}$ to the concatenation of a small path from $z$ to $w$,  the trajectory $I_{n}.w$  and a small path from $f(w)$ to $f(z)$. By continuity of $I_{n}$ this path is homotopic to the trajectory $I_{n}.z$. This proves the wanted continuity property.

Now according to Claim~\ref{claim.pi-1-Cz}, the restriction of $f$ to $M_{0}$ acts trivially on the fundamental group of $S \setminus F$. Then Proposition~\ref{p.W4-W3} applies and we get that the restriction is homotopic to the identity on $M_{0}$, as wanted.
\end{proof}

\subsection{Lower semi-continuity at $F$}

Let us consider the homotopy $I$ obtained in Proposition~\ref{p.class}. As before, we associate to each point $z\in S\setminus F$ the homotopy class $C_z$ of paths between $z$ and $f(z)$ in $S\setminus F$ which contains the trajectory $I.z$. We extend this definition by setting $C_z:=\{z\}$ for $z\in F$ (more formally, we define $C_z$ to be the singleton whose unique element is the constant path which remains at $z$).

\begin{lemm}\label{l.cz-continuite}
The map $z \mapsto C_{z}$ is lower semi-continuous at every point $x$ of $F$.
\end{lemm}

\begin{proof}[Proof of Lemma~\ref{l.cz-continuite}]
We consider a point $x\in F$ and a neighbourhood $U$ of $x$ in $S$. We have to prove that, for $z\in S\setminus F$ close enough to $x$, the homotopy class $C_z$ contains a path which is contained in $U$. 

If $x$ lies in the interior of $F$, this follows immediately from the definition of $C_z$ for $z\in F$. If $x$ is isolated in $F$, then we can use the argument of the first part of the proof of Lemma~\ref{lemm.semi-continuity-proper}.  

Now we assume that $x$ is not isolated in $F$ and not in the interior of $F$. Let $V$ be a closed disc containing $x$ such that $V \cup f(V) \subset U$. Since $x$ is not isolated in $F$, up to reducing $V$, we may assume that the boundary of $V$ meets $F$.  Let  $W$ be an open disc containing $x$ such that $W \cup f(W) \subset V$. Let $z\in W\setminus F$ (such a point $z$ does exist, since $z$ is not in the interior of $F$). Let $V'$ be  the connected component  of $V\setminus F$ containing  $z$. We distinguish two cases.

The first case is when $V'$ is not contractible: there exists some simple closed curve $\alpha$ in $V'$ which bounds a disc in $V$ which meets $F$.  We argue as in the case of an accumulated or non planar end in the proof of Lemma~\ref{lemm.semi-continuity-proper}. Let $M'$ be the connected component of $S\setminus F$ containing $V'$. Note that $f$ preserves $M'$, since $f_{S\setminus F}$ is homotopic to the identity. Let $X'$ be the universal cover of $M'$ which is a topological plane. We consider a lift $\tilde V'$ of $V'$, a lift $\tilde z$ of $z$ in $\tilde V'$, a lift $\tilde\alpha$ of $\alpha$ included in $\tilde V'$, and the lift $\tilde f$ of $f$ obtained by lifting the homotopy $I$. We denote by $T$ the covering automorphism associated to $\tilde\alpha$. Assume by contradiction that $\tilde f(\tilde V')$ is disjoint from $\tilde V'$. The boundaries of $\tilde V'$  is a disjoint union of proper topological lines. One of these topological lines, say $L$, separates $\tilde V'$ from $\tilde f(\tilde V')$. By construction, $T$ preserves $\tilde V'$. Since $\tilde f$ commutes with the covering automorphisms, $T$ also preserves $\tilde f(\tilde V')$. Hence $T$ must preserve $L$. As a further consequence, the projection of $L$ in $S$ must be a closed curve, contained in the boundary of $V'$ and disjoint from $F$. This contradicts the fact that the boundary of $V$ is a simple closed curve which meets $F$. Hence $\tilde f(\tilde V')$ must meet $\tilde V'$, and we may find a path $\beta$ in $\tilde f(\tilde V')\cup\tilde V'$ joining $\tilde z$ to $\tilde f(\tilde z)$. The projection of the path is a curve joining $z$ to $f(z)$, homotopic to the trajectory of $z$ for the homotopy $I$, hence an element of $C_z$. Moreover this curve is included in $V'\cup f(V')$, hence in $U$. So we have proved that $C_z$ contains a path which is contained in $U$.

The second case is when $V'$ is contractible (see Figure~\ref{f.preuve-semi-continuite}).
Let us fix some integer $p$. Since $x$ is not isolated in $F$,
there exists a path $\alpha\subset W$ which connects $z$ to a
point $x'\in F \setminus F_{p}$ and is disjoint from $F\setminus \{x'\}$.
By Corollary~\ref{coro.uniqueness-homotopy},
for any $n\geq p$, and any $y\in F_n$, the loop $I_p.y$
is homotopic in $S \setminus F_{p}$ to $I_n.y$, which is constant.
The points in $\bigcup_n F_n$ are thus contractible fixed points for the homotopy $I_p$.
Since the set of contractible fixed points is closed, 
and $\bigcup_n F_n$ is supposed to be dense in $F$,
the trajectory $I_p.x'$ also is contractible in $S\setminus F_p$.

\begin{figure}[ht]
\begin{center}
\def\svgwidth{0.8\textwidth}
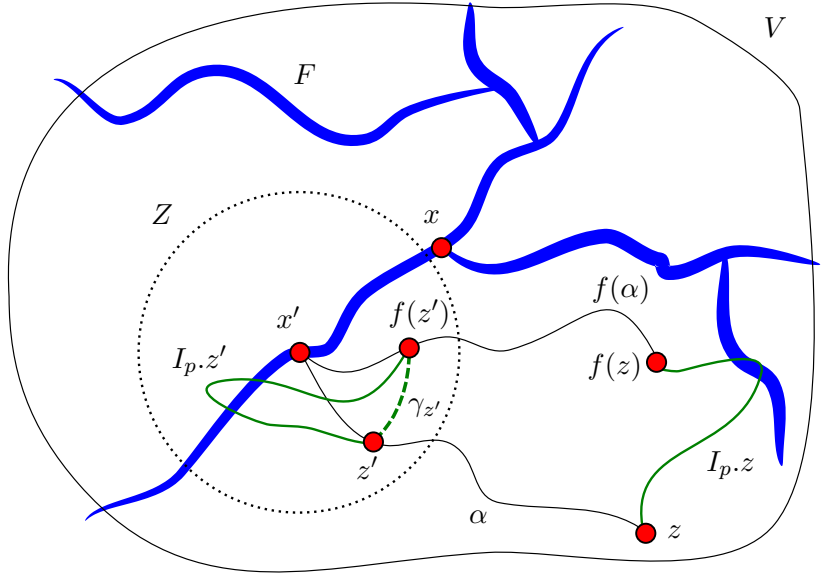
\end{center}
\caption{\label{f.preuve-semi-continuite}Proof of lower semi-continuity}
\end{figure}

Let $Z$ be a small neighborhood of $x'$ homeomorphic to a disc and disjoint from $F_p$.
By continuity, for any point $z'\in \alpha$ close enough to $x'$, the trajectory $I_p.z'$
is homotopic in $S\setminus F_p$ to a path contained in $Z$.
Remember that $f$ locally preserves the orientation at every point of $F$ (Claim~\ref{c.preserves-orientation}).
Consequently by the local version of the Brown-Kister theorem (Proposition~\ref{p.Brown-Kister} below),
 there exists a path $\gamma_{z'}\subset Z\setminus F$ which connects $z'$ to $f(z')$.
Since $Z$ is a disc and disjoint from $F_{p}$, the trajectory $I_p.z'$ is homotopic to $\gamma_{z'}$
in $S\setminus F_p$.

The homotopy $I_p$ in $S\setminus F_p$ allows to build a homotopy
in $S\setminus F_p$ between the trajectory $I_p.z$
and the concatenation $\alpha_{\mid [zz']}\ast I_{p}.z'\ast f(\alpha^{-1}_{\mid [zz']})$
where $\alpha_{\mid [zz']}$ is a path contained in $\alpha$ which connects $z$ to $z'$.
One deduces that $I_p.z$ is homotopic in $S\setminus F_p$ to the path
$\beta_p:=\alpha_{\mid [zz']}\ast\gamma_{z'} \ast f(\alpha^{-1}_{\mid [zz']})$
which is contained in $(W \cup f(W) ) \setminus F$, and thus in $V'$ (in particular we see that $f(z)$ belongs to $V'$).
Since $V'$ is contractible, the homotopy class of $\beta_{p}$ in $V'$ does not depend on $p$. Let $\beta$ be any curve in this homotopy class.
For every $p$, the trajectory $I_{p}.z$ is homotopic in $S \setminus F_{p}$ to our path $\beta$.
According to lemma~\ref{l.uniqueness-Cz}, this property characterizes
the fact that $\beta$ belongs to the homotopy class $C_{z}$. So we found an element of $C_z$ which is contained in $U$.
\end{proof}

\subsection{Construction of strong homotopies}
\label{ss.strong-homotopies}

\begin{proof}[Proof of Proposition~\ref{prop.finitely-homotopic}]
Assume the hypotheses of Proposition~\ref{prop.finitely-homotopic}. Apply Proposition~\ref{p.class} to get a homotopy $I$ between the restriction $f_{|S\setminus F}$ and the identity.  For every $z \in S \setminus F$, denote as before by $C_{z}$ the homotopy class containing the trajectory $I.z$. Extend this by letting $C_{z}=\{z\}$ whenever $z$ belongs to $F$. Claim~\ref{c.fibration} applies: every $C_{z}$ is arcwise connected and simply connected, and the map $z \mapsto C_{z}$ is locally trivial. According to Lemma~\ref{l.cz-continuite}, the map $z\mapsto C_z$ is also lower semi-continuous. Moreover, the argument in the proof of Lemma~\ref{lemm.ELC-ELSC} apply in our slightly different setting\footnote{With the following modifications. In our new setting, $F$ is not necessarily totally disconnected and $M=S\setminus F$ might be a plane. Yet, since $F$ is not a single point (we have assumed that $F$ is infinite), we can still find a closed neighbourhood $V$ of $x$ in $S$, included $U$, bounded by a simple closed curve, so that $F$ meets $S\setminus \mathrm{int}(V)$. This is enough to ensure that the map induced by the inclusion of $V\setminus F\hookrightarrow S\setminus F$ between the corresponding fundamental groups is injective. Also, in our new setting, $V\setminus F$ is not necessarily connected. To overcome this difficulty, it is enough to replace $V\setminus F$ by the connected component of $V\setminus F$ containing~$z$. }, and we get that the map $z \mapsto C_{z}$ is equi-locally connected and equi-locally simply connected at every point $x$ of $F$. Hence all the hypotheses of Proposition~\ref{prop.selection} are satisfied. We apply this proposition and get a continuous selection $z \mapsto \gamma_{z}$. This provides a strong homotopy from the identity to $f$ relative to $F$, as wanted.
\end{proof}

\newpage


\part{Isotopies}

\section{Straightening a topological line}
\label{s.isotopy-arc}

In the previous sections, we have only dealt with \emph{homotopies}. We will now start to construct \emph{isotopies}. In the present section, we will consider the simplest situation: we consider two properly embedded topological lines $\alpha,\alpha'$ which coincide outside a compact set, and we aim to build an isotopy ``bringing back" $\alpha'$ on $\alpha$. 

A \emph{properly embedded line} in a surface $M$ is a proper one-to-one map $\alpha:\bbR \to M$. Observe that the existence of a properly embedded line forces $M$ to be non compact. Two properly embedded topological lines $\alpha,\alpha'$ in a surface $M$ are said to be \emph{compactly homotopic} if there is a homotopy $H: \bbR \times [0,1] \to M$, joining $\alpha$ to $\alpha'$, with $H(s,t) = H(s,0)$ for every $s$ outside a compact interval of $\bbR$ and every $t\in [0,1]$; note that this implies that $\alpha$ and $\alpha'$ coincide outside a compact set. The main goal of the section is to prove the following result:

\begin{prop}
\label{p.straightening-lines-non-precise}
Let $\alpha,\alpha'$ be two compactly homotopic properly embedded lines in a surface $M$. Then there exists an isotopy $I$ of homeomorphisms of $M$, from the identity to a homeomorphism $f$ such that $f(\alpha') = \alpha$. 
\end{prop}

This result is essentially due to Epstein (\cite{epstein}). We will prove it for sake of completeness, and also because we will need a more precise statement that cannot be found in the literature (see Proposition~\ref{p.straightening-arc-easy} below).

In subsection~\ref{ss.filling}, we introduce the set $\fill(\alpha\cup\alpha')$. Roughly speaking, $\fill(\alpha\cup\alpha')$ is the part of $M$ ``lying between $\alpha$ and $\alpha'$".  All the isotopies that we will construct will be supported in a small neighbourhood of $\fill(\alpha\cup\alpha')$. In subsection~\ref{ss.transverse}, we prove a technical lemma which allows to put the lines $\alpha,\alpha'$ in \emph{quasi-transverse position}. In subsection~\ref{ss.bigons}, we will introduce the notion of \emph{bigon}, and prove the existence of minimal bigons bounded by subarcs of $\alpha$ and $\alpha'$. Finally, in subsection~\ref{ss.construction-isotopy-arc}, we state and prove a precise version of proposition~\ref{p.straightening-lines-non-precise}. The isotopy $I$ will obtained as a concatenation of elementary isotopies, each of which consists in ``removing a minimal bigon". 

Two compactly homotopic properly embedded lines certainly belong to the same connected component of $M$. Thus, without loss of generality, we will assume throughout this section that $M$ is connected.

\subsection{Filling of two properly embedded topological lines}
\label{ss.filling}


\begin{defi}
\label{d.filling}
Let $E$ be a subset of a connected surface $M$. The \emph{filling} of $E$ in $M$, denoted by $\fill(E)$ is the union of $E$ and all the connected components of $M\setminus E$ that are relatively compact, \emph{i. e.} with compact closure. The set $E$ is said to be \emph{filled} if it coincides with its filling, \emph{i.e.} if no connected component of $M\setminus E$ is relatively compact. 
\end{defi}

Observe that the notion of filling is monotonic: if $E\subset E'$ then $\fill(E)\subset\fill(E')$. Also note that the filling of a closed set is a closed set, and the filling of a compact set is a compact set. The three lemmas below state some other properties of fillings:

\begin{lemm}
\label{l.filled-neighbourhood}
If $E$ is a compact filled subset of a surface $M$, then every neighbourhood of $E$ contains a filled neighbourhood of $E$. 
\end{lemm}

\begin{proof}
Let $W$ be a neighbourhood of $E$. We want to find a filled neighbourhood of $E$ contained in $W$. First, since $E$ is compact, we can find a compact surface with boundary $W'$ so that $E\subset\mathrm{int}(W')\subset W$. We denote by $A_1,\dots,A_n$ the relatively compact connected components of $M\setminus W'$. For every $i$, since $E$ is filled, we can find a properly embedded topological half-line $\beta_i$ in $M$, connecting a point of $A_i$ to infinity, so that $\beta_i\cap E=\emptyset$. We consider the set $W'':=W'\setminus (\beta_1\cup\dots\cup\beta_n)$. This is a filled neighbourhood of $E$ contained in $W$.
\end{proof}

\begin{lemm}
\label{lemm.filling-disjoint}
Let $E,E'$ be disjoint subsets of $M$. Assume that  no connected component of $E'$ is relatively compact. Then $E'$ is disjoint from $\fill(E)$.
\end{lemm}

\begin{proof}
By contradiction, assume that $E'$ intersects $\fill(E)$. Since $E'$ is disjoint from $E$, it must intersect a relatively compact connected component $C$ of $M\setminus E$. Since the boundary of $C$ is included in $E$, it follows that $E'$ must be included in $C$. This is impossible since no connected component of $E'$ is relatively compact.
\end{proof}

\begin{lemm}
\label{lemm.fillings-disjoint}
Let $E,E'$ be disjoint subsets of $M$. Assume that no connected component of $E$ or $E'$ is relatively compact. Then the sets $\fill(E)$ and $\fill(E')$ are disjoint.
\end{lemm}

\begin{proof}
By contradiction, assume that $\fill(E)$ intersects $\fill(E')$. Since $E$ and $E'$ are disjoint, there must exist a relatively compact connected component $C$ of $M\setminus E$ and a relatively compact connected component $C'$ of $M\setminus E'$, such that $C$ intersects $C'$. Observe that the boundary of $C$ is disjoint from the boundary of $C'$ (because these boundaries are respectively included in $E$ and $E'$). Hence, up to exchanging $C$ and $C'$, the closure of $C$ must be included in $C'$. Since the boundary of $C$ is included in $E$, and since $E$ is disjoint from the boundary of $C'$, it follows that there must be a connected component of $E$ included in $C'$. But this is impossible since every connected connected component of $E$ is proper and $C$ is relatively compact.
\end{proof}

\subsection{Putting lines in quasi-transverse position}
\label{ss.transverse}

Recall that two injectively immersed topological lines $\alpha,\alpha'$ in a surface $M$ are called \emph{(topologically) transverse} if: 
\begin{enumerate}
\item the intersection of $E:=\alpha\cap\alpha'$ with any compact set of $M$ is finite,
\item for every $x\in E$, there is a neighbourhood $D_x$ of $x$ in $M$, and a homeomorphism $h$ from $D_x$ to the closed unit disc $\mathbb{D}^2$ so that $h(\alpha\cap D_x)$ and $h(\alpha'\cap D_x)$ are two different diameters of $\mathbb{D}^2$.
\end{enumerate}
As explained before, we want to work with pairs of properly embedded topological lines $\alpha,\alpha'$ that coincide outside a compact set. Such topological lines are never transverse in the above sense. This is the reason why we introduce a slightly twisted notion of transversality:


%
%

\begin{defi}
\label{d.quasi-transverse}
Let $M$ be a connected surface, and $\alpha,\alpha'$ be two properly embedded topological lines in $M$ which coincide outside a compact set. We say that $\alpha$ and $\alpha'$ are \emph{quasi-transverse} if:
\begin{enumerate}
\item $\alpha\cap\alpha'$ is the union of two properly embedded half-lines and a finite set $E$,
\item for every $x\in E$, there is a neighbourhood $D_x$ of $x$ in $M$, and a homeomorphism $h$ from $D_x$ to the closed unit disc $\mathbb{D}^2$ so that $h(\alpha\cap D_x)$ and $h(\alpha'\cap D_x)$ are two different diameters of $\mathbb{D}^2$.
\end{enumerate}
\end{defi}

\begin{prop}
\label{p.transverse}
Let $M$ be a connected surface, and $\alpha,\alpha'$ be two properly embedded topological lines in $M$ which coincide outside a compact set. There exists an isotopy $I$ of homeomorphisms of $M$, joining the identity to a homeomorphism $f$ so that $\alpha$ and $f(\alpha')$ are quasi-transverse. 

Moreover, given any neighborhood $U'$ of $\alpha'$, the isotopy $I$ can be chosen so that its support is compact and contained $U'$.
\end{prop}

\begin{figure}[ht]
\def\svgwidth{\textwidth}
\begin{center}
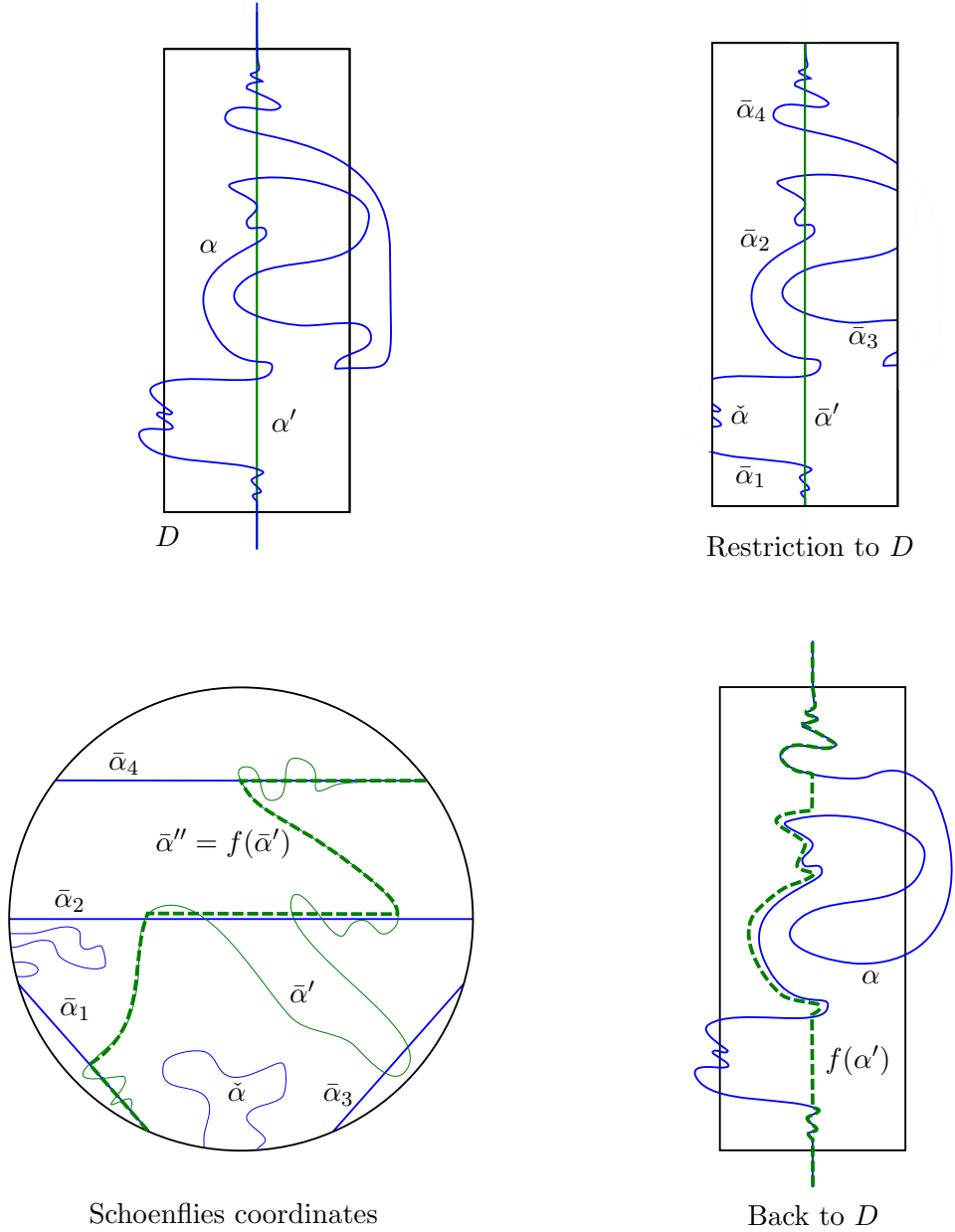
\caption{\label{f.quasi-transversification}Putting topological lines in quasi-transverse positions: proof of Proposition~\ref{p.transverse}}
\end{center}
\end{figure}

\begin{proof}
One can find a topological closed disc $D\subset U'$, such that $\bar\alpha':=\alpha'\cap D$ is a topological diameter\footnote{If $D$ is a closed topological disc, we will call \emph{topological diameter} of $D$ any arc whose endpoints lie on $\partial D$, and whose interior is contained in the interior of $D$.} of $D$, and such that $\alpha$ coincides with $\alpha'$ outside a compact subset of $\mathrm{int}(D)$ (see Figure~\ref{f.quasi-transversification}). Since $\alpha$ is a properly embedded line, and since $\alpha$ coincides with $\alpha'$ outside a compact subset of $\mathrm{int}(D)$, there are only finitely many connected components of $\alpha\cap\mathrm{int}(D)$ which meet $\bar\alpha'$. Denote by $\bar\alpha_1,\dots,\bar\alpha_n$ the closures of these connected components. These are topological diameters of $D$, with pariwise disjoint interior. Morever, $\check\alpha:=(\alpha\cap D)\setminus(\alpha_1\cup\dots\cup\alpha_n)$ is a compact set, whose distance to $\bar\alpha'$ is strictly positive. We can find a topological diameter $\bar\alpha''$ of $D$ with the same endpoints as $\bar\alpha'$, so that $\bar\alpha''$ is transverse to $\bar\alpha_1,\dots,\bar\alpha_n$ and disjoint from $\check\alpha$ (to see this, for instance one may first use repeatedly the Schoenflies theorem  to build a homeomorphism between $D$ and the standard unit disk, which sends each curve $\alpha_{i}$ onto a chord of the unit disk; see subsection~\ref{subsection.schoenflies}).
 Applying the Schoenflies theorem in the disc $D$, one gets a homeomorphism $f$ of $M$, supported in $D$, such that $f(\bar\alpha')=\bar\alpha''$. The topological lines $f(\alpha')$ and $\alpha$ are quasi-transverse. Furthermore Alexander trick (Proposition~\ref{p.alexander}) provides an isotopy from the identity to $f$, supported in $D$, as wanted.
\end{proof}

%
%

\subsection{Existence of bigons}
\label{ss.bigons}

\begin{defi}
\label{d.bigon}
A \emph{bigon} in a surface $M$ is a triple $(B,\bar\alpha,\bar\alpha')$ where $B$ is a topological closed disc in $M$, and where $\bar\alpha,\bar\alpha'$ are two arcs with the same endpoints, with disjoint interiors, such that $\bar\alpha\cup\bar\alpha'=\partial B$. The arcs $\bar\alpha,\bar\alpha'$ are the \emph{edges} of the bigon, the endpoints of $\bar\alpha,\bar\alpha'$ are the \emph{vertices} of the bigon. 

Let $\alpha,\alpha'$ are topological lines or arcs or simple closed curves in a surface $M$. A \emph{bigon for $\alpha,\alpha'$} is a bigon $(B,\bar\alpha,\bar\alpha')$ where $\bar\alpha$ is a subarc of $\alpha$, and $\bar\alpha'$ is a subarc of $\alpha'$. This bigon is called \emph{minimal} if the interior of $B$ is disjoint from $\alpha$ and $\alpha'$. 
\end{defi}

\begin{prop}
\label{p.exists-bigon}
Let $\alpha,\alpha'$ be two compactly homotopic quasi-transverse properly embedded topological lines in a surface~$M$. Then, either $\alpha=\alpha'$, or one can find a minimal bigon for $\alpha,\alpha'$ in $M$.
\end{prop}


In order to prove this proposition, we will work in the universal cover of $M$. The following lemma will be used several times. 

\begin{lemm}
\label{l.minimal-bigon}
Let $(B,\alpha,\alpha')$ be a bigon in a surface $X$, and $\cA$ be a non-empty finite collection of pairwise disjoint arcs, whose ends lie in the interior of $\alpha'$, and whose interiors lie in $\mathrm{int}(B)$. Then one can find a bigon $(\widehat B,\hat\alpha,\hat\alpha')$ such that $\widehat B\subset B$, such that $\hat\alpha$ is an element of  $\cA$, such that $\hat\alpha'$ is a sub-arc of $\alpha'$, and such that the interior $\widehat B$ is disjoint from the elements of $\cA$. 
\end{lemm}

\begin{proof} 
For each arc $\hat \alpha$ in $\cA$,, the set $ B \setminus \hat \alpha$ has two connected components (see  See figure~\ref{f.trouver-bigone}). One of these two connected components contains $\alpha$. Denote by $B(\hat \alpha)$ the closure of the other connected component. Since the elements of $\cA$ are pairwise disjoint, the set $\{B(\hat \alpha), \hat \alpha \in \cA \}$ is partially ordered by inclusion. Since it is a finite set, there exists an element $\hat \alpha$ of $\cA$ such that $B(\hat \alpha)$ is a minimal element for this partial order. There exists a subarc $\hat \alpha'$ of $\alpha'$ such that the triple $(B(\hat \alpha), \hat \alpha,\hat \alpha')$ is a bigon. This bigon suits our needs.
\end{proof}

\begin{figure}[ht]
\def\svgwidth{0.8\textwidth}
\def\hap{\hat \alpha'}
\begin{center}
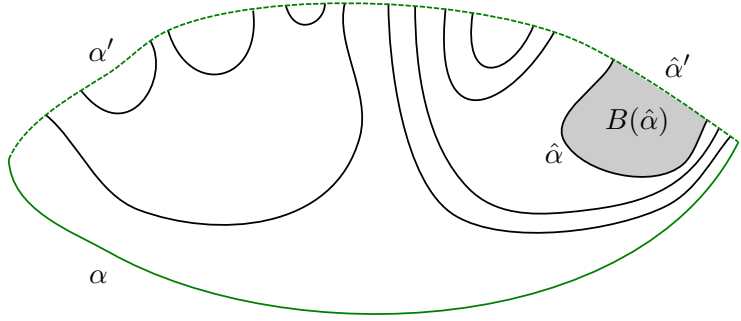
\caption{\label{f.trouver-bigone}The minimal bigon $(B(\hat \alpha), \hat \alpha,\hat \alpha')$ provided by the proof of Lemma~\ref{l.minimal-bigon}}
\end{center}
\end{figure}

\begin{proof}[Proof of Proposition~\ref{p.exists-bigon}]
Let $\pi:X\to M$ be a universal cover of the (non compact) surface $M$. We choose a lift $\tilde \alpha$ of $\alpha$ in $X$. Since $X$ is a topological disc and $\tilde\alpha$ is a properly embedded topological line in $X$, using the Schoenflies theorem, we can, and we will, see $X$ as the unit disc in $\mathbb{R}^2$, and $\tilde\alpha$ as the horizontal diameter of this disc. Lifting a compactly supported homotopy between $\alpha,\alpha'$, we get a lift $\tilde\alpha'$ of $\alpha'$ which is compactly homotopic and quasi-transverse to $\tilde\alpha$. Without loss of generality, we will assume that $\tilde\alpha'$ intersects the upper half disc.
In order to find the desired bigon, we will apply Lemma~\ref{l.minimal-bigon} three times. 

We first consider the bigon $(B_0,\alpha_0,\alpha_0')$ in $\mathbb{R}^2$, where $B_0$ is the closed upper-half unit disc, $\alpha_0$ is the upper-half unit circle, and $\alpha_0'$ is the horizontal diameter of the unit disc (\emph{i.e.} the closure of $\tilde\alpha$ in $\mathbb{R}^2$). We denote by $\cA_{0}$ the family of the closures of the connected components of $\tilde \alpha'\cap \mathrm{int}(B_0)$. Since $\tilde \alpha$ and $\tilde \alpha'$ coincide outside a compact set of $X$ and are quasi-transverse,  $\cA_0$ is a finite collection of arcs whose ends lie in the interior of $\alpha_0'$, and whose interiors lie in the interior of $B_0$. We apply Lemma~\ref{l.minimal-bigon} to the bigon $(B_0,\alpha_0,\alpha_0')$ and the collection $\cA_0$. It provides us with a bigon $(B_1,\alpha_1,\alpha_1')$, where $\alpha_1$ is a subarc of $\tilde\alpha'$, where $\alpha_1'$ is a subarc of $\tilde\alpha$, and where the interior of $B_1$ is disjoint from $\tilde\alpha$ and $\tilde\alpha'$. 
In other words, $(B_1,\alpha_1,\alpha_1')$ is a minimal bigon in $X$ for $\tilde\alpha,\tilde\alpha'$ in the sense of Definition~\ref{d.bigon}.

Now let $\cA_1$ be the collection of the closures of the connected components of $\mathrm{int}(B_{1}) \cap \pi^{-1}(\alpha')$.  Note that  $\cA_1$ is a collection of arcs whose ends lie on $\alpha_1'$, because $\pi^{-1}(\alpha')$ is a collection of pairwise disjoint properly embedded topological lines, which are topologically transverse to $\alpha_1'\subset\tilde\alpha$. Moreover, $\cA_1$ is finite,
because of the quasi-transversality between $\pi^{-1}(\alpha')$ and $\alpha_1'\subset\tilde\alpha$. Hence, we can apply Lemma~\ref{l.minimal-bigon} to the bigon $(B_{1},\alpha_{1},\alpha_{1}')$ and the collection $\cA_{1}$. It provides us with a bigon $(B_2,\alpha_2,\alpha_2')$, where $\alpha_2$ is a subarc of a lift of $\alpha'$, where $\alpha_2'$ is a subarc of $\tilde\alpha$, and where the interior of $B_2$ is disjoint from $\pi^{-1}(\alpha')$. Moreover the interior of $B_2$ is also disjoint from $\tilde \alpha$ since $B_2\subset B_1$.

Finally, we consider the collection $\cA_2$ of the connected components of  $B_{2} \cap \pi^{-1}(\alpha)$. By the same type of arguments as above, $\cA_2$ is a finite collection of arcs, whose ends lie in $\alpha_2$, and whose interiors are contained in the interior of $B_2$. Hence, we can apply Lemma~\ref{l.minimal-bigon} to the bigon $(B_{2} ,\alpha_{2},\alpha_{2}')$ and the collection of arcs $\cA_{2}$. It provides us with a bigon  $(B_3,\alpha_3,\alpha_3')$, where $\alpha_3$ is a subarc of a lift of $\alpha$, where $\alpha_3'$ is a subarc of $\alpha_2$ (hence a subarc of a lift of $\alpha'$), and where the interior of $B_3$ is disjoint from $\pi^{-1}(\alpha)$. The interior of $B_3$ is disjoint from $\pi^{-1}(\alpha')$ since $B_3\subset B_2$. 
   
It remains to check that the projection of $(B_3,\alpha_3,\alpha_3')$ in $M$ is a bigon satisfying the desired properties. Let $\theta$ be a deck transformation of the universal cover $\pi:X\to M$. Note that the situation $B_3\subset \theta(B_3)$ is excluded since this would imply the existence of a fixed point for $\theta$ by Brouwer fixed point theorem. On the other hand, by construction, the interior of $B_3$ is disjoint from  $\pi^{-1} (\alpha \cup \alpha')$. It follows that  the interior of $B_3$ is disjoint from the image under $\theta$ of the boundary of $B_3$. Hence, the only possibility is that the interior $B_3$ is disjoint from its image under $\theta$. Furthermore, each edge is disjoint from its image under $\theta$, since $\alpha$ and $\alpha'$ are embedded lines; and by transversality each edge is disjoint from the image of the other edge. Thus the closed disk  $B_{3}$ is disjoint from its image under $\theta$.
Since this is true for any deck transformation $\theta$, it follows that the projection $(\bar B_3,\bar\alpha_3,\bar\alpha_3')$ of $(B_3,\alpha_3,\alpha_3')$ in $M$ is one-to-one. Clearly, this projection is a bigon, for $\alpha,\alpha'$. Moreover, this bigon is minimal: indeed the interior of $\bar B_3$ is disjoint from $\alpha$ and $\alpha'$, since the interior of $B_3$ is disjoint from $\pi^{-1}(\alpha)$ and also from $\pi^{-1} (\alpha)$. 
\end{proof}

\begin{figure}[ht]
\begin{center}
\def\svgwidth{0.8\textwidth}
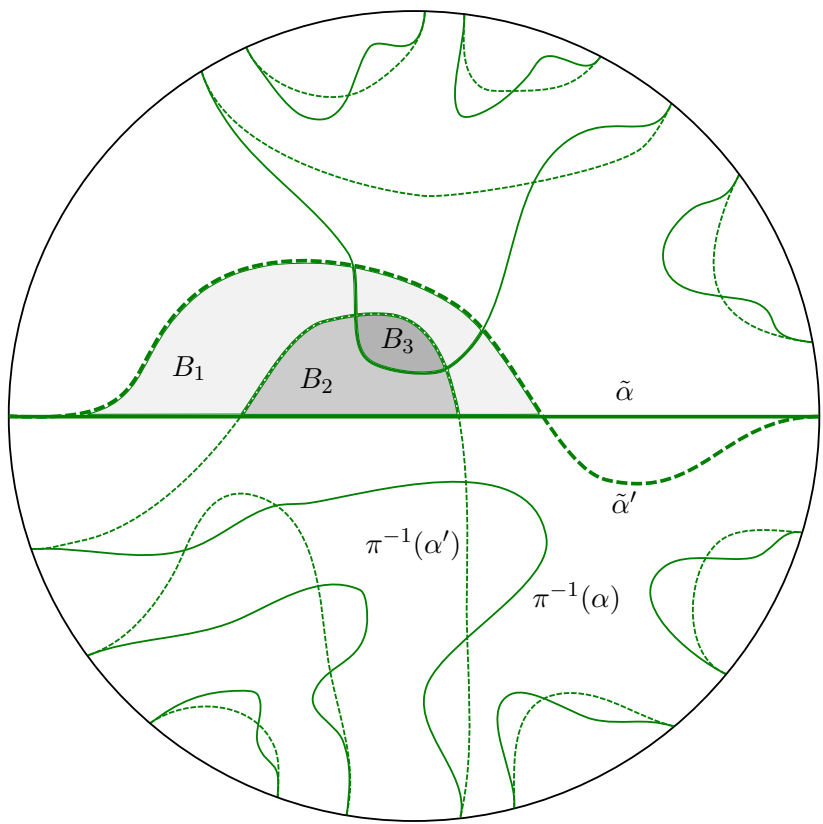
\end{center}
\caption{The successive bigons in the proof of Proposition~\ref{p.exists-bigon}}
\end{figure}


\subsection{Construction of the isotopy}
\label{ss.construction-isotopy-arc}

We will prove the following precise version of Proposition~\ref{p.straightening-lines-non-precise}. 

\begin{prop}
\label{p.straightening-arc-easy}
Let $\alpha,\alpha'$ be two compactly homotopic properly embedded lines in a surface $M$. Then there exists an isotopy of homeomorphisms of $M$, joining the identity to a homeomorphism $f$ such that $f \circ \alpha' = \alpha$. 

Moreover, let $\hat \alpha,\hat \alpha'$ be compact subarcs of $\alpha,\alpha'$ respectively, such that $\alpha$ and $\alpha'$ coincide outside $\hat \alpha$ and $\hat \alpha'$.
Let $W$ be a neighbourhood of the set $\fill(\hat \alpha\cup \hat \alpha')$ in $M$. Let $C,C'$ be two closed subsets of $M$ such that $C \cap \alpha  = C' \cap \alpha'= \emptyset$.
 Then the isotopy $I$ can be chosen such that its support is included in $W$ and such that $C\cap f(C')\subset C\cap C'$.
\end{prop}

Let us anticipate the use of Proposition~\ref{p.straightening-arc-easy} to explain the role that will be played by the sets $C$ and $C'$. In section~\ref{sec.arc-straightening} we will consider a surface $S$ with some closed subset $F$. Given two arcs $\gamma,\gamma'$ in $S$ (with some specific properties), we will construct an isotopy that sends $\gamma'$ to $\gamma$. For an elementary step of the construction, we will consider some component $\alpha$ of $\gamma \setminus F$, and a corresponding component $\alpha'$ of $\gamma' \setminus F$, and we will get an isotopy sending $\alpha'$ to $\alpha$ by applying Proposition~\ref{p.straightening-arc-easy}  in the surface $M = S \setminus F$. Letting $C= \gamma \setminus (F \cup \alpha)$, and likewise $C'= \gamma' \setminus ( F \cup \alpha')$, the last sentence of the statement of Proposition~\ref{p.straightening-arc-easy} will guarantee that the image of $\gamma'$ under the isotopy has no new intersection point with $\gamma$ outside $\alpha$.

The proof of Proposition~\ref{p.straightening-arc-easy} relies on Proposition~\ref{p.exists-bigon} which allows to find a bigon, and on the two following Lemmas~\ref{l.remove-bigon-1} and~\ref{l.remove-bigon-2} which allow to ``remove" a bigon. 

\begin{lemm}
\label{l.remove-bigon-1}
Let $\alpha,\alpha'$ be two properly embedded lines a surface $M$, which coincide outside a compact subset of $M$. Assume that $\alpha,\alpha'$ are quasi-transverse, and that there exists a minimal bigon $(B,\bar\alpha,\bar\alpha')$ for $\alpha,\alpha'$. Then, one can find a neighbourhood $V$ of $B$, and an isotopy $I$ of homeomorphisms of $M$ supported in $V$, joining the identity to a homeomorphism $f$, so that $f(\alpha')\cap\alpha\cap \mathrm{int}(V)=\emptyset$. 

Moreover, given any neighbourhoods $U$ and $U'$ of $\bar\alpha$ and $\bar\alpha'$ respectively, $V$ and $I$ can be chosen so that $V \cap \alpha$ is connected,
 $V$ is included in $B\cup U\cup U'$, $f(V\setminus U')$ is included in $U$, and $\alpha$ and $f(\alpha')$ are quasi-transverse.
\end{lemm}

\begin{figure}[ht]
\def\svgwidth{0.9\textwidth}
\begin{center}
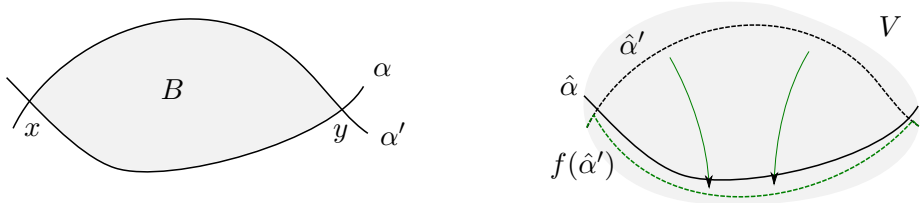
\caption{\label{f.cas-1}Removing a bigon: illustration of the proof of Lemma~\ref{l.remove-bigon-1}}
\end{center}
\end{figure}

\begin{proof}
Let $x,y$ be the vertices of the bigon $(B,\bar\alpha,\bar\alpha')$ and $\alpha_x,\alpha_y$ the connected components in $\alpha\cap\alpha'$ of $x,y$ respectively. Since the properly embedded lines $\alpha$ and $\alpha'$ are quasi-transverse, $\alpha_x$ (resp. $\alpha_y$) either is reduced to the point $x$ (resp. $y$), or is a properly embedded topological half-line with endpoint $x$ (resp. $y$). 

\bigskip

Let us first consider the case where $\alpha_x=\{x\}$ and $\alpha_y=\{y\}$ (see figure~\ref{f.cas-1}); the other cases will be treated at the end of the proof. Since $\alpha$ and $\alpha'$ are quasi-transverse, the point $x$ and $y$ are isolated in $\alpha\cap\alpha'$. Since the bigon $(B,\bar\alpha,\bar\alpha')$ is minimal, $\alpha\setminus B$ (resp. $\alpha'\setminus B$) is the union of two disjoint properly embedded half-lines~: one with endpoint $x$, the other with end point $y$. Using these properties, one can find a neighbourhood $V$ of $B$ so that 
\begin{itemize}
\item $V$ is a closed topological disc, 
\item $\hat\alpha:=V\cap \alpha$ and $\hat\alpha':=V\cap \alpha'$ are ``topological diameters" of $V$ (\emph{i.e.} arcs with endpoints in $\partial V$ whose interiors are contained in the interior of $V$),
\item $V\cap \alpha\cap\alpha'=\{x,y\}$.
\end{itemize}
Moreover, $V$ can be chosen ``as small as desired", \emph{i.e.} included in any given neighbourhood of $B$. 

We denote by $H_+$ the connected component of $V\setminus\hat\alpha$ containing $\mathrm{int}(B)$, and by $H_-$ the connected component disjoint from $B$. We denote by $H_-'$ the connected component of $V\setminus\hat\alpha'$ containing $\mathrm{int}(B)$, and by $H_+'$ the connected component disjoint from $B$. Up to replacing $V$ by a smaller topological disc, we may (and we will) assume that $H_-\subset U$ and $H_+'\subset U'$. This implies that $V\subset B\cup U\cup U'$.

Since  $\hat\alpha$ and $\hat\alpha'$ are quasi-transverse, and since $\hat\alpha$ and $\hat\alpha'$ have exactly two intersection points in $V$, the two endpoints of $\hat\alpha'$ lie in $H_-$. As a consequence, by Schoenflies theorem, one can find a isotopy $I$ supported in $V$ joining the identity to a homeomorphism $f$ such that $f(\hat\alpha')\subset H_-$ (see figure~\ref{f.cas-1}). It follows that 
\begin{equation}
\label{e.inclusion-1}
f(\alpha')\cap \alpha\cap V =f(\hat\alpha')\cap\hat\alpha\subset H_-\cap\hat\alpha = \emptyset.
\end{equation}
Moreover, since $\clos(H_+')\subset U'$ and $H_-\subset U$, one has $V\setminus U'\subset V\setminus\clos(H_+')=H_-'$. Also note that $f(H_-')\subset H_-$ (this follows from the inclusion $f(\hat\alpha')\subset H_-$). As a consequence, on gets
\begin{equation}
\label{e.inclusion-2}
f(V\setminus U')\subset f(V\setminus\clos(H_+'))=f(H_-')\subset H_-\subset U.
\end{equation}
Hence the homeomorphism $f$ satisfies all the required properties, and the proof is complete in the case where $\alpha_x=\{x\}$ and $\alpha_y=\{y\}$.

\bigskip

Now, we explain how to modify the proof in the case where $\alpha_x$ and/or $\alpha_y$ is a properly embedded half-line. The only difference with the previous case is that the ``diameters" $\hat\alpha$ and $\hat\alpha'$ have one or two common endpoint(s), and coincide on some neighbourhood of this (these) endpoint(s). Since the homeomorphism $f$ must coincide with the identity on $\partial V$,  the arcs $\alpha$ and $f(\hat\alpha')$ will have one or two common endpoints as well, and therefore will not be disjoint. Nevertheless, he isotopy $I$ can be chosen such that $f(\hat\alpha')\cap\mathrm{int}(V)\subset H_-$. This inclusion implies $f(H_-')\subset H_-$, and therefore~\eqref{e.inclusion-2} is still valid. Moreover, replacing $V$ by $\mathrm{int}(V)$ in \eqref{e.inclusion-1}, one gets $f(\alpha')\cap \alpha\cap \mathrm{int}(V) = \emptyset$ as desired.
\end{proof}

\begin{lemm}
\label{l.remove-bigon-2}
Let $\alpha,\alpha'$ be two properly embedded lines a surface $M$, which coincide outside a compact subset of $M$. Assume that $\alpha\cap\alpha'$ has only two connected components (\emph{i.e.} $\alpha\cap\alpha'$ is the union of two disjoint properly embedded half-lines). Then there an isotopy $I$, joining the identity the identity to a homeomorphism $f$, so that $f \circ\alpha'=\alpha$. 

Moreover, given any neighbourhoods $U$ and $U'$ of $\hat\alpha:=\clos(\alpha\setminus\alpha')$ and $\hat\alpha':=\clos(\alpha'\setminus\alpha)$ respectively, $I$ can be chosen so that it is supported in $V := \fill(\hat \alpha\cup\hat \alpha')\cup U\cup U'$ and, so that $f(V\setminus U')\subset U$. 
\end{lemm}

\begin{proof}
The proof is entirely similar to those of Lemma~\ref{l.remove-bigon-1}, in the case where $\alpha_x$ and $\alpha_y$ is a properly embedded half-line. The only difference is that we demand $f(\hat\alpha')=\hat\alpha$ instead $f(\hat\alpha')\subset H_-$.
\end{proof}

\begin{proof}[Proof of Proposition~\ref{p.straightening-arc-easy}]
Given a pair of compactly homotopic properly embedded lines $\beta,\beta'$, we will denote by $\hat \beta$ the minimal compact subarc of $\beta$ such that $\beta' = \beta$ outside $\hat \beta$. The subarc $\hat \beta'$ of $\beta'$ is defined symmetrically. The set $\fill(\hat \beta \cup \hat \beta')$ will be denoted by $\widehat \fill(\beta,\beta')$.

By assumption, the set $W$ is a neighborhood of the filled compact set $\widehat \fill(\alpha,\alpha')$.
Every neighbourhood of a filled compact subset of $M$ contains a filled neighbourhood of this subset (see Lemma~\ref{l.filled-neighbourhood}). Therefore, up to replacing $W$ by a smaller neighbourhood of  $\fill(\alpha\cup\alpha')$, we can, and we will, assume that $W$ is filled. 

We will construct a finite sequence $I_0,I_1,\dots$ of isotopies of homeomorphisms of $M$ compactly supported in $W$, so that $I_i$ is joining the identity to a homeomorphism $f_i$. We will define a sequence of properly embedded topological lines $\alpha_0',\alpha_1',\dots$ by setting $\alpha_0':=f_0(\alpha')$ and $\alpha_{i+1}':=f_{i+1}(\alpha_i')$. The main property that we want is that the number of connected component of $\alpha_i'\cap \alpha$ decreases with $i$. So after a finite number of step, we will obtain a line $\alpha_{i_0+1}'$ which coincides with $\alpha$. In other words, the homeomorphism $f:=f_{i_0+1}\circ\dots\circ f_0$ will satisfy $f \circ \alpha'=\alpha$ as desired, and the concatenation of the isotopies $I_0,\dots,I_{i_0+1}$ will provide an isotopy $I$, supported in $W$, joining the identity to $f$. Moreover, we will consider the closed sets $C_0',C_1',\dots$ defined by $C_0':=f_0(C')$ and $C_{i+1}':=f_{i+1}(C_i')$. The construction will be made in such a way that $C_i' \cap  C\subset C_{i-1}'\cap C$. In particular, we will have $f(C')\cap C=C_{i_0+1}'\cap C\subset C'\cap C$, as desired.

\bigskip

\noindent \textit{Construction of the isotopy $I_0$: putting the arcs in quasi-transverse position.} 

\smallskip

Since $C'$ is disjoint from $\alpha'$ and since $W$ is a neighbourhood of $\widehat \fill(\alpha,\alpha')\supset\hat \alpha'$, we can find a neighbourhood $U'$ of $\hat \alpha'$ so that $U'\subset W$ and $C'\cap U'=\emptyset$. Proposition~\ref{p.transverse} provides us with an isotopy $I_0$, compactly supported in $U'$, joining the identity to a homeomorphism $f_0$ so that the topological lines $\alpha$ and $f_0(\alpha')$ are quasi-transverse. Since the support of $I_0$ is disjoint from $C'$,  we have $C\cap f_0(C')=C\cap C'$.

So, we have constructed an isotopy $I_0$, supported in $W$, joining the identity to a homeomorphism $f_0$, such that, if we set $\alpha_0':=f_0(\alpha')$ and $C_0':=f_0(C')$, then the properly embedded lines $\alpha$ and $\alpha_0'$ are quasi-transverse (in particular, $\alpha\cap\alpha_0'$ has finitely many connected components), and $C_0'\cap C_0=C'\cap C$. 

\bigskip

\noindent \textit{Construction of the isotopies $I_1,\dots,I_{i_0}$ : reducing the number of intersection points.} 

\smallskip

Suppose that we are given a properly embedded line $\alpha_i'$ and a closed set $C_i'$ in $M$, such that $C_i'\cap\alpha_i'=\emptyset$. Assume that:
\begin{itemize}
\item[$(\star)$] $\alpha$ and $\alpha_i'$ are quasi-transverse and compactly homotopic, $W$ is a neighbourhood of the $\widehat \fill(\alpha,\alpha_i')$.
\end{itemize}
Assume moreover that $\#\pi_0(\alpha\cap\alpha_i')\geq 3$, where $\#\pi_{0}$ denotes the number of connected components. 
 
The lines $\alpha$ and $\alpha_i'$ satisfy the hypotheses of Proposition~\ref{p.exists-bigon}, which provides us with a minimal bigon $(B,\bar\alpha,\bar\alpha_i')$ for $\alpha,\alpha_i'$. On the one hand, $B$ is contained in the set $\widehat\fill(\alpha\cup\alpha_i')$; hence $W$ is a neighbourhood of $B$. On the other hand, the closed sets $C$ and $C_i'$ satisfy $C\cap\alpha=\emptyset$ and $C_i'\cap \alpha_i'=\emptyset$. Thus we can find some neighborhoods $U$ and $U'$ of the arcs $\bar\alpha$ and $\bar\alpha_i'$ respectively, so that $U\cup U'\subset W$, $C\cap U=\emptyset$ and $C_{i}'\cap U'=\emptyset$. We apply Lemma~\ref{l.remove-bigon-1} to the properly embedded lines $\alpha$ and $\alpha_i'$, the bigon $(B,\bar\alpha,\bar\alpha_i'))$, and the neighbourhoods $U,U'$. It provides us with a neighborhood $V$ of $B$ and an isotopy $I_{i+1}$ supported in $V$, joining the identity to a homeomorphism $f_{i+1}$. We set $\alpha_{i+1}':=f_{i+1}(\alpha_i')$ and $C_{i+1}':=f_{i+1}(C_i')$. Note that $C_{i+1}'\cap\alpha_{i+1}'=f_{i+1}(C_i'\cap\alpha_i')=\emptyset$.

Since $I_{i+1}$ is an isotopy with compact support, $\alpha_{i+1}'$ is a properly embedded line, compactly homotopic to $\alpha_i'$. It follows that $\alpha_{i+1}'$ is compactly homotopic to $\alpha$. According to the statement of Lemma~\ref{l.remove-bigon-1}, the lines $\alpha_{i+1}'=f_{i+1}(\alpha_i')$ and $\alpha$ are quasi-transverse. Since $W$ is filled, and since $\alpha$ and $\alpha_{i+1}'$ coincide outside $W$, it follows that $W$ is a neighborhood of $\widehat \fill(\alpha,\alpha_{i+1}')$. Hence the properly embedded line $\alpha_{i+1}'$ satisfies property $(\star)$.

Let us prove that $C_{i+1}'\cap C\subset C_i'\cap C$. According to the statement of Lemma~\ref{l.remove-bigon-1}, the neighbourhood $V$ and the isotopy $I_{i+1}$ can be chosen such that $V\subset B\cup U\cup U'$ and $f_{i+1}(V \setminus U')\subset U$. Let us write 
$$C_{i+1}'\cap C=f_{i+1}(C_i')\cap C= \left(f_{i+1}(C_i'\cap V)\cap C\right) \sqcup  \left(f_{i+1}(C_i'\setminus V)\cap C\right).$$ 
On the one hand, we have $f_{i+1}(C_i'\setminus V)=C_i'\cap V$, since $f_{i+1}$ is supported in $V$. On the other hand, we have $f_{i+1}(C_i'\cap V)\cap C=\emptyset$: indeed $C_{i}'$ is disjoint from $U'$, thus 
$C_{i}' \cap V \subset V \setminus U'$ which implies $f_{i+1}(C_{i}' \cap V) \subset f_{i+1}(V \setminus U') \subset U$ and $U$ is disjoint from $C$.
Hence we obtain
$$C_{i+1}'\cap C=(C_i'\setminus V)\cap C\subset C_i'\cap C.$$

Now let us prove that $\#\pi_0(\alpha\cap\alpha_i')<\#\pi_0(\alpha\cap\alpha_{i+1}')$. By assumption, the arcs $\alpha,\alpha_i'$ are quasi-transverse, and the intersection $\alpha\cap\alpha_i'$ has at least three connected components. Hence, at least one of the two vertices of the bigon $B$ is an isolated point $z$ in $\alpha\cap\alpha_i'$. 
The isotopy is supported in $V$ and $\alpha_{i+1}' \cap \alpha$ does not meet $V$. Thus
 (1) every connected component of $\alpha_{i+1}' \cap \alpha$ is included in a connected component of $\alpha_{i}' \cap \alpha$. Furthermore, according to the statement of Lemma~\ref{l.remove-bigon-1}, the neighborhood $V$ can be chosen so that $V \cap \alpha$ is connected; this entails that (2) no two connected components of $\alpha_{i+1}' \cap \alpha$ are included in the same connected component of $\alpha_{i}' \cap \alpha$. Finally, we have that (3) no connected component of $\alpha_{i+1}' \cap \alpha$ is included in the connected component $\{z\}$ of $\alpha_{i}' \cap \alpha$ (since $z\in V$ and since $\alpha_{i+1}'\cap \alpha\cap V=\emptyset$). From properties (1), (2) and (3), we get that   $\#\pi_0(\alpha\cap\alpha_i')<\#\pi_0(\alpha\cap\alpha_{i+1}')$, as wanted.


\bigskip

Since the line $\alpha_{i+1}'$ satisfies Property~$(\star)$, we can iterate the construction, as long as $\#\pi_0(\alpha\cap\alpha_{i+1}')\geq 3$. We obtain a sequence of properly embedded topological lines $\alpha_1',\alpha_2',\dots$. Since $\#\pi_0(\alpha\cap\alpha_{i+1}')<\#\pi_0(\alpha\cap\alpha_i')$, the construction stops after a finite number of steps, and we get a properly embedded line $\alpha_{i_0}'$ such that $\#\pi_0(\alpha\cap\alpha_{i_0}')<3$. Note that we necessarily have $\#\pi_0(\alpha\cap\alpha_{i_0}')=2$, since $\alpha$ and $\alpha_{i_0}$ are two properly embedded lines which coincide outside a compact set, and since $\alpha_{i_0}\neq\alpha$ (because $\alpha_{i_0}'$ is obtained by applying Lemma~\ref{l.remove-bigon-1}).
 
\bigskip

\noindent \textit{Construction of the isotopy $I_{i_0+1}$ : removing the last bigon.} 

\smallskip

Now, we have constructed a properly embedded topological line $\alpha_{i_0}'$ and a closed set $C_{i_0}'$, such that $C_{i_0}'\cap\alpha_{i_0}'=\emptyset$, such that $\alpha_{i_0}'$ satisfies property $(\star)$ 
and such that $\#\pi_0(\alpha\cap\alpha_{i_0}')=2$.


By assumption $W$ is a neighborhood of $\widehat \fill(\alpha,\alpha_{i_0}')$, and we have $C\cap\alpha=C_{i_0}'\cap \alpha_{i_0}'=\emptyset$. Hence, we can find  some neighborhoods $U_{i_0}$ and $U_{i_0}'$ of $\clos(\alpha\setminus\alpha_{i_0}')$ and $\clos(\alpha_{i_0}'\setminus\alpha)$ respectively, so that $U_{i_0}\cup U_{i_0}'\subset W$, $C\cap U_{i_0}=\emptyset$ and $C_{i_0}'\cap U_{i_0}'=\emptyset$. We can apply Lemma~\ref{l.remove-bigon-2} to the arcs $\alpha,\alpha_{i_0}'$ and the neighbourhoods $U_{i_0},U_{i_0'}$. We get an isotopy $I_{i_0+1}$, supported in $W$, joining the identity to a homeomorphism $f_{i_0+1}$ so that $f_{i_0+1} \circ \alpha_{i_0'}=\alpha$. We denote $\alpha_{i_0+1}':=f_{i_0+1}(\alpha_{i_0}')$ and $C_{i_0+1}':=f_{i_0+1}(C_{i_0}')$. The same arguments as above (in the construction of the isotopy $I_i$) show that $f_{i_0+1}(C_{i_0}')\cap C\subset C_{i_0}'\cap C.$

\bigskip

As explained at the beginning of the proof, we set $f:=f_{i_0+1}\circ \dots\circ f_0$. The isotopies $I_0,I_1,\dots,I_{i_0+1}$ can be glued together to obtain an isotopy $I$, supported in $W$, joining the identity to $f$. We get $f \circ \alpha'=\alpha_{i_0+1}'=\alpha$. Similarly, the inclusions $C_{i+1}\cap C=f_{i+1}(C_i)\cap C\subset C_i\cap C$ imply that $f(C')\cap C\subset C'\cap C$. This completes the proof of Proposition~\ref{p.straightening-arc-easy}.
\end{proof}

\section{Straightening a closed curve}
\label{sec.redresser-les-courbes}

In this section, we adapt the statement of the previous section to the case of simple closed curves. Basically, we will see that two simple closed curves that are homotopic are also isotopic. Remember that, if $E$ is a subset of a surface $M$, the filling of $E$ in $M$ is the union of $E$ and of the connected components of $M \setminus  E$ which have compact closures. The following is a refinement on~\cite{epstein}, Theorem 4.1.

\begin{prop}[Simple closed curves] \label{prop.redresser-les-courbes}
Let $c,c': \bbS^1 \to M$ be two essential two-sided simple closed curves on a surface $M$, and $V$ be a neighborhood of the filling $\fill(c \cup c')$. Assume there exists a homotopy $H$ from $c$ to $c'$. Then there exists an isotopy $I$ on $M$, supported in $V$, joining the identity to a homeomorphism $g$ such that $g \circ c = c'$.

Moreover, if $c(0)=c'(0)$ and if the trajectory of $c(0)$ under $H$ is nulhomotopic in $M$, 
then we may require that $I$ fixes $c(0)$.
\end{prop}

The proof of Proposition~\ref{prop.redresser-les-courbes} is very similar to those of Proposition~\ref{p.straightening-lines-non-precise}. We outline it in the remainder of this section. Note that last sentence of the statement of Proposition~\ref{prop.redresser-les-courbes} is not true when the curves are one-sided, see~\cite{Epstein-Zieschang}. 

The proof of the following lemma is similar to that of Proposition~\ref{p.transverse}, and is left to the reader.

\begin{lemm}
Let $c,c': \bbS^1 \to M$ be two simple closed curves on a surface $M$. Then there exists an isotopy of homeomorphisms of $M$, from the identity to a homeomorphism $f$ such that $f(c')$ and $c$ are transverse.
Furthermore, if $c(0)=c'(0)$, then the isotopy can be chosen to fix $c(0)$.
\end{lemm}

We define a \emph{bigon} and a \emph{minimal bigon} for two simple closed curves $c,c'$ as we have done for two properly embedded lines (see definition~\ref{d.bigon}). The proof of the lemma below is similar to that of Lemma~\ref{p.exists-bigon}: one first find a minimal bigon in the universal cover, and then  prove that the bigon projects injectively on $M$, using that the curve is two-sided; details are left to the reader.

\begin{lemm}
Let $c,c'$ be two transverse two-sided essential simple closed curves on a surface $M$. Assume there exists a homotopy $H$ from $c$ to $c'$.
If $c$ and $c'$ are not disjoint, then there exists a minimal bigon for $c,c'$ in $M$.
\end{lemm}

\begin{proof}[Proof of Proposition~\ref{prop.redresser-les-courbes}]
Let $c,c'$ as in the statement. Applyling the first lemma, we may assume that the curves are transverse. 
If $c$ and $c'$ are not disjoint, then according to the second lemma we may find a minimal bigon. Thus, proceeding as in the proof of Proposition~\ref{p.straightening-arc-easy}, we can find an isotopy that removes all the  intersection points. We are left with the case of two disjoint curves. Then they bound an annulus (Lemma~\ref{l.bound-annulus}), and they are isotopic. Let $I$ denote the isotopy, whose support is included in some given neighborhood $V$ of $\fill(c \cup c')$, and $f$ denotes its time one, we have $f(c')=c$.

It remains to see that the point $c(0)$ may be kept fixed under the isotopy $I$, provided that it has a nulhomotopic trajectory $\beta$ under the given homotopy $H$. For this purpose, we introduce the map $\Phi : \homeo(V) \to V$ given by $h \mapsto h(c(0))$. This map is a locally trivial fibration (Lemma~\ref{l.fibration}). If $\gamma$ is a loop based at $c(0)$, we denote by $[\gamma]$ the class of $\gamma$ in the fundamental group $\pi_1(M,c(0))$.  Let $\alpha$ be the trajectory of $c(0)$ under the inverse of the  isotopy $I$. The concatenation of the homotopy $H$ and the inverse of the isotopy $I$ is a homotopy $\widehat H$ from $c$ to $c$. The trajectory of  $c(0)$ under $\widehat H$ is the concatenation $\beta \star \alpha$. According to Fact~\ref{fact.commute}, $[\beta \star \alpha]$ commutes with $[c]$ in $\pi_1(M,c(0))$. Since $[\beta]$ is trivial by assumption, we deduce that $[\alpha]$ commutes with $[c]$.

If $M$ is the torus or the Klein bottle, then $V=M$ and the map $\Phi$ induces a surjective map on the fundamental groups. Thus, up to composing $I$ with an isotopy from the identity to the identity, we can assume that $[\alpha]$ is trivial. In the other cases, every abelian subgroup of the fundamental group of $M$ is cyclic (see section~\ref{ss.surface-groups}). Furthermore, since $c$ is a two-sided simple closed curve, $[c]$ is a primitive element (Proposition~\ref{prop.primitive}). Thus there exists $k$ such that $[\alpha] = [c]^k$. Again up to composing $I$ with an isotopy that ``rotates $c$'' from the identity to a homeomorphism $g$ such that $g(c) = c$, we may assume that $[\alpha]$ is trivial.

We deduce that $\alpha$ is also nulhomotopic in $V$. Indeed, according to Lemma~\ref{l.filled-neighbourhood}, up to decreasing $V$ we may assume that $V$ is filled. Let $P: X \to M$ be the universal cover, and $\tilde V = P^{-1}(V)$. Note that $\tilde V$ is also filled. Since $\alpha$ is nulhomotopic in $M$, it  lifts to a closed curve $\tilde \alpha$. Since $\tilde V$ is a neighborhood of the compact set $\tilde \alpha$, there exists a topological disk $D$ containing $\tilde \alpha$ bounded by a curve contained in $\tilde V$ (Lemma~\ref{lemma.approximation-disc}). Since $\tilde V$ is filled it contains $D$. Now the curve $\tilde \alpha$ is nulhomotopic in $D$, thus in $\tilde V$, and the homotopy projects to a homotopy in $V$ from $\alpha$ to a trivial curve.

Seeing $I: [0,1] \to \homeo(V)$ as a path in $\homeo(V)$, we have $\alpha = \Phi \circ I$.
Let $(\alpha_{t})_{t \in [0,1]}$ be a homotopy of closed curves, included in $V$, with basepoint fixed, from $\alpha$ to the constant curve $\alpha(0) = c(0)$. Since the map $\Phi: \homeo(V) \to M$ is a locally trivial fibration, there exists a homotopy $(I_{t})_{t \in [0,1]}$ of paths in $\homeo(V)$, with fixed end-points, such that $I_{0}=I$ and $\Phi \circ I_{t} = \alpha_{t}$ for every $t$. Then $I_{1}$ is an isotopy supported in $V$, from the identity to $f$, for which the trajectory of $c(0)$ is fixed.
\end{proof}

\bigskip

For completeness we state a lemma which allows to extend the proof of the first part of Proposition~\ref{prop.redresser-les-courbes} to the case of one-sided curve. This case  will not be used in the present text, and the proof is left to the reader.
\begin{lemm}
Let $c,c'$ be two transverse $1$-sided essential simple closed curves on a surface $M$. Assume there exists a homotopy $H$ from $c$ to $c'$.
Then either $c$ and $c'$ have more than one intersection point, and then there exists a minimal bigon for $c,c'$ in $M$; or $c$ and $c'$ have exactly one intersection point, and then there is a neighborhood of $c$ in $M$ which is homeomorphic to the M\"obius band, and in which $c$ and $c'$ are cores of the M\"obius band.
\end{lemm}



\section{Straightening an arc relatively to a closed set}
\label{sec.arc-straightening}

All along this section, we consider a connected surface $S$ (not necessarily compact nor orientable), and a closed subset $F$ of this surface. 

\subsection{Statement and strategy}

The aim of the present section is to prove the following result.

\begin{prop}[``Arc straightening lemma'']
\label{l.multicoupure}
Let  $f \in \homeo(S, F)$. Assume that $f$ is strongly homotopic to the identity relatively to $F$. Let $\alpha$ be an arc whose endpoints belong to $F$. Then there exists an isotopy $(f_{t})_{t \in [0,1]}$ in $\homeo(S,F)$ such that $f_{0} = f$ and $f_{1}$ fixes $\alpha$ pointwise.
\end{prop}

The proof has two main steps. In the first step, we will make a preliminary isotopy relative to $F$, from $f$ to a homeomorphism $f''$ that fixes lots of subarcs of $\alpha \setminus F$. We denote by $F''$ the union of $F$ and these subarcs of $\alpha$. The important property is that each component of $\alpha \setminus F''$ is at positive distance from the union of the other components of $\alpha \setminus F''$. We also prove that $f''$ is strongly homotopic to the identity relatively to $F''$. Surprinsingly, this last property appears to be quite difficult to prove: we need to apply the machinery of the previous sections, and in particular the finite criterium of section~\ref{sec.fini-homo-implique-homo}. This first step is detailed in sections~\ref{sub.bringing-points} and~\ref{sub.bringing-neighborhoods}.

In the second step, we will build the isotopy required by Proposition~\ref{l.multicoupure}. Up to replacing $F$ by $F''$ and $f$ by $f''$, we may now assume that each component $\alpha_{i}$ of $\alpha \setminus F$ is compactly homotopic to its image. Section~\ref{s.isotopy-arc} provides an isotopy that sends $f(\alpha_{i})$ back to $\alpha_{i}$. Of course in general there are infinitely many $\alpha_{i}$'s and we have to make sure that the infinite concatenation of the corresponding isotopies converges. The convergence will be obtained by performing the isotopies in a carefully chosen order. More precisely, we will show that these isotopies may be gathered as a first sequence $(I_{i})$ having supports that are pairwise disjoint, disjoint from $F$, and with diameters converging to zero, and a second sequence $(I'_{i})$ sharing the same properties. These properties ensure the convergence. The second step is detailed in section~\ref{sub.end-isotopy}.

\subsection{Bringing back sequences of points}
\label{sub.bringing-points}

We consider $f$ and $\alpha$ as in Proposition~\ref{l.multicoupure}. We denote by $\Lambda$ the set of all couples $\lambda:=(\beta^\lambda,z^\lambda_\infty)$, where $\beta^\lambda$ is a connected component of $\alpha \setminus F$, where $z^\lambda_\infty$ is one of the two endpoints of $\beta^\lambda$, and such that $z^\lambda_\infty$ is not isolated in $F$. For every  $\lambda=(\beta^\lambda,z^\lambda_\infty)\in\Lambda$, we choose a sequence of points $(z_k^\lambda)_{k\in\mathbb{N}}$ on $\beta^\lambda$ converging to $z^\lambda_\infty$. 
The lemma below provides a first isotopy relative to $F$ which ``brings back" subsequences of the points $\{z_k^\lambda \mid k\in\bbN, \lambda\in\Lambda\}$.

\begin{lemm}\label{l.straightening-points}
Up to replacing, for each $\lambda\in\Lambda$, the sequence of points $(z^\lambda_k)_{k\in\bbN}$ by a subsequence, there exists a homeomorphism $f'\in\homeo(S,F)$ such that
\begin{enumerate}
\item $f'$ is isotopic to $f$ relatively to $F$,
\item $f'$ fixes the set $F' := F \cup \left\{z^\lambda_k \mid \lambda\in \Lambda, k \in\mathbb{N} \right\}$ pointwise,
\item $f'$ is strongly homotopic to the identity relatively to $F'$.
\end{enumerate}
\end{lemm}

\begin{rema}
\label{rema.endpoints-1}
For further use, let us note that every endpoint of a connected component of $\alpha \setminus F'$ is an isolated point of $F'$.
\end{rema}

Roughly speaking, the proof of Lemma~\ref{l.straightening-points} goes as follows.
We consider the strong homotopy relative to $F$ from the identity to $f$. Up to extracting, the trajectories of the $z^\lambda_{k}$ are pairwise disjoint.
By pushing the points backwards along their trajectories, we first construct an isotopy from $f$ to a homeomorphism $f'$ that fixes all the $z^\lambda_{k}$. Then comes the hard part, namely proving that $f'$ is strongly homotopic to the identity relatively to the union $F'$ of $F$ and the $z^\lambda_{k}$'s. We apply the criterium of section 4: it suffices to prove the existence of a strong homotopy relative to any finite subset of $F'$.   
For this, we first make an isotopy from $f'$ to a homeomorphism that pointwise fixes a finite collection of simple closed curves, one through each point $z^\lambda_{k}$ in the given finite subset of $F'$ (this uses section 6).
Then we prove that the resulting homeomorphism is strongly homotopic to the identity relatively to the union of $F$ and the curves (this uses either hyperbolic geometry or the selection result of section~\ref{sec.selection}).

\begin{proof}[Proof of Lemma~\ref{l.straightening-points}]
We remark that the lemma is useless in the case when every point of $F$ is isolated (since, in this case, $\Lambda$ is empty). Thus we may, and we will, assume that $F$ is infinite. By assumption, there exists a strong homotopy $H$ from the identity to $f$, relative to $F$.

\begin{claim}[Figure~\ref{f.bringin-back1}]
\label{claim.straightening-points}
Up to replacing, for each $\lambda\in\Lambda$, the sequence $(z^\lambda_k)_{k\in\mathbb{N}}$ by a subsequence, there exists a homeomorphism $f'$ such that
\begin{enumerate}
\item there is an isotopy $I$ from $f$ to $f'$ relative to $F$,
\item $f'$ fixes  the set $F' := F \cup \{z^\lambda_{k} \mid \lambda\in\Lambda, k\in\mathbb{N} \}$ pointwise,
\item for every $k\in\mathbb{N}$ and every $\lambda\in\Lambda$, the trajectory of the point $z^\lambda_{k}$ under the homotopy $H \star I$ is nulhomotopic in $S \setminus F$.
\end{enumerate}
\end{claim}

\begin{figure}[ht]	
\begin{center}
\def\svgwidth{0.8\textwidth}
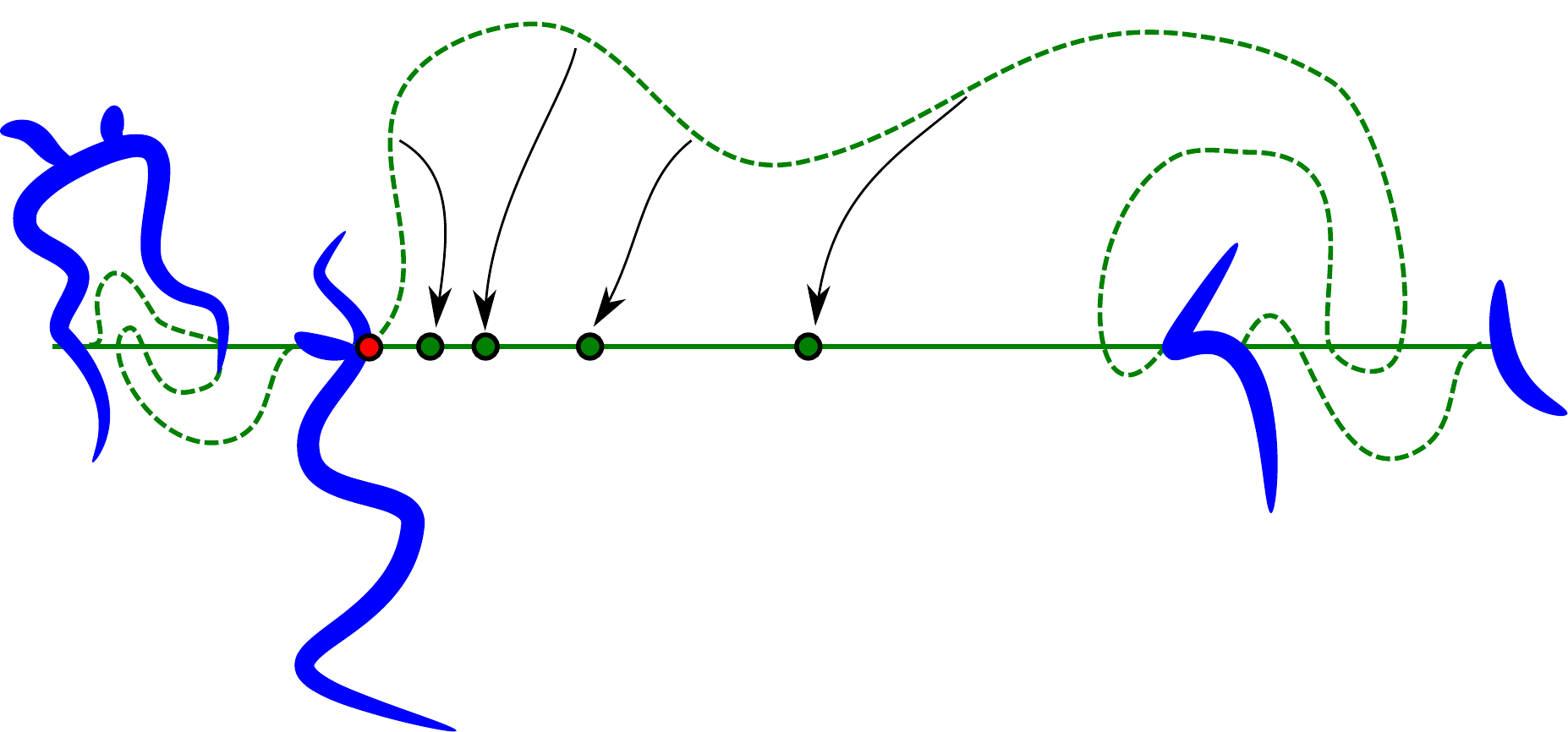
\end{center}
\caption{\label{f.bringin-back1}Bringing back points (I)}
\end{figure}

\begin{proof}[Proof of Claim~\ref{claim.straightening-points}]
For every $\lambda\in\Lambda$ and every  $k\in\bbN$, let $\gamma_k^\lambda$ denote the trajectory of $z_k^\lambda$ under the homotopy $H$. Note that since $H$ is a strong homotopy, the image of $\gamma_k^\lambda$ is a compact subset of $S \setminus F$  for every $\lambda\in\Lambda$ and every $k\in\bbN$. Furthermore, for every $\lambda\in\Lambda$, the sequence $(\gamma^\lambda_k)$ uniformly converges to the point $z^\lambda_{\infty}$. Thus, up to taking subsequences, we may assume that the trajectories $\gamma^\lambda_{k}$ and $\gamma^{\lambda'}_{k'}$ are disjoint for $(\lambda,k)\neq (\lambda',k')$, and that, for every $\varepsilon>0$, the diameter of $\gamma^\lambda_{k}$ is less than $\varepsilon$ except for a finite number of $(\lambda,k)$. Then, for each $(\lambda,k)$, we may choose a neighborhood $V^\lambda_{k}$ of the image of the trajectory $\gamma_{k}^\lambda$ in $S\setminus F$ with the same properties, that is such that $V^\lambda_{k}$ and $V^{\lambda'}_{k'}$ are disjoint for $(\lambda,k) \neq (\lambda',k')$ and such that, for every $\varepsilon>0$, the diameter of $V^\lambda_{k}$ is less than $\varepsilon$ except for a finite number of $(\lambda,k)$. The map associating to any homeomorphism $h$ of a surface the image under $h$ of a given point is a fiber bundle (lemma~\ref{l.fibration}): thus we may find for each $(\lambda,k)$, an isotopy $I_k^\lambda$ of homeomorphisms of $S$, supported in $V^\lambda_{k}$, such that the trajectory of the point $z^\lambda_{k}$ for the isotopy $I_k^\lambda$ is the same as the trajectory $\gamma_k^\lambda$ but in the opposite direction of travel. Note that the isotopies $\{I_k^\lambda, k\in\bbN, \lambda\in\Lambda\}$ are pairwise commuting, since their supports are pairwise disjoint. Let $I'$ be the infinite composition of all these isotopies. Pre-composing with $f$, we get an isotopy $I$ from $f$ to a homeomorphism $f'$ having the desired properties.
\end{proof}

To complete the proof of Lemma~\ref{l.straightening-points} it remains to see that, up to a new extraction of subsequences, the homeomorphism $f'$ provided by Claim~\ref{claim.straightening-points} is strongly homotopic to the identity relatively to the set $F'$. We will apply the results of section~\ref{sec.fini-homo-implique-homo} which roughly say that it is enough to check that $f'$ is homotopic to the identity relatively to every finite subset of $F'$.

Let $Z_\infty = \{z^\lambda_\infty \mid \lambda \in \Lambda\}$. Remember that every point in $Z_\infty$ is non-isolated in $F$. In other words, each singleton $\{z^\lambda_\infty\}$ has empty interior in $F$. Since $\Lambda$ is countable, it follows by Baire theorem that the set $Z_\infty$ has empty interior in $F$. Thus we may choose a countable dense subset $F_0$ of $F$ which is disjoint from $Z_\infty$.

\begin{claim}[Figure~\ref{f.bringing-back2}]
\label{claim.exist-curves}
Up to replacing $(z_k^\lambda)_{k\in\mathbb{N}}$ by a subsequence for every $\lambda$, one can find some simple closed curves $\{c^\lambda_{k} \mid \lambda\in\Lambda, k \in \bbN\}$ in $S$ with the following properties:
\begin{enumerate}
\item for every $(\lambda,k)$, the curve $c^\lambda_{k}$ goes through the point $z^\lambda_{k}$,
\item  for every $(\lambda,k)$, the curve $c^\lambda_{k}$ is disjoint from the set $F_0 \cup Z_\infty$,
\item for $(\lambda,k)\neq (\lambda',k')$, 
the curves $c^\lambda_{k}$ and $c^{\lambda'}_{k'}$ are essential, and not homotopic in $S \setminus F_{0}$,
\item for $(\lambda,k)\neq (\lambda',k')$, 
the fillings of the sets $c^\lambda_{k}\cup f'(c^\lambda_{k})$ and $c^{\lambda'}_{k'}\cup f'(c^{\lambda'}_{k'})$ in the surface $S\setminus F_{0}$ are disjoint,
\item for every $\lambda$, the sequence $(c^\lambda_{k})_{k \geq 0}$ uniformly converges towards $z^\lambda_\infty$.
\end{enumerate}
\end{claim}

\begin{figure}[ht]	
\begin{center}
\def\svgwidth{0.8\textwidth}
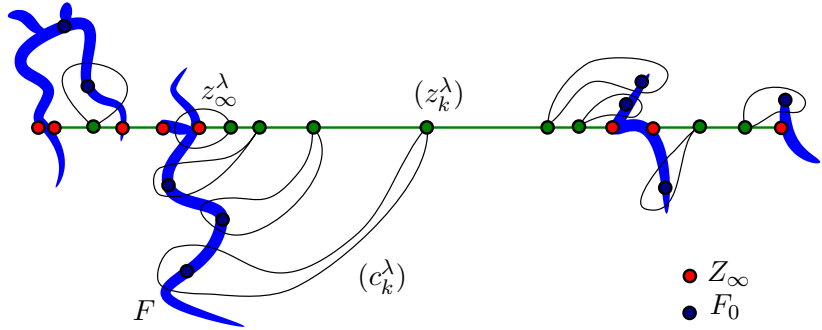
\end{center}
\caption{\label{f.bringing-back2}Bringing back points (II)}
\end{figure}

We refer to section~\ref{ss.filling} for the definition of the filling of a set. Note that the curves are allowed to intersect the set $F$. During the proof of this claim, we shall use the two easy facts below, whose proofs are left to the reader.

\begin{fact}
\label{fact.separating-fillings}
Let $D$ be a topological closed disc in $S$, $E$ be a closed subset of $\mathrm{int}(D)$ and $E'$ be a closed subset of $S\setminus D$. Let $x$ be a point in the connected component of $D\setminus E$ containing $\partial D$. Then the fillings of $E$ and $E'$ in the surface $S\setminus\{x\}$ are disjoint. More precisely, the filling of $E$ is include in $D$ while the filling of $E'$ is disjoint from $D$.
\end{fact}

\begin{fact}
\label{fact.not-homotopic}
Let $c,c'$ be two simple closed curves in $S$. Assume $c$ is included in a disk $D$, that $c'$ is disjoint from $D$, and $x, y$ are two points of $D \setminus c$ that are separated by $c$ in $D$. Then $c$ and $c'$ are not homotopic in $S \setminus \{x, y\}$.
\end{fact}



\begin{proof}[Proof of Claim~\ref{claim.exist-curves}]
We proceed by induction. Assume that one has already constructed some curves $c_1=c_{k_1}^{\lambda_1},\dots,c_{n-1}=c_{k_{n-1}}^{\lambda_{n-1}}$ satisfying Properties 1 to 4. Fix $\lambda_n\in\Lambda$, an integer $k_{min}\in\bbN$, and a neighbourhood $U_n$ of the point $z^{\lambda_n}_\infty$ in $S$. We want to construct an interger $k_n\geq k_{min}$ and a simple closed curve  $c_n=c_{k_{n}}^{\lambda_{n}}$ included in $U_n$, so that the curves $c_1,\dots,c_n$ satisfy Properties 1 to 4. Note that Property 5 will follow automatically from the construction, provided that we choose $U_n$ in such a way that $\mathrm{diam}(U_n)$ goes to $0$ as $n$ goes to infinity.

Let us proceed to the construction of the interger $k_n$ and the curve $c_n$. Since the curves $c_1,\dots,c_{n-1}$ are disjoint from the set $Z_\infty$ (Property 2), we can find a topological closed disc $D$ disjoint from the curves $c_1,\dots,c_{n-1},f'(c_1),\dots,f'(c_n)$, included in $U_{n}$, and such that $z_\infty^{\lambda_n}\in\mathrm{int}(D)$. Then we can choose another topological closed disc $D'$, such that $z_\infty^{\lambda_n}\in\mathrm{int}(D')$, and such that $D'\cup f'(D')\subset D$. Since $z_\infty^{\lambda_n}$ is not isolated in $F$ and since $F_0$ is dense in $F$, we can assume moreover that $(D\setminus D')\cap F_0\neq\emptyset$. The sequence $(z_k^{\lambda_n})$ converges  towards the point $z_\infty^{\lambda_n}$. Hence, we can pick an interger $k_n\geq k_{min}$, such that the point $z_{k_n}^{\lambda_n}$ is in the interior of the disc $D'$. Finally, we may choose the simple closed curve   $c_n$ which goes through the point $z_{k_n}^{\lambda_n}$, which is contained in $D'$, and which is not nullhomotopic in $D'\setminus\{z_{\infty}^{\lambda_n}\}$. Since $F_0 \cup Z_\infty$ is countable, we may assume moreover that the curve $c_n$ is disjoint from $F_0 \cup Z_\infty$. 

By construction, the curve $c_n$ satisfies Properties 1 and 2, and is included in $U_{n}$, and it is not nullhomotopic in $D'\setminus\{z_\infty^{\lambda_n}\}$. Since $z_\infty^{\lambda_n}$ is non-isolated in $F$ and since $F_0$ is dense in $F$, we can find a point $x\in D'\cap F_0$ such that  the curve $c_n$ is not nullhomotopic in $D'\setminus\{x\}$. Again by construction, we can find a point $y\in F_0$ in the annulus $(D\setminus D')$.  The points $x$ and $y$ are separeted by the curve $c_n$ in the disc $D$. From Lemma~\ref{lemm.bounding-disk}, it follows that $c_n$ is essential in $S \setminus \{x,y\}$, hence essential in $S\setminus F_0$. Moreover, from Fact~\ref{fact.not-homotopic}, it follows that the curves $c_n$ and  $c_i$ for $i=1,\dots,n-1$ are not homotopic in $S\setminus\{x,y\}$, thus neither homtopic in $S \setminus F_{0}$, and we get Property 3. Moreover, $c_n$ is included in the disc $D'$, and the interior of the disc $D$ contains 
  $D'\cup f'(D')$, thus also $c_{n} \cup f'(c_{n})$, while
the curves $c_1,\dots,c_{n-1},f'(c_1),\dots,f'(c_{n-1})$ are disjoint from the disc $D$.
 Applying Fact~\ref{fact.separating-fillings} 
 we obtain that for $i=1,\dots,n-1$, the fillings of  $c_n\cup f'(c_n)$ and $c_i\cup f'(c_i)$ in $S\setminus\{y\}$ are disjoint, hence their fillings in $S \setminus F_{0}$ are also disjoint, and we get property 4. The proof is complete.
\end{proof}

Let $Z:=\{z^\lambda_{k}\mid \lambda\in\Lambda, k \in \bbN\}$. We want to show that $f'$ is strongly homotopic to the identity relatively to the set $F'=F\cup Z$. Note that $F_0 \cup Z$ is a dense subset of $F'$. According to proposition~\ref{prop.finitely-homotopic}, it suffices to check that $f'$ is strongly homotopic to the identity relatively to any finite subset of $F_0 \cup Z$. Let  $F_{1}$ and  $Z_{1}$ be finite subsets of $F_0$ and $Z$ respectively. We denote by $z_1=z_{k_1}^{\lambda_1},\dots,z_r=z_{k_r}^{\lambda_r}$ the points of $Z_1$, and we consider the corresponding simple closed curves $c_1=c_{k_1}^{\lambda_1},\dots,c_r=c_{k_r}^{\lambda_r}$ provided by Claim~\ref{claim.exist-curves}.  According to properties 3 and 4, up to increasing $F_{1}$, we may assume that:
\begin{itemize}
\item[a.] the curves $c_1,\dots,c_r$ are essential and pairwise non homotopic in $S-F_{1}$, 
\item[b.] the fillings of $c_1\cup f'(c_1),\dots,c_r\cup f'(c_r)$ in $S-F_1$ are pairwise disjoint. 
\end{itemize}

\begin{claim}\label{c.redresser-courbes}
The homeomorphism $f'$ is isotopic relatively to $F_{1} \cup Z_{1}$ to a homeomorphism $f''$ which fixes the curves $c_{0}, ... , c_{r}$ pointwise.
\end{claim}

\begin{proof}[Proof of Claim~\ref{c.redresser-courbes}]
We work in the surface $M_1 = S \setminus F_{1}$. We will use Proposition~\ref{prop.redresser-les-courbes}. For each $i\in\{1,\dots,r\}$, we consider the closed curves $c_i$ and $c_i':=f'(c_i)$. These curves are homotopic in $M_1$, since $f'$ is strongly homotopic to the identity relatively to $F$ (more precisely, $f'$ is isotopic to $f$ relatively to $F$, and $f$ is strongly homotopic to the identity relatively to $F$). The point $z_i$ is on the curve $c_i$, and its trajectory under the homotopy is contractible (this is point 3 of Claim~\ref{claim.straightening-points}). Proposition~\ref{prop.redresser-les-courbes} provides an isotopy $I_i$ on $M_1$, fixing the points $\{z_i\}$, joining $f'$ to a homeomorphism which fixes $c_i$ pointwise. Moreover, the support of $I_{i}$ in included in an arbitrarily small neighbourhood of the filling of $c_i\cup f'(c_i)$ in $M_1$. Hence, by Property b above, we can choose the isotopies $I_1,\dots,I_n$ with pairwise disjoint supports, and such that the point $z_j$ is outside the support of $I_i$ for $j\neq i$. It follows that the isotopies $I_1,\dots,I_n$ may be combined into an isotopy on $M_1$ relative to $Z_1$ (or equivalently an isotopy on $S$ relative to $F_1\cup Z_1$), from $f'$ to a homeomorphism $f''$ which fixes the curve $c_i$ pointwise for every $i$. 
\end{proof}

\begin{claim}\label{c.homotopie}
The homeomorphism $f''$ is strongly homotopic to the identity relatively to $ F_{1} \cup Z_1 $.
\end{claim}

We propose two alternative proofs of this claim.

\begin{proof}[Proof of claim~\ref{c.homotopie} using hyperbolic geometry]
Let $H$ be a homotopy on $S\setminus F_1$ from the identity to the restriction of $f''$, obtained by concatenating a homotopy from the identity to $f$ which exists by hypothesis, the isotopy from $f$ to $f'$ provided by Claim~\ref{claim.straightening-points}, and the isotopy from $f'$ to $f''$ provided by Claim~\ref{c.redresser-courbes}.
 Observe that the trajectory under $H$ of the points $z_{1}, \dots , z_{r}$ are contractible in $S \setminus F_{0}$: indeed they are fixed under the isotopy from $f''$ to $f'$ given by Claim~\ref{c.redresser-courbes}, and contractible under the concatenation of the isotopies from $f'$ to $f$ and from $f$ to the identity (point 3 of Claim~\ref{claim.straightening-points}). Since the curve $c_{i}$ is fixed by $f''$ and contains the point $z_{i}$ whose trajectory is contractible, the trajectory of every point in $c_i$ is also contractible. 

As noted just after the statement of Lemma~\ref{l.straightening-points}, we may assume that the set $F$ is infinite. Thus we may also assume that $F_{1}$ has at least three points, so that there exists a hyperbolic metric on $S \setminus F_{1}$. Since by Claim~\ref{claim.exist-curves} the simple closed  curves $c_{1},\dots,c_r$ are essential, pairwise disjoint, and pairwise non homotopic in $S \setminus F_{1}$, the hyperbolic metric can be chosen so that these curves are geodesics. For every $z$ in $S \setminus F_{1}$, let $\gamma_z$ be the unique geodesic arc joining  $z$ to $f''(z)$ in $S \setminus F_{1}$ and homotopic to the trajectory of $z$ under $H$. Let $\widehat H$ be the homotopy from the identity to $f''$ on $S\setminus F_1$, obtained by following the geodesic arc $\gamma_z$ for every $z$, as in section~\ref{sub.hyperbolic-geometry}. 

We will prove that $\widehat H$ extends on $S$, and provides a strong homotopy on $S$ relative to $F_1\cup Z_1$. For this purpose, we consider the hyperbolic plane which is a universal cover of $S \setminus F_{1}$, and the lift $\tilde f''$ of $f''$ associated to the homotopy $\widehat H$. Note that $\tilde f''$ fixes pointwise the lifts of the curves $c_{i}$ (indeed, the trajectory of every point of $c_i$ under $H$ is  contractible in $S\setminus F_1$, and thus the same is true for the trajectory under $\widehat H$). In particular, $\tilde f''$ fixes all the lifts of the $z_i$'s. The homeomorphism $f''$ preserves the orientation at every $z_{i}$ (see the proof of Claim~\ref{c.preserves-orientation}). Thus $\tilde f''$ preserves the orientation, and it leaves invariant every connected component of the complement of the union of the lifts of the $c_{i}$'s. Since the curves $c_{i}$ lift to geodesics, the lift of the homotopy $\widehat H$ also leaves these components invariant. Thus, $\widehat H$ is a strong homotopy on $S\setminus F_1$ relative to the union of the curves $c_1,\dots,c_r$. In particular, this is a strong homotopy relative to $Z_1=\{z_1,\dots,z_r\}$. Furthermore, the arguments of section~\ref{sub.hyperbolic-geometry} apply and show that $\widehat H$ is proper. Hence $\widehat H$ extends to a homotopy on $S$, which is of course a strong homotopy relative to  $F_1\cup Z_1$.
\end{proof}

\begin{proof}[Proof of claim~\ref{c.homotopie} using selection techniques]
Just as in the first proof, we consider a homotopy $H$  from the identity to the restriction of $f''$ on $S\setminus F_1$, so that the trajectory under $H$ of any point on the curves $c_1,\dots,c_r$ is contractible. For every $z\in S$, we consider the set of paths $C_z$ from $z$ to $f''(z)$ defined as follows: 
\begin{itemize}
\item if $z$ belongs to $F_{1}$ or to $c_1\cup\dots\cup c_r$, then $C_z$ is reduced to the constant trajectory $\{z\}$,
\item otherwise, $C_z$ is the set of all paths from $z$ to $f''(z)$, homotopic to the trajectory of $z$ under $H$, and disjoint from the curves $c_1,\dots,c_r$. 
\end{itemize}
We want to prove that the map $z\mapsto C_z$ satisfies the hypotheses of Proposition~\ref{prop.selection}. For this purpose, we consider the universal cover $X$ of $S\setminus F_1$, 
the lift $\tilde f''$ of $f''$ obtained by lifting $\widetilde H$, and  the union $\widetilde \Gamma \subset X$ of all the lifts of the curves $c_1,\dots,c_r$. Since these curves are essential in $S\setminus F_1$, their lifts are topological lines, and the connected components of $ X \setminus \widetilde \Gamma$ are planes.
Furthermore $\tilde f''$ fixes $\widetilde \Gamma$ pointwise, and the same arguments as in the proof of Claim~\ref{c.preserves-orientation} show that $f''$ preserves the orientation, and thus it leaves invariant each connected component of $X\setminus \widetilde \Gamma$. Observe that, for $z\notin c_1\cup\dots\cup c_r$, for any lift $\tilde z$ of $z$, the paths in $C_z$ are exactly the projections of the paths from $\tilde z$ to $\tilde f''(\tilde z)$ in $X\setminus \widetilde \Gamma$. Since  $\tilde f''$ preserves every connected component of $X\setminus \widetilde \Gamma$, and since these connected components are simply connected, one deduces that $C_z$ is non-empty, arcwise connected and simply connected for every $z$, and that the map $z\mapsto C_z$ is locally trivial on $S\setminus (F_1 \cup c_1\cup\dots\cup c_r)$,  lower semi-continuous, equi-locally arcwise connected and equi-locally arcwise simply connected at every point of $c_1\cup\dots\cup c_r$. Furthermore, it is clearly  lower semi-continuous, equi-locally arcwise connected and equi-locally arcwise simply connected at every point of $F_{1}$.
Proposition~\ref{prop.selection} provides us with a homotopy $\widehat H$ on $S$ from the identity to $f''$. By definition of the sets $C_z$'s, this homotopy is a strong homotopy relative to 
$F_1 \cup c_1\cup\dots\cup c_r$, hence also relative to $F_{1} \cup Z_1$.
\end{proof}

According to claim~\ref{c.homotopie}, $f''$ is strongly homotopic to the identity relatively to $F_{1} \cup Z_{1}$. Therefore the same is true for $f'$ by Claim~\ref{c.redresser-courbes}. 

We have checked that $f'$ is strongly homotopic to the identity relatively to every finite subset of $F_0\cup Z$, which is a dense subset of $F'$. We may apply Proposition~\ref{prop.finitely-homotopic} which tells us that $f'$ is strongly homotopic to the identity relatively to $F'$. This completes the proof of Lemma~\ref{l.straightening-points}.
\end{proof}

\subsection{Bringing back neighborhoods of isolated points}
\label{sub.bringing-neighborhoods}

Let $f'$ and $F'$ be given by Lemma~\ref{l.straightening-points}: in particular $f'$ is strongly homotopic to the identity relatively to $F'$. The next lemma allows to ``bring back'' neighborhoods of isolated points of $F'$.

\begin{lemm}\label{l.bring-back-neighborhoods}
There exists a closed set $F''$ of $S$ containing $F'$, and an element $f''$ of $\homeo(S,F)$ such that
\begin{enumerate}
\item $f''$ is isotopic to $f'$ relatively to $F'$,
\item $f''$ fixes $F''$ pointwise,
\item $f''$ is strongly homotopic to the identity relatively to $F''$,
\item every isolated point of $F'$ is included in the interior of $F''$.
\end{enumerate}
\end{lemm}

\begin{rema}
\label{rema.endpoints-2}
Remember that every endpoint of every connected component $\beta$ of $\alpha \setminus F'$ was an isolated point of $F'$ (remark~\ref{rema.endpoints-1}). Thus these endpoints belong to the interior of the set $F''$ provided by Lemma~\ref{l.bring-back-neighborhoods}. 
\end{rema}

\begin{proof}[Proof of Lemma~\ref{l.bring-back-neighborhoods}]
We fix a metric on the surface $S$. Let $Y$ be the set of isolated points of $F'$. 

For every point $y\in Y$, we choose an open topological disc $U_y$, such that $U_y\cap F'=\{y\}$. We make this choice in such way that all the $U_{y}$'s are pairwise disjoint, and such that, for every $\epsilon>0$, there are only finitely many points $y\in Y$ such that the diameter of $U_y$ is bigger than $\epsilon$. These properties will ensure that all the future constructions can be made independently in each chart $U_{y}$. Now, for every point $y\in Y$, we consider a homeomorphism $g_y$ of $S$, which is supported in $U_{y}$, and which coincides with $(f')^{-1}$ in some smaller neighborhood of $y$. Note that, due to the properties of the $U_y$'s, the homeomorphism 
$$
f'' = \left( \prod_{y \in Y} g_{y} \right) f'
$$
is well-defined, it coincides with $f'$ outside the union of the $U_{y}$'s, and coincides with the identity on some neighborhood of $Y$. For every $y\in Y$, the Alexander trick (Proposition~\ref{p.alexander}) provides an isotopy from the identity to $g_{y}$ supported in $U_{y}$, and fixing $y$. Considering the product of these isotopies (which, again, is well-defined, thanks to the properties of the $U_y$'s), and composing with $f'$, we get an isotopy $I$ joining from $f'$ to $f''$. This isotopy $I$ fixes $F'$. 

It remains to show that, up to composing $f''$ with some Dehn twists, $f''$ is strongly homotopic to the identity relatively to the union of $F'$ and of some neighborhoods of $Y$.  

Denote by $H$ the strong homotopy from the identity to $f'$ provided by Lemma~\ref{l.straightening-points}. The concatenation $H_{1} = H \star I$ is a strong homotopy relative to $F'$ from the identity to $f''$. Since $f''$ is the identity near every isolated point $y$ of $F'$, the trajectory of a point $z$ near $y$ under this homotopy is a loop in the disc $U_{y}$ around $y$. Up to composing $f''$ with a Dehn twist supported in $U_{y}$, we may assume that this loop is nulhomotopic in $U_{y} \setminus \{y\}$. We will now modify $H_{1}$ to get a strong homotopy relative to a larger set $F''$.

For every $y$ in $Y$, we choose a closed topological disc $V_y\subset U_y$ such that $y\in\mathrm{int}(V_{y})$, such that $f''$ coincides with the identity on $V_y$, and such that the trajectories of all the points of $V_{y}$ under $H_{1}$ are included in $U_{y}$. We choose a smaller closed topological disc $W_{y}\subset \mathrm{int}(V_{y})$. It is easy to see that there is a homotopy $H_{2}$ from the identity to $f''$ which equals $H_{1}$ outside $V_{y}$ and which is the constant homotopy on $W_{y}$ for every $y$. Up to replacing $W_y$ by a smaller disc, we may also demand that the trajectories under $H_{2}$ of all the points in the boundary of $V_{y}$ are disjoint from $W_{y}$. 

For the sake of simplicity, let us first consider the case where the homotopy $H$ provided by Lemma~\ref{l.straightening-points} is proper. Observe that, in this case, the homotopy $H_2$ is also proper. Let $y\in Y$. The homotopy $H_2$ fixes every point in $W_y$. Moreover the trajectory under $H_2$ of every point $S\setminus W_y$ is disjoint from $y$ (recall that $H_2$ is a strong homotopy relative to $F'$ and that $Y\subset F'$). By properness, there exists a positive number $\varepsilon_{y}$, such that the distance from the point $y$ to the trajectory of every point $z\in S\setminus W_{y}$ is larger than $\varepsilon_{y}$. Let $D_y$ be the disc centered at $y$ of radius $\varepsilon_y$. By construction $H_2$ is a strong homotopy relative to $D_y$. We define $F''$ to be the union of $F'$ and of all the discs $D_y$'s for $y\in Y$. By construction, every point of $F'$ is in the interior of $F''$, and $H_2$ is a strong homotopy relative to $F''$. This completes the proof in this case where the homotopy $H$ is proper.

When $H$ is not necessarily proper, we modify $H_{2}$ as follows. For $y\in Y$, let $r_{y}: S \setminus \{y\} \to S \setminus \inte(W_{y})$ be the radial retraction, \emph{i. e.}  the map which is the identity outside the interior of $W_{y}$ and which sends (radially) each point in the punctured disc $W_{y} \setminus \{y\}$ to a point on the circle boundary of $W_{y}$. Observe that the $r_y$'s are pairwise commuting. For every $z$ in $S$, define $R_{z}$ to be the identity if $z$ belongs to $\bigcup_{y\in Y} V_{y}$, and the composition of all the $r_y$'s  otherwise. Now let $H_{3}$ be the homotopy defined by
$$
H_{3}(z,t) = R_{z} ( H_{2}(z,t)).
$$
Note that $H_3$ is continuous, since, for every $y\in Y$, the trajectories under $H_{2}$ of all the points in the boundaries of $V_{y}$ are disjoint from all the $W_{y'}$'s. The homotopy $H_{3}$ may still be not proper, but the above argument nevertheless applies to show that it is a strong homotopy relative to a set $F''$ which is the union of $F'$ and of a neighborhood of $Y$, as wanted.
\end{proof}

\subsection{End of the isotopy}
\label{sub.end-isotopy}

Let $f \in \homeo(S,F)$ be a homeomorphism which is strongly homotopic to the identity. Let $\alpha$ be an arc with end-points in $F$ as in Proposition~\ref{l.multicoupure}. We want to construct an isotopy relative to $F$ from $f$ to a homeomorphism that fixes $\alpha$.  Let $F',F'', f', f''$ be given by Lemmas~\ref{l.straightening-points} and~\ref{l.bring-back-neighborhoods}. Let $Z'$ be the set of all the end-points of all the connected components of $\alpha \setminus F'$. Recall that every point of $Z'$ is in the interior of $F''$ (remark~\ref{rema.endpoints-2}). Hence, for $z\in Z'$, one can find a topological closed disk neighborhood $D_{z}$, such that $z$ is in the interior of $D_z$, such that $D_z$ is included in the interior of $F''$, and such that $\alpha \cap D_{z}$ is connected. We choose such a disc $D_z$ for every $z\in Z'$, and we assume that the $D_z$' are pairwise disjoint. We denote by $F'''$ the union of $F'$ with all the $D_{z}$'s. Note that $F'''$ contains $F$, and that $f''$ is isotopic to $f$ relatively to $F$ and strongly homotopic to the identity relatively to $F'''$. Thus up to replacing $f$ by $f''$ and $F$ by $F'''$, we may assume the following two properties.


Let $\cA:= \{\alpha_{0}, \alpha_{1}, \dots, \}$ be the collection of the connected components of $\alpha \setminus F$. This is a countable collection of properly embedded topological lines in $S\setminus F$. For each $i$, let $\alpha'_{i}:=f(\alpha_{i})$. Then (see Figure~\ref{f.properties12}):
\begin{enumerate}
\item[Property 1] The elements of $\cA$ are \emph{isolated} is $S$: there exists a family $\{V_{0}, V_{1}, \dots \}$ of pairwise disjoint open sets of $S$ such that the closure of $\alpha_i$ in $S$ is contained in $V_{i}$.
\item[Property 2] For each $i$, the topological lines $\alpha_{i}$ and $\alpha'_{i}$ are compactly homotopic in $S\setminus F$. 
\end{enumerate}

\begin{figure}[ht]	
\begin{center}
\def\svgwidth{0.8\textwidth}
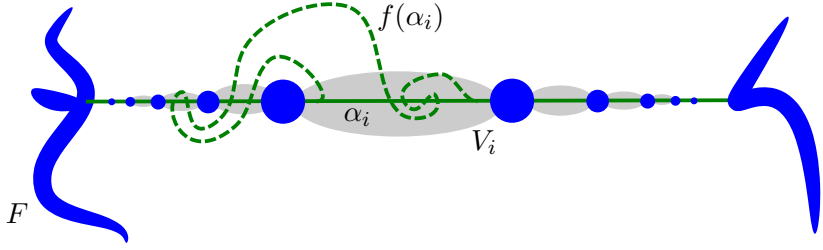
\end{center}
\caption{\label{f.properties12}Properties 1 and 2}
\end{figure}

To complete the proof, it suffices to construct an isotopy relative to $F$, from the identity to a homeomorphism $\varphi$ such that $\varphi \alpha'_{i} = \alpha_{i}$ for each $i$. Indeed, by composing this isotopy with $f$ we get an isotopy from $f$ to a map $g$ that globally fixes $\alpha$. A final isotopy, moving each point of $\alpha$ along $\alpha$, reach a map $f_{1}$ that fixes $\alpha$ pointwise as wanted. 

In Section~\ref{s.isotopy-arc} we have seen how to get an isotopy $I_{i}$ that sends $\alpha'_{i}$ to $\alpha_{i}$ for a given $i$. Furthermore, the sequence of diameters of the $\alpha_{i}$'s tends to zero (for any given metric on $S$). Thus one is tempted to argue that it suffices to successively concatenate the isotopies $I_0,I_1,\dots$. But a given point $x$ close to $\alpha'_{0}$ could be successively pushed  by all these isotopies, and the successive trajectories of $x$ could converge to a point of $F$. Then the infinite concatenation of the isotopies would converge to a non invertible map. Therefore our task is to combine all the isotopies $I_{i}$ in a carefully chosen order, in such a way that their infinite concatenation converges to a homeomorphism of $S$, and thus provides an isotopy. Note that the naive strategy works in the case when the set $\cA$ is finite. Thus in what follows we may consider only the case when $\cA$ is infinite (this slightly simplifies the notations). 

\bigskip

We consider the \emph{incidence graph} $\cG$ defined as follows. The vertices of the graph are the elements of $\cA$. There is an edge in the graph between two given vertices $\alpha_{i}, \alpha_{j}$ if and only if $i \neq j$ and $\alpha'_{i}$ meets $\alpha_{j}$ or $\alpha'_{j}$ meets $\alpha_{i}$.  

This graph is locally finite (this will be a key point in what follows). Indeed, let us fix a distance on $S$.  Property 1 implies that the distance between $\alpha_i$ and the other elements of $\cA$ is bounded away from $0$. Furthermore, for any metric on $S$, the sequence of diameters $(\diam(\alpha_{j}))_{j \geq 0}$ converges to zero. By uniform continuity of $f$ on $\alpha$, the same is true for  the sequence $(\diam(\alpha'_{j}))_{j \geq 0}$. Since $\alpha'_{j}$ meets $\alpha_{j}$, it follows that only a finite number of the $\alpha'_{j}$'s meet $\alpha_{i}$. Likewise and symetrically, only a finite number of the $\alpha_{j}$'s meet $\alpha'_{i}$.

Since the graph $\cG$ is locally finite, one can find a partition $\{\cA_{r} , r\in\bbN\}$ of $\cA$ into finite sets with the following property: for every integers $r,s$, if $\abs{r-s} > 1$ then there is no edge in $\cG$ joining a vertex in $\cA_{r}$ to a vertex in $\cA_{s}$. (For instance if $\cG$ is connected then one can define $\cA_{r}$ to be the sphere of radius $r$ and center $\alpha_{0}$ in $\cG$.) For each $r$, let $F_{r}$ be the union of the elements of $\cA_r$ and their images under $f$: 
$$
F_{r} = \bigcup_{\alpha_i\in \cA_r} \alpha_{i} \cup \alpha'_{i} .
$$

Note that each $F_{r}$ is a finite union of interior of arcs in $S$, in particular it has finitely many connected components, and  that $F_{r}$ does not meet $F_{s}$ when $\abs{r-s} > 1$. In particular, the $F_{2r}$'s are pairwise disjoint; this will play a key role in the sequel. 

\begin{lemm}\label{l.neighborhoods}~
There exists a sequence of sets $(W_{2r})_{r \geq 0}$ with the following properties.
\begin{enumerate}
\item The set $W_{2r}$ is a neighborhood of $F_{2r}$ in $S \setminus F$, and it contains the filling of $F_{2r}$ in $S \setminus F$.
\item The sets $W_{2r}, r \geq 0$ are pairwise disjoint.
\item Let $\delta'_{2r}$ be the supremum of the diameters of the connected components of $W_{2r}$. Then the sequence $(\delta'_{2r})$ tends to zero.
\end{enumerate}

\end{lemm}

\begin{proof}
We introduce the following notations.
Let $\dist$ denote the distance between two subsets of $S$ defined by the formula
$$
\dist(E,E') = \inf\{d(x,x') \mid x \in E, x' \in E'\}.
$$ 
\begin{itemize}
\item Let  $\mu_{2r}$ denotes the minimum of the distances between two connected components of $F_{2r}$.
\item Let $\delta_{2r}$ denotes the maximum of the diameters of the connected components of $F_{2r}$.
\item Let 
$$
\varepsilon_{2r} := \mathrm{dist}\left(F_{2r}, \bigcup_{s \neq r} F_{2s}\right).
$$
\end{itemize}

\begin{claim} 
The sequence $(\delta_{2r})$ tends to zero, and $\mu_{2r}$ and $\varepsilon_{2r}$ are positive.
\end{claim}

\begin{proof}
The positivity of $\mu_{2r}$ follows from Property 1.

\medskip

Let us prove that $\varepsilon_{2r}$ is positive. Let $i_{0} \geq 0$. According to Property 1, there exists a neighbourhood $V_{i_0}$ of the closure of $\alpha_{i_0}$ in $S$ such that the topological line $\alpha_i$'s is disjoint from $V_{i_{0}}$ for $i\neq i_0$. Moreover, each $\alpha'_{i}$ meets $\alpha_{i}$, and the sequence of diameters $(\diam(\alpha'_{i}))$ tends to zero. Thus there exists $\eta_{i_0}>0$ such that all the $\alpha'_{i}$'s that are disjoint from $\alpha_{i_{0}}$ satisfy $\mathrm{dist}(\alpha'_{i}, \alpha_{i_{0}})>\eta_{i_0}$. Symmetrically, there exists $\eta_{i_0}'$  such that all the $\alpha_{i}$'s that are disjoint from $\alpha_{i_{0}}'$ satisfy $\mathrm{dist}(\alpha_{i}, \alpha_{i_{0}}')>\eta_{i_0}'$. By definition of $\varepsilon_{2r}$, $\eta_{i}$
 and $\eta_{i}$, we have 
$$\varepsilon_{2r} \geq \mathop{\min}_{\alpha_i\subset F_{2r}} \min(\eta_i,\eta_i').$$ 
Since $F_{2r}$ contains only finitely many $\alpha_i$'s, it follows that $\varepsilon_{2r}$ is positive.

\medskip

Finally let us prove that $\delta_{2r}$ tends to zero as $r$ goes to infinity. We argue by contradiction. If $\delta_{2r}$ does not tend to zero, there exists $\delta>0$, and a sequence of arcs $(\beta_{n})$ of diameters greater than $\delta$, and such that for each $n$, $\beta_{n}$ is included in 
$$
\bigcup_{i \geq n} \alpha_{i} \cup \alpha'_{i}.
$$
 Up to extracting subsequences, we may assume that $(\beta_{n})$ converges towards a compact connected subset $B$ of $S$ for the Hausdorff distance. Obviously, $B$ is included in the closure of the set $\bigcup_i \alpha_i\cup \alpha_i'$. Actually, since the $\alpha_{i}'$'s are pairwise disjoint and since ithe sequence $(\diam(\alpha_{i}'))$ tends to zero, it follows that $B$ must actually be included in the closure of the set $\bigcup_i \alpha_i$. In particular, it is included in the arc $\alpha$. Since $B$ is connected, it a a subarc of $\alpha$. This subarc is non-trivial, since its diameter must be larger than $\delta$. Moreover, Property~1 provides a neighbourhood $V_i$ of $\alpha_i$ so that $\beta_n$ is disjoined from $V_i$ for every $n\geq i$. It follows that $B$ must be disjoint from $\alpha_i$ for every $i$. As a consequence, $B$ is included in $\alpha\setminus \bigcup_i \alpha_i=F$. This gives the desired contradiction : by construction, every point of $B$ is accumulated by $\bigcup_{i} \alpha_{i} \cup \alpha'_{i}$, but no point of the interior of $F$ in $\alpha$ is accumulated by $\bigcup_{i} \alpha_{i} \cup \alpha'_{i}$.
 \end{proof}

For a fixed value of $r$, let us consider the $c_{2r}$-neighborhood $\hat W_{2r}$of $F_{2r}$, where 
$$
c_{2r} = \frac{1}{2} \min\left(\mu_{2r}, \varepsilon_{2r}, \frac{1}{r}\right).
$$
The set $\hat W_{2r}$ is an open neighborhood of $F_{2r}$. Since $c_{2r}$ is less than one half of $\varepsilon_{2r}$, the sets $\hat W_{2r}$ and $\hat W_{2s}$ are disjoint for $r \neq s$.
Furthermore, since $c_{2r}$ is less than one half of $\mu_{2r}$, every connected component of $\hat W_{2r}$ meet a unique connected component of $F_{2r}$; thus $\hat W_{2r}$ has finitely many connected components, and they have diameter less than
$$
2\mu_{2r} + \delta_{2r}.
$$
Note that this number tends to zero when $r$ tends to $+\infty$. 

\medskip
We define $W_{2r}$ as the filling of $\hat W_{2r}$.
Remember that this is the union of $W_{2r}$ and the connected components of $(S\setminus F)\setminus \hat W_{2r}$ that are relatively compact in $S\setminus F$ (see subsection~\ref{ss.filling}). 

Clearly $W_{2r}$ contains $\hat W_{2r}$ which is open and contains $F_{2r}$. It is filled, thus it contains the filling of $F_{2r}$. The sets $\hat W_{2r}$'s are pairwise disjoint, and furthermore
every connected component of $\hat W_{2r}$ contains a connected component of $F_{2r}$, and thus is proper in $S \setminus F$. By Lemma~\ref{lemm.fillings-disjoint}, the fillings $W_{2r}$ are also pairwise disjoint. It remains to see that the sequence $(\delta'_{2r})$ goes to zero.

Let $\varepsilon>0$. We assume that $\varepsilon$ is smaller than $\varepsilon_{0}$.
The compact set $\alpha$ is included in a compact subsurface of $S$; thus there exists $\eta>0$ such that every subset of $S$ of diameter less than $\eta$ is included in a topological disc of diameter less than $\varepsilon$. Let $r$ be large enough, so that the diameter of every connected component of $\hat W_{2r}$ is less than $\eta$. Let $K_{1}, ..., K_{i_{0}}$ be the connected components of $\hat W_{2r}$. Then each  $K_{i}$ is included in a closed topological disc $D_{i}$ of diameter less than $\varepsilon$. Note that $D_{i}$ is disjoint from the closure of $F_{0}$, because $\varepsilon$ is less than $\varepsilon_{0}$ and $D_{i}$ meets $F_{2r}$.
Furthermore, we can choose $D_{i}$ so that its boundary is arbitrarily close to $K_{i}$ (see Lemma~\ref{lemma.approximation-disc}); then for $i \neq j$ the boundaries of $D_{i}$ and $D_{j}$ are disjoint, and thus the two discs are either disjoint or one contains the other. We conclude that $\hat W_{2r}$ is included in a finite union $O$ of pairwise disjoint discs of diameters less than $\varepsilon$. The set $S \setminus O$ is connected and contains the closure of $F_{0}$, thus it contains a point of $F$. This entails that $O \setminus F$ is filled in $S \setminus F$. Thus $W_{2r}= \fill(\hat W_{2r})$ is included in $O$. This in turn implies that every connected component of $W_{2r}$ has diameter less than $\varepsilon$, and completes the proof of Lemma~\ref{l.neighborhoods}.
\end{proof}

\begin{figure}[ht]	
\begin{center}
\def\svgwidth{\textwidth}
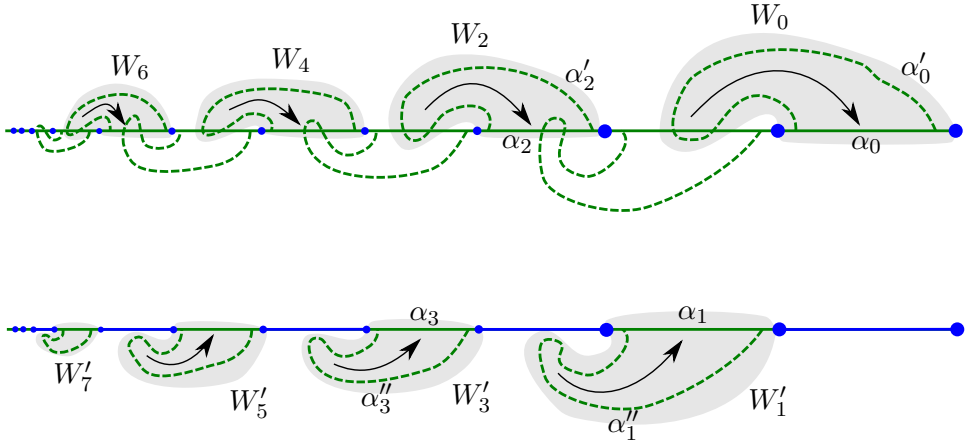
\end{center}
\caption{\label{f.fin-isotopie-A}The isotopy in the easy case when the incidence graph $\cal G$ is linear, and ${\cal A}_{i} = \{\alpha_{i} \}$ for each $i$. Above, the construction of $g_{\infty}$; below, the construction of $h_{\infty}$.}
\end{figure}

We go back to the proof of Proposition~\ref{l.multicoupure} (see Figure~\ref{f.fin-isotopie-A}).
For every $i\in \bbN$ we will denote by $r(i)$ the integer such that $\alpha_{i} \in \cA_{r(i)}$.
Let $\alpha_{i_{0}},\alpha_{i_{1}}, ...$ be a numbering of the elements of all the $\cA_{r}$'s with $r$ even. (We may assume that the sequence $r(i_{0}), r(i_{1}), ...$ is non decreasing, but this is unimportant.) 
Remember that the properly embedded topological lines $\alpha_{i_{0}}$ and $\alpha'_{i_{0}}$ are compactly homotopic in $S \setminus F$. 
Thus we may apply Proposition~\ref{p.straightening-arc-easy}, and we get an isotopy $I_0$ from the identity to a homeomorphism $f_{0}$ such that $f_{0}(\alpha'_{i_{0}}) = \alpha_{i_{0}}$. Since both $\alpha_{i_{0}},\alpha'_{i_{0}}$ are included in $F_{r(i_{0})}$, their filling $\fill(\alpha_{i_{0}}, \alpha'_{i_{0}})$ is included in $W_{r(i_{0})}$, and we may assume that the support of the isotopy $I_0$ is also included in $W_{r(i_{0})}$. Let 
$$
C_{0} = \bigcup_{j \neq i_{0}} \alpha_{j}, \ \ C'_{0} = \bigcup_{j \neq i_{0}} \alpha'_{j}.
$$
Then $C_{0}$ is disjoint from $\alpha_{i_{0}}$ and $C'_{0}$ is disjoint from $\alpha'_{i_{0}}$, and both are closed subsets of $S \setminus F$. According to Proposition~\ref{p.straightening-arc-easy}, we may further assume that $f_{0}(C'_{0}) \cap C_0\subset C_0'\cap C_0$. 
From this we deduce that we have not created any new intersection: more precisely, for every $i,j$, 
$$
f_{0}(\alpha'_{j}) \cap \alpha_{i} \neq \emptyset \Rightarrow \alpha'_{j} \cap \alpha_{i} \neq \emptyset.
$$

The topological lines $\alpha_{i_{1}}$ and $f_{0}(\alpha'_{i_{1}})$ are again compactly homotopic in $S \setminus F$. Furthermore, since the $W_{r}$'s for even values of $r$ are pairwise disjoint, and $f_{0}$ is supported in one of them,  the arc $f_{0}(\alpha'_{i_{1}})$ is still included in $W_{r(i_{1})}$.
We apply Proposition~\ref{p.straightening-arc-easy} a second time to get an isotopy with support included in $W_{r(i_{1})}$, whose time one is a  homeomorphism $f_{1}$ such that $f_{1}f_{0}(\alpha'_{i_{1}}) = \alpha_{i_{1}}$. 
Note that since $f_{0}(\alpha'_{i_{0}}) = \alpha_{i_{0}}$,  this arc is disjoint from both arcs $\alpha_{i_{1}}$ and $f_{0}(\alpha'_{i_{1}})$, and since it meets $F$ it is also disjoint from the set $\fill(\alpha_{i_{1}}, f_{0}(\alpha'_{i_{1}}))$. Thus we can also assume that the support of the isotopy is disjoint from $\alpha_{i_{0}}$, and in particular we get that $f_{1}f_{0}(\alpha'_{i_{0}})$ is still equal to $\alpha_{i_{0}}$.
Again we may assume that this isotopy do not create any new intersection. We proceed inductively to get a sequence of isotopies with time one $f_{k}$ supported in $W_{r(i_{k})}$ such that, for each $\ell \leq k$, 
$$
f_{k} \cdots f_{0} \alpha'_{i_{\ell}} = \alpha_{i_{\ell}},
$$
and which do not create any new intersection.

For each even value of $r$, we concatenate in the obvious order those isotopies which are supported in $W_{r}$ (corresponding to the $k$'s such that $r(i_{k}) = r$). Note that this is a finite concatenation since $\cA_{r}$ is finite. Let $\cI_{r}$ be the resulting isotopy.  
Consider the infinite concatenation $\cI_{0} \star \cI_{1} \star \cdots = (g_{t})_{t \in [0,+\infty)}$ where $g_{t}$ follows $\cI_{r}$ for $t$ in $[r,r+1]$. The isotopies $\cI_{r}$'s have pairwise disjoint supports, and according to Lemma~\ref{l.neighborhoods} the sequence of diameters of the supports is converging to zero. Therefore the limit
$$
g_{\infty} = \lim_{t \to +\infty} g_{t}
$$
exists in $\homeo(S,F)$. Thus we get an isotopy $(g_{t})_{t \in [0,+\infty]}$ such that $g_{\infty}(\alpha'_{i}) = \alpha_{i}$ for every $i$ such that $\alpha_{i}$ belongs to $\cA_{2r}$ for some $r$. 

Denote $\alpha''_{i} = g_{\infty}(\alpha'_{i})$.
Since no new intersection has been created, the sets
$$
F'_{r} = \bigcup \{\alpha_{i} \cup \alpha''_{i} , i \in \cA_{r} \} 
$$
for odd values of $r$ are still pairwise disjoint. Furthermore since $\alpha''_{i} = \alpha_{i}$ for all $i$ in $\cA_{r}$ when $r$ is even, the sets $F'_{r}$ are now pairwise disjoint for every (odd and even) values of $r$. Note that the families $\cA$ and $\cA'' := \{\alpha''_{0}, \dots, \}$ still satisfies the properties 1 and 2 introduced ate the beginning of this subsection.

As in Lemma~\ref{l.neighborhoods}, let us choose for each odd value of $r$ a neighborhood $W_{r}'$ of the set $\fill(F'_{r})$, in such a way  that the sequence of diameters of the $W_{r}'$'s tends to zero, the $W_{r}'$ are pairwise disjoint and disjoint from $F$ and from all the $F'_{s}$'s for even values of $s$. Exactly as before we get an isotopy $(h_{t})_{t \in [0,+\infty]}$ relative to 
$$
F \cup \bigcup \{F'_{r}, r  \mbox{ is even}   \}
$$ 
such that $h_{\infty} \alpha''_{i} = \alpha_{i}$ for every $i$. 
The composition $\Phi = h_{\infty} g_{\infty}$ sends each $\alpha'_{i}$ to $\alpha_{i}$, as wanted.
The proof of Proposition~\ref{l.multicoupure} is now complete.


\section{Straightening a homeomorphism}
\label{sec.homo-implique-iso}

The purpose of this section is to prove that Property (S2) implies Property (S1), that is:

\begin{prop}
\label{p.strong-homotopy-implies-strong-isotopy}
Let $S$ be a boundaryless connected surface, $F$ be a non-empty closed subset of $S$, and $f$ be a homeomorphism of $S$ fixing $F$ pointwise. If $f$ is strongly homotopic to the identity relatively to $F$, then  $f$ is isotopic to the identity relatively to $F$.
\end{prop}

Very roughly speaking, the proof of Proposition~\ref{p.strong-homotopy-implies-strong-isotopy} will consist in considering a countable family of arcs $(\alpha_n)_{n>0}$ so that the union $\bigcup_n \alpha_n$ is dense in $S$, and constructing an isotopy $(f_t)_{t\in [0,+\infty]}$ joiging $f_0=f$ to $f_\infty=\mathrm{Id}$, so that $f_t$ fixes the arcs $\alpha_1,\dots,\alpha_n$ pointwise for $t\geq n$. The construction of the isotopy $(f_t)_{t\in [0,+\infty]}$ mainly relies on the Arc Straightening Lemma (Proposition~\ref{l.multicoupure}).

The analogous statement in the case where $F$ is empty is ``(W2) implies (W1)". This is the main result of Epstein's paper~\cite{epstein}. We will treat this case in subsection~\ref{ss.empty-F} for completeness.

\subsection{A key lemma}

The following lemma will allow us to apply Proposition~\ref{l.multicoupure} inductively, and construct an isotopy $(f_t)_{t\in [0,+\infty]}$ so that $f_t$ fixes more and more arcs as $t$ increases.

\begin{lemm}
\label{l.still-homotopic}
Let $\alpha$ be an arc whose end-points belong to $F$. Assume $f$ is strongly homotopic to the identity relatively to $F$, and $f$ fixes $\alpha$ pointwise. Then $f$  is strongly homotopic to the identity relatively to $F \cup \alpha$.
\end{lemm}

\begin{proof}
By assumption $f$ is strongly homotopic to the identity relatively to $F$.
Hence there exists a homotopy $(f_t)_{t\in [0,1]}$ joining the identity to $f$, so that $f_t$ fixes $F$ pointwise for every $t$, and so that $f_t(S\setminus F)\subset S\setminus F$. The following statement is the key point of the proof of Lemma~\ref{l.still-homotopic}:

\begin{sublemm}
\label{sl.loop-trivial}
If there exists a point $x_0$ in $\alpha\setminus F$ such that the loop $(f_{t}(x_0))_{t \in [0,1]}$ is not homotopic to $0$ in $S\setminus F$, then $S$ is the sphere
and $F$ is made of only two points (the ends of $\alpha$). 
\end{sublemm}

We will use the following fact, whose proof is left to the reader: 

\begin{fact}
\label{f.curves-in-annulus}
Let $A$ be an open annulus, and $(\tilde\gamma_s)_{s\in (0,1)}$ be a continuous family of closed curves in $A$ that are not homotopically trivial. Assume that, for every compact set $\widetilde K\subset A$, the curve $\tilde\gamma_s$ is disjoint from $\widetilde K$ for $s$ small enough. Then, for every $\epsilon>0$, the set $\widetilde V_\epsilon:=\bigcup_{s\in (0,\epsilon]}\widetilde\gamma_s$ is a neighbourhood of one of the two ends of $A$.
\end{fact}

\begin{proof}[Proof of Sublemma~\ref{sl.loop-trivial}]
Assume that there is a point $x_0$ in $\alpha\setminus F$ such that the loop $(f_{t}(x_0))_{t \in [0,1]}$ is not homotopic to $0$ in $S\setminus F$. Denote by $\beta$ the closure of the connected component of $\alpha\setminus F$ which contains $x_0$, and choose an injective parametrization of $\beta$ by $[0,1]$ so that $\beta(1/2)=x_0$. Note that $\beta(0),\beta(1)\in F$ and $\beta((0,1))$ is disjoint from $F$.  The loop $\gamma_s:=(f_{t}(\beta(s)))_{t \in [0,1]}$ depends continously on $s$, is disjoint from $F$ for $s\in (0,1)$, is non-homotopic to $0$ in $S\setminus F$ for $s=\frac{1}{2}$ (hence for every $s\in (0,1)$), is reduced to the point $\beta(0)$ for $s=0$ and reduced to the point $\beta(1)$ for $s=1$. 

Denote by $U$ the connected component of $S\setminus F$ which contains $\beta$, by $\widetilde U$ the universal cover of $U$, and by $T$ an automorphism of $\widetilde U$ associated to the loop $\gamma_{1/2}$. We will use the intermediate cover $\pi:\widetilde U/T\to U$. Note that $\widetilde U$ is homeomorphic to the plane (since $F$ contains at least two elements). Furthermore, since $\gamma_{1/2}$ is homotopic in $S$ to the constant curve $\gamma_{0}$, it cannot be homotopic to the core of a M\"obius band, and therefore $\widetilde U/T$ is homeomorphic to the open annulus. By definition of $T$, the loop $\gamma_{1/2}$ lifts to an essential loop $\widetilde\gamma_{1/2}$ in $\widetilde U/T$. The loop $\gamma_s$ depends continously on $s$ and is disjoint from $F$ for $s\in (0,1)$. Hence, we can lift $(\gamma_s)_{s\in (0,1)}$ to a continuous family $(\tilde\gamma_s)_{s\in (0,1)}$ of essential loops in $\widetilde U/T$. When $s$ goes to $0$, the loop $\gamma_s$ converges towards the point $\beta(0)$; in particular, for every compact set $K$ of $U$, the loop $\gamma_s$ is disjoint from $K$ for $s$ small enough. As a consequence, for every compact set $\widetilde K\subset \widetilde U/T$, the curve $\tilde\gamma_s$ is disjoint from $\widetilde K$ for $s$ small enough. Using Fact~\ref{f.curves-in-annulus}, it follows that,  for every $\epsilon>0$, the set $\widetilde V_\epsilon:=\bigcup_{s\in (0,\epsilon]}\widetilde\gamma_s$ is a neighbourhood of one of the two ends of the annulus $\widetilde U/T$. Since the two ends of $\beta$ are distinct, for small $s,t$, the curves $\tilde \gamma_{s}$ and $\tilde \gamma_{1-t}$ are disjoint.
Thus symmetrically, for every $\epsilon>0$, the set $\widetilde W_\epsilon:=\bigcup_{s\in [1-\epsilon,0)}\widetilde\gamma_s$ is a neighbourhood of the other end
 of the annulus $\widetilde U/T$. The projection $\pi(\widetilde V_\epsilon)=\bigcup_{s\in (0,\epsilon]}\gamma_s$ is contained in an arbitrary small neighbourhood of the point $\beta(0)$ when $\epsilon$ is small enough.
The projection $\pi(\widetilde W_\epsilon)=\bigcup_{s\in [1-\epsilon,1]}\gamma_s$ is contained in an arbitrary small neighbourhood of the point $\beta(1)$ when $\epsilon$ is small enough.
Let $p,q$ be the two ends of the annulus.
The covering map $\pi: \bbS^2\setminus\{p,q\} \simeq \widetilde U/T\to U$ extends to a continuous onto map $\bar\pi:\bbS^2\to U\cup\{\beta(0),\beta(1)\}$ so that $\bar \pi(p)=\beta(0)$ and $\bar\pi(q)=\beta(1)$.

We deduce that $U\cup\{\beta(0),\beta(1)\}$ is compact, and hence closed in $S$. In particular, the boundary of $U$ is reduced to the set $\{\beta(0),\beta(1)\}$.
This in turn implies that $U\cup\{\beta(0),\beta(1)\}$ is open in $S$.
By connectedness we get $S=U\cup\{\beta(0),\beta(1)\}$, $F=\{\beta(0),\beta(1)\}$ and $\beta=\alpha$; in
particular $F$ has only two points.
As a further consequence, the covering map $\pi$ maps a pointed neighborhood of $p$ (resp. $q$) in $\bbS^2$ onto a pointed neighborhood of $\beta(0)$ (resp. $\beta(1)$) in $S$. In particular, the preimage of a point $x\in U$ by $\pi$ is compact , hence finite.
The fundamental group of the surface $U$ is a finite extension of $\mathbb{Z}$, hence is $\mathbb{Z}$. The surface $U$ has two ends, hence is the
annulus.
This implies that $S$ is homeomorphic to the sphere $\bbS^2$.
%
\end{proof}

\medskip

One can now end the proof of Lemma~\ref{l.still-homotopic}.
By assumption, $F$ contains at least two points (the ends of $\alpha$). In the special case where $S$ is the sphere
and $F$ is made of exactly two points, $f$ is isotopic to the identity relatively to $\alpha \cup F = \alpha$ by the Alexander isotopy (Proposition~\ref{p.alexander}).
In the sequel, we exclude this special case. According to Sublemma~\ref{sl.loop-trivial}, the loop $(f_{t}(x))_{t \in [0,1]}$ is null-homotopic in $S\setminus F$ for every point $x$ in $\alpha\setminus F$.
	
We have to prove that $f$ is strongly homotopic to the identity relatively to $F \cup \alpha$. We will make use of the criterions proven in the previous sections. According to Proposition~\ref{prop.finitely-homotopic}, it is enough to check that $f$ is  homotopic to the identity on the complement of every finite subsets of $F \cup \alpha$. So, we consider a finite set $F'\subset F$, and a finite set $E\subset\alpha\setminus F$. We have to prove that the restriction of $f$ to $S\setminus (F'\cup E)$ is homotopic to the identity (in $S\setminus (F'\cup E)$). According to Propositions~\ref{p.W5-W4} and~\ref{p.W4-W3}, it is enough to check that every closed curve $\gamma$ in $S\setminus (F'\cup E)$ is freely homotopic to its image under $f$ in $S\setminus(F'\cup E)$. The bigger $F'$ is, the more difficult it is. We can thus assume that $F'$ contains the ends of $\alpha$. 

We consider the universal cover 
$$\pi:X\to S\setminus F'.$$
Since $F'$ contains the ends of $\alpha$, the set $\pi^{-1}(\alpha\setminus F')$ is a collection of pairwise disjoint proper lines in $X$. 
 As a consequence, every connected component of $X\setminus\pi^{-1}(\alpha\setminus F')$ is simply connected (see Figure~\ref{f.F-union-alpha}). 

\begin{figure}[ht]	
\begin{center}
\def\svgwidth{0.8\textwidth}
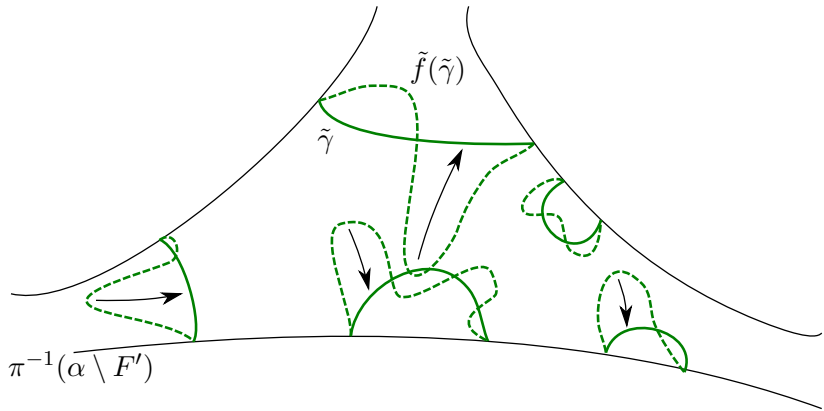
\end{center}
\caption{\label{f.F-union-alpha}Construction of a free homotopy between $\gamma$ and $f(\gamma)$.}
\end{figure}

The restriction to $S\setminus F'$ of the homotopy $(f_t)_{t\in [0,1]}$ can be lifted in $X$. This provides a homotopy $(\tilde f_t)_{t\in [0,1]}$ joining the identity to a lift $\widetilde f$ of $f_{|S\setminus F'}$. Recall that, for every point $x$ in $\alpha\setminus F$, the loop $(f_{t}(x))_{t \in [0,1]}$ is null-homotopic in $S\setminus F$, hence also in $S\setminus F'$. For $x\in \alpha\cap (F\setminus F')$, the loop $(f_{t}(x))_{t \in [0,1]}$ is reduced to a point, hence null-homotopic in $S\setminus F'$. As a consequence, the lift $\widetilde f$ fixes $\pi^{-1}(\alpha\setminus F')$ pointwise. 

Let $\gamma$ be a closed loop in $S\setminus(F'\cup E)$. We have to prove that $\gamma$ is freely homotopic to its image  under $f$ in $S\setminus(F'\cup E)$. 
We lift  $\gamma$ to a curve $\widetilde \gamma$ in $X$. Since every connected component of $X\setminus\pi^{-1}(\alpha\setminus F')$ is simply connected and since $\widetilde f$ fixes $\pi^{-1}(\alpha\setminus F')$ pointwise, the closure of each connected components of $\widetilde\gamma\setminus\pi^{-1}(\alpha\setminus F')$ is homotopic (with fixed ends) to its image under $\widetilde f$ in $X\setminus\pi^{-1}(\alpha\setminus F')$. Projecting in $S\setminus F'$, we obtain a free homotopy between $\gamma$ and $f(\gamma)$ in $S\setminus(F'\cup E)$, as wanted.
This completes the proof of Lemma~\ref{l.still-homotopic}.
\end{proof}

\subsection{The case of the closed disc}

We start by proving a version of Proposition~\ref{p.strong-homotopy-implies-strong-isotopy} in the closed disc (strictly speaking, this will not be a particular case of Proposition~\ref{p.strong-homotopy-implies-strong-isotopy}, since the closed unit disc is not a boundaryless surface). More precisely, we want to prove the following statement: 

\begin{prop}
\label{p.disc-case}
Denote by $D$ be the closed unit disc in $\bbR^2$. Let $f$ be a homeomorphism of $D$ fixing $\partial D$ pointwise, and $F$ be a closed subset of $D$ containing $\partial D$. If $f$ is strongly homotopic to the identity relatively to $F$, then it is isotopic to the identity relatively to $F$.
\end{prop}
We refer to Figure~\ref{f.strategy} for a geometric picture of the isotopy constructed in the proof.
\begin{proof}
We want to use the results of the previous sections, which deal with boundaryless surface. For this reason, we consider an open neighbourhood $\widehat D$ of the closed disc $D$ in the plane, the closed set $\widehat F:=F\cup (\widehat D\setminus D)$, and the homeomorphism $\widehat f:\widehat D\to\widehat D$ which coincides with $f$ on $D$ and is the identity outside $D$. 

We fix a sequence of segments $(\alpha_n)_{n\geq 0}$, each of them contained in $D$ and joining two points of $\partial D$, and cutting the disc $D$ into smaller and smaller pieces: more precisely, we demand that the supremum of the diameters of the connected components of $D\setminus (\alpha_{0} \cup \cdots \cup \alpha_{n})$ goes to zero 
when $n$ goes to infinity.

Let $\widehat f_{0}=\widehat f$ and $\widehat F_{0}=\widehat F$. We apply the ``arc straightening lemma'' (Proposition~\ref{l.multicoupure}) to $\widehat f_{0}$, $\widehat F_{0}$ and the segment $\alpha_0$. We get an isotopy $\widehat I_{0} = (\widehat f_{t})_{t \in [0,1]}$ from $\widehat f_{0}$ to a homeomorphism $\widehat f_{1}$ which fixes $\widehat F_1:=\widehat F_0\cup\alpha_0$ pointwise. According to Lemma~\ref{l.still-homotopic}, the homeomorphism $\widehat f_{1}$ is strongly homotopic to the identity relatively to $\widehat F_{1}$. Now we  apply Proposition~\ref{l.multicoupure} to $\widehat f_{1}$, $\widehat F_{1}$, and the segment $\alpha_{1}$. We get a second isotopy $\widehat I_{1} = (\widehat f_{t})_{t \in [1,2]}$ from $\widehat f_{1}$ to a homeomorphism $\widehat f_{2}$ which fixes $\widehat F_2:=\widehat F_1\cup\alpha_1$ pointwise. Repeating this construction endlessly, we obtain an isotopy 
$$\widehat I = \widehat I_{0} \star \widehat I_{1} \star \cdots = (\widehat f_{t})_{t \in [0,+\infty)}$$  so that $\widehat f_t$ fixes $\widehat F\cup\alpha_0\cup\dots\alpha_n$ pointwise for every $t\geq n$. The condition on the connected components of the complement of the $\alpha_{n}$'s implies that the map $(\widehat f_{t})$ uniformly converges to the identity when $t$ tends to infinity. Hence, we get an isotopy $(\widehat f_t)_{t\in [0,+\infty]}$ on $\widehat D$, from $\widehat f$ to the identity, relatively to $\widehat F$. Considering the restriction of this isotopy, we obtain an isotopy $(f_t)_{t\in [0,+\infty]}$ on $D$,  from $f$ to the identity, relatively to $F$.
\end{proof}

\subsection{The case where $F$ contains at least two different points}
\label{ss.two-different-points}

We will now prove Proposition~\ref{p.strong-homotopy-implies-strong-isotopy} in the case where $F$ contains at least two points. In this case, we can find an arc $\alpha_{0}$ in $S$ whose ends are in $F$. Mreover, since $S$ is $\sigma$-compact, we can find a sequence of arcs $(\alpha_{n})_{n>0}$ in $S$ such that
\begin{itemize}
\item the ends of $\alpha_{n+1}$ belong to the set $F_{n} = F \cup \alpha_{0} \cup \cdots \cup \alpha_{n}$,
\item the family  $(\alpha_{n})_{n\geq 0}$ is locally finite, and the $\alpha_{n}$'s are pairwise topologically transverse,
\item the connected components of $S \setminus \bigcup_{n\geq 0} \alpha_{n}$ are topological open disks $(D_p)_{p\geq 0}$,
\item the closure of these disks are topological closed disks. 
\end{itemize}
Note that the second item implies that the family of closed discs $(\mathrm{Clos}(D_{p}))_{p\geq 0}$ is locally finite.

Applying Proposition~\ref{l.multicoupure} and Lemma~\ref{l.still-homotopic} inductively as in the case of the disc, we obtain an isotopy $I = I_{0} \star I_{1} \star \cdots = (f_{t})_{t \in [0,+\infty)}$ so that $f_t$ fixes $F\cup\alpha_0\cup\dots\alpha_n$ pointwise for every $t\geq n$. Note that, for a fixed integer $p_0$, there exists an integer $n_{0}$ such that, for $n\geq n_0$, the arc $\alpha_n$ is disjoint from the closures of the discs $D_0,\dots,D_{p_0}$. In particular, the homeomorphism $f_t$ fixes $\partial D_0\cup \dots\cup\partial D_{p_0}$ pointwise for every $t\geq n_0$. Thus we may modify the isotopy $I =(f_{t})_{t\in [0,+\infty)}$ into an isotopy $I'=(f'_{t})_{t \in [0,+\infty)}$ so that all the $f'_t$'s, $t\geq n_0$, coincide
on $D_0\cup \dots\cup D_{p_0}$. This ensures that $f'_t$ converges towards a homeomorphism $f_{\infty}'$ when $t$ goes to $\infty$. 

On each disk $D_{p}$, the restriction of $f_{\infty}'$ is strongly homotopic to the identity relatively to $F \cap D_{p}$: indeed on $D_{p}$ the homeomorphism $f_{\infty}'$ coincides with $f_{n}'$ for large enough $n$, on $S$ the homeomorphism $f_{n}'$ is strongly homotopic to the identity relatively to $F\cup\alpha_0\cup\dots\alpha_n$, and since this set contains the boundary of $D_{p}$, the strong homotopy restricts to a strong homotopy in $D_{p}$.

It remains to find an isotopy relative to $F_{\infty} = F \cup (\bigcup_p\partial D_{p})$, joining $f'_{\infty}$ to the identity. This is done successively on each disc $D_{p}$, by applying Proposition~\ref{p.disc-case}. More formally, for each $p$, Proposition~\ref{p.disc-case} provides us with an isotopy on $D_p$, joining the restriction of $f$ to the identity, relatively to $F\cap D_p$. We glue the isotopies corresponding to the different discs $D_p$ together. The local finiteness of the family $(D_{p})_{p\geq 0}$ ensures that this defines an isotopy on $S$. 
This completes the proof of Proposition~\ref{p.strong-homotopy-implies-strong-isotopy} in the case where $F$ contains at least two different points.

\subsection{The case where $F$ consists in a single point}
\label{s.epstein}

We end up the proof of Proposition~\ref{p.strong-homotopy-implies-strong-isotopy} by treating the case where the set $F$ is made of a single point $x_{0}$. 

\bigskip

By assumption, there exists a strong homotopy $H : S \times [0,1] \to S$ relative to $x_0$ joining the identity to $f$. 

Let $U$ be a closed disk neighborhood of $x_{0}$, and $D$ be another closed disk neighborhood of $x_{0}$ so that $D\cup f(D)\subset \mathrm{int}(U)$. Any orientation in $U$ induces an orientation of the curves $\partial D$ and $\partial f(D)$. Assuming that $D$ is small enough, the strong homotopy $H$ gives homotopy between the oriented closed curves $\partial D$ and $f(\partial D)$ in $U\setminus\{x_0\}$. Hence, $f$ sends the oriented curve $\partial D$ on the oriented curve $\partial f(D)$. Now, there exists a homeomorphism $\Phi$ between the closed annuli $U \setminus \mathrm{int}(D)$ and $U \setminus \mathrm{int}(f(D))$, and since $f$ sends the oriented curve $\partial D$ on the oriented curve $\partial f(D)$, it is easy to modify $\Phi$ so that it coincides with $f$ on $\partial D$ and with the identity on the boundary of $U$. We extend $\Phi$ by the identity outside $U$ and by $f$ on $D$. By Alexander's trick (Proposition~\ref{p.alexander}), the map $\Phi$ is isotopic to the identity on $U$, relatively to $x_{0}$.  Up to composing $f$ with $\Phi^{-1}$ and composing the strong homotopy $H$ with an isotopy from the identity to $\Phi^{-1}$ relative to $x_0$, we may assume that $f$ is the identity on $D$. Now, if $f$ is the identity on $D$, the trajectory under $H$ of $z\in D\in\{x_0\}$ is a closed curve. This curve is included in $D\setminus \{x_0\}$ provided that $z$ is close enough to $x_0$. Up to composing $f$ and $H$ with a twist around $x_{0}$, we may assume that this curve is homotopic to a constant curve in $D \setminus \{x_{0}\}$.

From now on, we assume that $f$ is the identity on a disk neighbourhood $D$ of $x_0$, and that the trajectory under $H$ of any point $z\in D\setminus \{x_0\}$ close to $x_0$ is homotopic to a constant curve in $D \setminus \{x_{0}\}$.

\begin{lemm}
\label{l.from-point-to-disk}
In this situation, the homeomorphism $f$ is strongly homotopic to the identity relatively to the disk $D$.
\end{lemm}

\begin{proof}
Let $z$ be a point of $S \setminus D$, we define  a homotopy class of curve from $z$ to $f(z)$ in $S \setminus D$ as follows. Let $V$ be a disk neighborhood of $D$ that does not contain $z$, and choose a homeomorphism $\Psi$ between $S \setminus \{x_{0}\}$ and  $S \setminus D$ that is supported in $V$. Now $C_{z}$ is defined as the homotopy class of the image under $\Psi$ of the trajectory of $z$ under the strong homotopy $H$. Note that $C_{z}$ does not depend on the choice of $\Psi$, and in particular $C_{z}$ depends continuously on $z$ in the sense of Claim~\ref{claim.pi-1-Cz}. For $z \in D$, define $C_{z}$ as the constant curve at $z$. We leave it to the reader to check that the map $z \mapsto C_{z}$ satisfies the hypotheses of Proposition~\ref{prop.selection}, with $F=D$. Note in particular that the semi-continuity at points of $\partial D$ follows from the fact that the trajectories of points of $D$ under $H$ are homotopic to constants in $D \setminus \{x_{0}\}$. The conclusion of Proposition~\ref{prop.selection} yields a strong homotopy from the identity to $f$ relatively to $D$, as wanted.
\end{proof}

Lemma~\ref{l.from-point-to-disk} allows to apply the result of subsection~\ref{ss.two-different-points} with $F=D$ (which contains more than two points). Thus we get an isotopy from $f$ to the identity relative to $D$. Since $D$ contains $x_{0}$, this gives an isotopy relative to $\{x_{0}\}$, as wanted.

%
%
%
%
%

\subsection{The case where $F$ is empty}
\label{ss.empty-F}

As already explained at the beginning of the section, the analog of Proposition~\ref{p.strong-homotopy-implies-strong-isotopy} in the case where $F$ is empty is the fact that Property (W2) implies Property (W1). The last statement is the main result of Esptein's classical paper~\cite{epstein}, but we include a sketch of proof below for sake of completness. 

Let $S$ be a boundaryless connected surface (not necessarily compact, nor orientable), and $f$ be a homeomorphism of $S$. We assume that $f$ is properly homotopic to the identity, and we aim to prove that $f$ is actually isotopic to the identity.

\paragraph{\it First case: $S$ has an isolated planar end $z$.}
In this case, one can extend the homeomorphism $f$ as an homeomorphism $\bar f$ of the surface $S\cup \{z\}$, fixing $z$. The proper homotopy extends as a proper homotopy on $S\cup \{z\}$, fixing $z$, between $\bar f$ and the identity. We are thus reduced to the case $F$ has a single point $z$. Proposition~\ref{p.strong-homotopy-implies-strong-isotopy} applies in this case, and provides us with an isotopy joining $f$ to the identity.

\bigskip

\emph{In the sequel, we assume that $S$ has no isolated planar end. In particular, $S$ is not the plane, the annulus, nor the M\"obius band.}

\paragraph{\it Second case: $S$ contains an essential $2$-sided simple closed curve $c$.}
Let $H$ be a homotopy from the identity to $f$.
Up to composing $H$ by an isotopy, one can assume that the point $c(0)$ is fixed by $f$,
and that its trajectory under $H$ is nulhomotopic.
Proposition~\ref{prop.redresser-les-courbes} applies:
there exists an isotopy $I$ from the identity to a homeomorphism $g$ such that
$g \circ f \circ c=c$ and such that $c(0)$ is fixed under $I$.
Up to composing $H$ by the isotopy $I$, one can assume that
$f \circ c=c$ and that the trajectory of $c(0)$ under $H$ is still nulhomotopic.

Let us consider the universal cover $X$ of $S$ and a lift
$\tilde c$ of $c$. We also consider the lift $\tilde f$ of $f$ associated to the homotopy $H$.
Since $f \circ c=c$ and the trajectory of a point of $c$ is contractible,
each lift of $c$ is pointwise fixed by $\tilde f$.
We claim that the two connected components of $X\setminus \tilde c$
are fixed by $\tilde f$. Otherwise, using that any other lift of $c$ is fixed,
one deduces that $\tilde c$ is the unique lift of $c$. It follows that the automorphism
group of the covering is either trivial or $\mathbb{Z}$ and $S$ is either the plane, the annulus
or the M\"obius band. These cases have been excluded.
Arguing as in the proof of Claim~\ref{c.homotopie}, one deduces that $f$ is strongly homotopic to the
identity relatively to $c$. Proposition~\ref{p.strong-homotopy-implies-strong-isotopy} allows to conclude that $f$ is isotopic to the identity.

\paragraph{\it Third case: $S$ does not contain an essential $2$-sided simple closed curve.}
From Corollary~\ref{c.essential-2-sided}, $S$ is the plane, the sphere or the projective plane.
The first case has already been excluded. We solve the two other cases using the following propositions.

\begin{prop}\label{p.sphere}
Let us consider a homeomorphism $f$ of the sphere. The three properties are equivalent:
\begin{itemize}
\item[(i)] $f$ preserves the orientation,
\item[(ii)] $f$ is homotopic to the identity,
\item[(iii)] $f$ is isotopic to the identity.
\end{itemize}
\end{prop}

\begin{proof}
Up to deforming $f$ by an isotopy, one can assume that
it has (at least) two fixed points $p,q$.

If $f$ acts trivially on the fundamental group of the annulus
$A:=S\setminus\{p,q\}$, then from Proposition~\ref{p.W4-W3}, the restriction $f|A$ is
homotopic to the identity; moreover it preserves each end. By Proposition~\ref{p.homotopic-implies-proper},
one deduces that $f|A$ is properly homotopic to the identity, hence $f$ is homotopic to the identity relatively to $\{p,q\}$.
One can then apply  Proposition~\ref{p.strong-homotopy-implies-strong-isotopy}: $f$ is isotopic to the identity.

In the other cases $f$ acts as $-\id$ on the fundamental group of $A$.
Let $h$ be a symmetry of the annulus which preserves the ends:
$f\circ h$ acts trivially on the fundamental group of the annulus, hence is isotopic to the
identity. We have thus shown that $f$ is isotopic to $h$. In particular it reverses the orientation.
Note also that it has degree $-1$, hence cannot be homotopic to the identity.
\end{proof}

\begin{prop}\label{p.projective-plane}
Any homeomorphism $f$ of the projective plane is isotopic to the identity.
\end{prop}
\begin{proof}
By Poincar\'e-Lefschetz Theorem, $f$ has at least one fixed point $p$.

If $f$ acts trivially on the fundamental group of the M\"obius band
$B:=\mathbb{R} P^2\setminus \{p\}$, then from Propositions~\ref{p.W4-W3}
and~\ref{p.homotopic-implies-proper}, the restriction $f|B$ is properly homotopic to the identity. Hence, $f$ is strongly homotopic to the identity relatively to $\{p\}$. One can then apply Proposition~\ref{p.strong-homotopy-implies-strong-isotopy}: $f$ is isotopic to the identity.

In the other cases $f$ acts as $-\id$ on the fundamental group of $B$. One considers the universal cover $\pi\colon \mathbb{S}^2\to \mathbb{R} P^2$
and the two preimages $q_1,q_2$ of $p$ (which are antipodal points). Let $S$ be the symmetry of the sphere $\mathbb{S}^2$ with respect to a plane through $\{q_1,q_2\}$. Since $S$ commutes with the antipodal map, it defines a homeomorphism $h$ of $\mathbb{R} P^2$ which acts as $-\id$ on the fundamental groups of $B$ and of $\mathbb{R} P^2$.
The homeomorphism $f\circ h$ acts trivially on the fundamental group of $B$. Therefore, $f\circ h$ is isotopic to the identity (see the beginning of the proof). Equivalently,  $f$ is isotopic to $h$. Now, consider the composition of the symetry $S$ with the antipodal map of $\mathbb{S}^2$. This is another lift of the homeomorphism $h$. Observe that $R$  is a rotation of $\mathbb{S}^2$. The ``obvious isotopy" (made of rotations) from the identity to $R$ commutes with the antipody, and therefore defines an isotopy of homeomorphisms of $\mathbb{R} P^2$ from the identity to $h$. Hence, $h$ and $f$ are isotopic to the identity.
\end{proof}


\section{Applications to unlinked continua}
\label{sec.unlinked-continua}

The purpose of this section is to explore the consequences of Theorem~\ref{theo.finitely-isotopic} in the particular case where the closed set $F$ is connected. We begin by the case where the underlying surface $S$ is the plane or the sphere.

\bigskip

Given a surface $S$ and a closed subset $F$ of $S$, we denote by $\homeo^+(S, F)$ the set of the homeomorphisms of $S$ that preserve the orientation and fix $F$ pointwise. This set is endowed with the topology of the uniform convergence on compact subsets of $S$. 

\begin{coro}
\label{c.plane-or-sphere}
Let $S$ be either the plane or the two-dimensional sphere, and $F$ be a closed connected subset of $S$. Then  every homeomorphism $f\in \homeo^+(S, F)$ is isotopic to the identity relatively to $F$. Equivalently, the space $\homeo^+(S, F)$ is arcwise connected. 
\end{coro}

\begin{rema}
The connectedness of $F$ is crucial. Actually, elaborating on the proof of Corollary~\ref{c.plane-or-sphere} presented below, one can prove the following more precise statements:
\begin{itemize}
\item The space $\homeo^+(\bbR^2, F)$ is arcwise connected if and only if $F$ is empty, or $F$ is connected, or $F$ consists in exactly two points.
\item The space $\homeo^+(\bbS^2, F)$ is arcwise connected if and only if $F$ is empty, or $F$ is connected, or $F$ has two connected components at least one of which is a singleton, or $F$ has three connected components at least two of which are singletons
\footnote{To go further, assume $S$ is the sphere, $F$ a closed connected subset with a finite number $m_{1}$ of connected components, no less than 2. Then the group of connected components of $\homeo^+(\bbS^2, F)$ is isomorphic to the product of the group of the sphere pure braids with $m_{1}$ strands by $\bbZ^{m_{2}}$, where $m_{2}$ is the number of non trivial connected components (each factor $\bbZ$ corresponds to Dehn twists around a non trivial component).}.
\end{itemize}
\end{rema}

\bigskip

Let us discuss the locally connected \emph{vs} non locally connected case. For simplicity we work on the sphere.
Let $F$ be a closed connected subset of the sphere, and $f\in \homeo^+(S, F)$. We want to construct an isotopy from $f$ to the identity relative to $F$.
Assume first that the complement of $F$ is connected, and denote it by $O$: this is a simply connected open set, hence there is a Riemann conformal one-to one mapping  $\Phi: \bbD^2 \to O$. It is well know that the map $\Phi^{-1} f \Phi$ extends to a homeomorphism $\bar f$ of the closed disk which is the identity on the boundary (this is a consequence of Caratheodory's theory of \emph{prime ends}). Alexander's trick provides an isotopy $(\bar f_{t})$ in the closed disk, from $\bar f$ to the identity (see~\ref{sub.alexander}). For a fixed $t \in [0,1]$, we define
$f_{t} = \Phi \bar f_{t} \Phi^{-1}$ on $O$, and extend $f_{t}$ by the identity on $F$. This defines a family of homeomorphism of the two-sphere, such that $f_{0} = f$ and $f_{1}= \mathrm{Id}$.

Assume now that  $F$ is locally connected. Then by Caratheodory's theorem, the map $\Phi$ extends to a continuous map from the closed disk to the closure of $O$.
Using the uniform continuity of this extension, it is not difficult to prove that the map $t \mapsto f_{t}$ is continuous, and in this case $(f_{t})$ is an isotopy from $f$ to the identity, as required by Corollary~\ref{c.plane-or-sphere}. Furthermore, this construction is still valid  when the complement of $F$ has several connected components: indeed the local connectedness entails that for every $\varepsilon>0$, all connected components but a finite number have diameter less than $\varepsilon$, thus a global isotopy may be constructed by applying independently Alexander's trick  on each component.

In the case when $F$ is not locally connected, however, this argument fails. 
We suspect that for some $f$ the family $(f_{t})$ given by the Alexander isotopy on the disk is \emph{not} continuous,
 but we have no proof to confirm this suspicion. 
In order to prove  Corollary~\ref{c.plane-or-sphere} in the general case, we will use the much more sophisticated techniques developped in this paper.

\bigskip

We now introduce some tools which will be used to prove Corollary~\ref{c.plane-or-sphere}. Given a metric space $C$, we will denote by $\mathcal{P}(C)$ the space of pairs $[x,y\}$ of distinct points of $C$, endowed with the natural metric.

\begin{lemm}
\label{l.connected}
If $C$ is connected, then so is $\cP(C)$. 
\end{lemm}

\begin{proof}
Given three pairwise distinct points $x,y,z$ of $C$, we say that $z$ separates $x$ from $y$ if $x$ and $y$ do not lie in the same connected component of $C\setminus\{z\}$. We will use the following fact.

\begin{fact}
If $z$ separates $x$ from $y$, then $x$ does not separate $y$ from $z$ and $y$ does not separate $x$ from $z$. 
\end{fact}

\begin{proof}
Denote by $X$ and $Y$ the connected components of $C\setminus\{z\}$ containing $x$ and $y$ respectively. Since $C$ is connected, the point $z$ belongs to the closure of both $X$ and $Y$. Hence, $Y\cup\{z\}$ is a connected subset of $C$. Both the points $y$ and $z$ belong to $Y\cup\{z\}$, but the point $x$ does not belong to $Y\cup\{z\}$. This shows that $x$ does not separate $y$ from $z$. A symmetric argument shows that $y$ does not separate $x$ from $z$. 
\end{proof}

Let us come back to the proof of the lemma. We denote by $\sim$ the equivalence relation ``being in the same connected component" of $\mathcal{P}(C)$. We want to prove that all the elements of $\mathcal{P}(C)$ lie in the same equivalence class. For this purpose, we consider two arbitrary elements $\{x,y\}$ and $\{x',y'\}$ of $\mathcal{P}(C)$. 
If $y$ does not separate $x$ from $x'$, then we immediately get $\{x,y\}\sim \{x',y\}$. Otherwise, the fact above states that $x$ does not separate $x'$ from $y$, and that $x'$ does not separate $x$ from $y$. Since $x$ does not separate $x'$ from $y$, we have $\{x,x'\}\sim \{x,y\}$. And since $x'$ does not separate $x$ from $y$, we have $\{x,x'\}=\{x',x\}\sim \{x',y\}$. By transitivity, we get $\{x,y\}\sim\{x',y\}$. So we have proved that $\{x,y\}\sim\{x',y\}$ in every case. By symmetry, we also have $\{x',y\}\sim \{x',y'\}$. Again by transitivity, this yields $\{x,y\}\sim \{x',y'\}$. Since $\{x,y\}$ and $\{x',y'\}$ are arbitrary elements of $\mathcal{P}(C)$, this shows that $\cP(C)$ is connected and completes the proof Lemma~\ref{l.connected}.
\end{proof}

The main tool to prove Corollary~\ref{c.plane-or-sphere} is the concept of \emph{linking number}. Let $f$ be an orientation-preserving homeomorphism of the plane $\mathbb{R}^2$ and $I=(f_t)_{t\in [0,1]}$ be an isotopy from the identity to a homeomorphism $f$. If $x$ and $y$ are two distinct fixed points of $f$, we define $\mathrm{link}(I,x,y)$ to be the number of turns made by the vector $\overrightarrow{f_t(x)f_t(y)}$ as $t$ ranges from $0$ to $1$ (we insist on the fact that this number is defined only if $x\neq y$). Obviously, we have $\mathrm{link}(I,x,y) = \mathrm{link}(I,y,x)$, and this number depends continuously on the pair $\{x,y\}$. Since it is an integer, it is locally constant on the space $\cP(\fix(f))$.  Using Lemma~\ref{l.connected}, we get the following consequence:

\begin{coro}
\label{c.linking-constant}
If $I$ is an isotopy of planar homeomorphisms joining the identity to a homeorphism $f$, then the linking number $\mathrm{link}(I,x,y)$ is constant on $\cP(C)$ for every connected subset $C$ of $\fix(f)$.
\end{coro}

Now we are ready to complete the proof of Corollary~\ref{c.plane-or-sphere}. 

\begin{proof}[Proof of Corollary~\ref{c.plane-or-sphere}]
We first treat the case where the surface $S$ is the plane. 

According to Theorem~\ref{theo.equivalences-weak}, there exists an isotopy $I$ joining the identity to $f$ in $\homeo^+(\mathbb{R}^2)$. According to Corollary~\ref{c.linking-constant}, since $F$ is connected, the linking number $\mathrm{link}(I,x,y)$ is constant on $\mathcal{P}(F)$. Therefore, composing $I$ by a loop of planar rotations if necessary, we may --- and we will --- assume that  $\mathrm{link}(I,x,y)$ is equal to $0$ for every $\{x,y\}\in\mathcal{P}(F)$.

According to Theorem~\ref{theo.finitely-isotopic}, in order to prove that $F$ is unlinked, it suffices to check that every finite subset $F_{0}$ of $F$ is unlinked, that is, $f$ is isotopic to the identity relatively to $F_{0}$. We proceed inductively on the cardinal of $F_{0}$. If $F_{0}$ is empty there is nothing to prove. If $F_{0}$ has only one element $z$, then the trajectory $I.z$ of $z$ under $I$ is contractible in the plane. By Lemma~\ref{l.fibration}, this implies that there exists an isotopy $I'$ from the identity to $f$ which fixes $z$. In other words, $f$ is isotopic to the identity relatively to $F_0=\{z\}$. Now we assume that $F_{0}$ has at least two elements (see Figure~\ref{f.unlinked-continua}). Let $z_{0} \in F_{0}$ and let $F'_{0} = F_{0} \setminus \{z_{0}\}$. By induction, we may assume that the isotopy $I$ fixes $F'_{0}$. We will prove that the trajectory $Iz_{0}$ of $z_{0}$ is contractible in $\bbR^2 \setminus F'_{0}$. Let $C_0$ be the connected component of $F \setminus F'_{0}$ that contains $z_{0}$. Since $F$ is connected, there exists a point $z_{1} \in F'_{0}$ in the closure of $C_{0}$. Thus the trajectory $I.z_{0}$ is freely homotopic in $\bbR^2 \setminus F'_{0}$ to the trajectory $I.z$ of a point $z \in C_{0}$ arbitrarily close to $z_{1}$. By continuity of $I$, since $z_{1}$ is fixed by $I$, the loop $I.z$ in included in an arbitrarily small neighborhood of $z_{1}$. In particular, we may assume that the trajectory $I.z$ is contained in a disc $D$, centered at $z_1$, and disjoint from the finite set $F_0'\setminus\{z_1\}$. Since the linking number $\mathrm{link}(I,z,z_{1})$ is zero, the trajectory $I.z$ must be contractible in $D\setminus\{z_1\}$. And since $D$ is disjoint from the finite set $F_0'\setminus\{z_1\}$, this implies that $I.z$ is contractible in $\mathbb{R}^2\setminus F_0'$, and so is $I.z_{0}$, as wanted. By Lemma~\ref{l.fibration}, this implies there exists of an isotopy $I'$ from the identity to $f$ which fixes $F_0'\cup\{z_0\}=F_0$. In other words, $f$ is isotopic to the identity relatively to $F_0$. We have proved Corollary~\ref{c.plane-or-sphere} in the case where the surface $S$ is the plane.  

\begin{figure}[ht]	
\begin{center}
\def\svgwidth{0.7\textwidth}
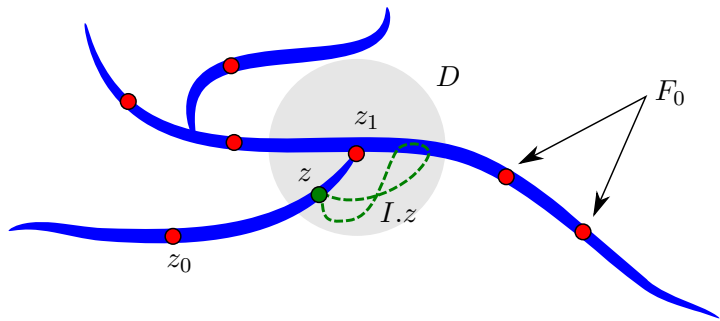
\end{center}
\caption{\label{f.unlinked-continua}The trajectory of $z_{0}$ is nulhomotopic in the complement of $F'_{0}$.}
\end{figure}

\bigskip

Now, we treat the case where the surface $S$ is the sphere $\mathbb{S}^2$. If $F=\bbS^2$, then $f$ is the identity, and there is nothing to prove. Otherwise we consider a point $z \in \mathbb{S}^2\setminus F$. By Brown-Kister Theorem~\ref{t.brown-kister}, the points $z$ and  $f(z)$ belong to the same connected component $O$ of $\bbS^2 \setminus F$. The open set $O$ is arcwise connected, hence we can find an isotopy $J$ supported in $O$ joining $f$ to a homeomorphism $f'$ so that $f'(f(z))=z$. Note that  $J$ is an isotopy relative to $F$, since it is supported in $O$. Now we identify $\bbS^2 \setminus \{z\}$ to the plane. Under this identification, the set $F$ becomes a compact connected subset of $\mathbb{R}^2$, and the restriction of the homeomorphism $f'$ becomes an element of $\homeo^+(\mathbb{R}^2,F)$. The planar case provides an isotopy $I$ from the identity to the restriction of $f'$ in $\homeo^+(\mathbb{R}^2,F)$. The isotopy extends to an isotopy $\bar I$ joining the identity to $f'$ in $\homeo^+(\mathbb{S}^2,F\cup\{z\})$. Concatenating this isotopy $\bar I$ with the reversed isotopy $J$, we obtained an isotopy joining the identity to $f$ in $\homeo^+(\mathbb{S}^2,F)$, and the proof is complete. 
\end{proof}

Now we briefly explain how to generalize Corollary~\ref{c.plane-or-sphere} to orientable surfaces $S$ different from the plane and the sphere. Let $F$ be a compact connected subset of $S$ which is not reduced to a single point.  Denote by $\homeo_{0}(S,F)$ the arcwise connected component of the identity in the space $\homeo(S,F)$, \emph{i.e.} the set of homeomorphisms of $S$ which are isotopic to the identity relatively to $F$. This subgroup may be described as follows. Let $f$ be an element of $\homeo(S,F)$ isotopic to the identity on $S$, and pick any isotopy $I$ from the identity to $f$. If $f \in \homeo_{0}(S,F)$, then we can assume that the trajectory under $I$ of every point $z\in F$ is a contractible loop in $S$: this is automatic when $S$ is neither the annulus nor the torus by Corollary~\ref{coro.uniqueness-homotopy}, and this can be achieved by composing $I$ by an appropriate loop of rigid rotations when $S$ is the annulus or the torus. Now consider the universal cover $\pi:\tilde S\to S$ of the surface $S$, and the lift $\tilde I$ of the isotopy $I$ starting at the identity. The surface $\widetilde S$ is a topological plane. Hence, given two points $\tilde x, \tilde y$ in the same connected component of $\tilde F:=\pi^{-1}(S)$, we can consider the linking number $\mathrm{link}(\widetilde I,\tilde x,\tilde y)$, as defined above. This number does not depend on the choice of the points $\tilde x, \tilde y$. Moreover, one can check that it also does not depend on the isotopy $I$. So this defines a linking number $\mathrm{link}_{S,F}(f)$ associated to $f$. If $f\in \homeo_{0}(S,F)$ then the linking number $\mathrm{link}_{S,F}(f)$ must be equal to zero since there exists an isotopy from the identity to $f$ which fixes $F$ pointwise. Using the same  tools as in the case of the sphere and the plane, one can prove  the converse implication: if the linking number $\mathrm{link}_{S,F}(f)$ is equal to zero, then $f$ is isotopic to the identity relatively to $F$. In other words, $\homeo_{0}(S,F)$ is exactly the subgroup of homeomorphisms $f$ of $\homeo(S,F)$ that are isotopic to the identity, such that  there is an isotopy from the identity to $f$ under which the trajectory of points of $F$ is contractible, and whose linking number $\mathrm{link}_{S,F}(f)$ vanishes.

\newpage

\appendix

\section{Miscellani on surface topology}\label{appe.surface-topology}

As before $S$ is a (non-necessarily orientable) surface with no boundary.
Throughout the text, we use the words sphere, plane, annulus, disk, and so on, to mean a topological space homeomorphic to the standard sphere, plane, annulus, disk, and so on. Here we collect some basic results, mainly about the topology of surfaces.

\subsection{Ends}\label{ss.ends}

Let $\Sigma$ be the end-compactification of $S$, and $F$ be the set of ends.
By Urysohn's metrisation theorem, $\Sigma$ is a compact metrizable space.
Moreover $F$ is a closed subset of $\Sigma$ which is totally discontinuous. Every point $x$ of $F$ has a basis of neighborhoods $U$, each of which is bounded by a simple closed curve $\alpha$ included in $S$ such that $U$ is one of the two connected components of $S \setminus \alpha$. An end is called \emph{planar} if it has a neighborhood in $\Sigma$ homeomorphic to the open disk. See~\cite{richards} for more details.

\subsection{Orientation}\label{ss.orientation}

When $S$ is oriented, any parametrized loop that bounds a disc is positively or negatively oriented as the boundary of the oriented disk.
One says that a homeomorphism $f$ \emph{preserves the orientation} if it sends positively oriented loops to positively oriented loops
(when $S$ is connected this property can be checked on a single loop since the space of loops that bound a disc is connected).
When $M$ is not orientable, this notion still makes sense near a fixed point or near a fixed planar end $x$. Then we say that $f$ \emph{locally preserves the orientation near $x$}.

Any simple closed curve admits arbitrarily small neighborhoods that are homeomorphic
to the annulus or to the M\"obius band. In the first case one says that the curve is \emph{$2$-sided}
and in the second that it is \emph{$1$-sided}.

\subsection{The Schoenflies theorem}\label{subsection.schoenflies}

The original Schoenflies theorem says that for every Jordan curve $\gamma$ in the sphere, every orientation preserving embedding of $\gamma$ in the sphere extends to an orientation preserving homeomorphism. The same statement holds for an arc (an embedding of the segment $[0,1]$). The same statements hold in the plane, where we can moreover require the resulting homeomorphism to be compactly supported.

\begin{theo*}
In the closed unit disk $D$, let $\gamma$ be a topological chord, \emph{i. e.} a simple curve with endpoints on $\partial D$ and interior disjoint from $\partial D$. Then there exists a homeomorphism of $D$, which is the identity on the boundary, and sends $\gamma$ on a (usual) chord of $D$.
\end{theo*}
More generally, given a family of topological chords with pairwise disjoint interior, there exists a homeomorphism of the disk which is the identity on the boundary and sends each topological chord of the family to a chord of $D$. This generalization is proved by applying the theorem iteratively.

\subsection{The Brown-Kister theorem: a local version}

Let us first recall Brown-Kister Theorem.

\begin{theo2}[Brown-Kister~\cite{brownkister}]\label{t.brown-kister}
Consider an orientation-preserving homeomorphism $f$ of an orientable connected surface $S$ and a closed set of fixed point $F$, then
$f$ preserves each connected component of $S \setminus F$.
\end{theo2}

We will need the following local version (which readily implies Theorem~\ref{t.brown-kister}).
\begin{prop}
\label{p.Brown-Kister}
Consider a homeomorphism $f$ of a surface $S$, a closed set $F$ of fixed points of $f$,
a point $x$ in  $F$ and a disk neighborhood $U$ of $x$.

Assume that $f$ locally preserves the orientation at $x$,
and let $V$ be a connected open set such that $V$, $f(V)$ and $f^2(V)$ are included in 
$U$.
Then for any $z\in V\setminus F$ there exists a path $\gamma_z$
between $z$ and $f(z)$ that is contained in $U\setminus F$.
\end{prop}

\begin{proof}[Proof of the proposition] 
The following argument is adapted from \cite{lecalvez06} (beginning of Section~4).
Remember that a \emph{triod} is the union of three simple arcs having a unique commun point, which is an end-point of each of them called the center of the triod. If $T$ is a triod in the oriented plane with center $c$, then the orientation of the plane determines a cyclic order on the three connected components of $T \setminus \{c\}$ (called the legs of the triod). An orientation preserving homeomorphism preserves the cyclic order on the legs of the triods.

Under the hypotheses of the proposition, consider a point $z$ in $V \setminus F$ and three simple arcs $\alpha,\beta,\gamma$ in $V \cup f(V)$ whose pairwise intersection equals $\{z,f(z)\}$. If one of the arcs is disjoint from $F$ then we are done. Thus we assume they all meet $F$.
Then the union of the arcs contains a triod $T$, with center $z$, which meets $F$ only at its three feet (the end-points of its legs).
If $T$ meets its image $f(T)$ at some point which is not one of the three feet, then again we are done, since $T \cup f(T)$ contains an arc disjoint from $F$ and joining $z$ to $f(z)$, and this arc is included in $V \cup f(V) \cup f^2(V)$. If not, then $T$ and $f(T)$ are two triods with commun feets and no other commun points. Then the orientation on the legs $f(\alpha),f(\beta),f(\gamma)$ of $f(T)$ must be opposite to the orientation on the legs $\alpha,\beta,\gamma$ of $T$. This contradicts the fact that $f$ preserves the orientation at $x$.
\end{proof}

\subsection{Alexander isotopy}\label{sub.alexander}

\begin{prop}[Alexander~\cite{alexander}]\label{p.alexander}
For any homeomorphism $f$ of the closed ball which coincides with the identity on the boundary,
there exists an isotopy $(f_t)_{t\in [0,1]}$ from $f_0=\id$ to $f_1=f$
that fixed pointwise the boundary of the ball.

If $p$ is a fixed point of $f$, one can furthermore require $f_t(p)=p$ for any $t$.
\end{prop}
\begin{proof}
In the standard ball $B(0,1)$, for $t\in (0,1]$,
one defines $f_t(x)$ to coincide with $t.f(x/t)$ when $\|x\|\leq t$ and with $x$ when $\|x\|>t$.
\end{proof}

\subsection{Topology of surfaces: some lemmas}

\begin{lemm}\label{lemma.approximation-disc} For any connected compact subset $K$ of the plane
and any neighborhood $U$ of $K$,
there exists a topological disc containing $K$ bounded by a curve included in $U$.
\end{lemm}
\begin{proof}
Let us consider a surface with boundary $M\subset U$ which contains $K$.
The disc is obtained as the union of $M$ with all the bounded connected components of
its complement.
\end{proof}


\begin{lemm}\label{lemm.bounding-disk}
Let $M$ be a (non necessarily orientable) surface.
Any contractible simple closed curve $\gamma$ is the boundary of a topological disc
$\Delta\subset M$.
\end{lemm}
\begin{proof}
Let us assume that $M$ is a surface different from the sphere and the projective plane.
Its universal cover $X$ is the plane.
Since $\gamma$ is contractible, it lifts in $X$ as a closed curve $\tilde \gamma$.
The Schoenflies theorem implies that $\tilde \gamma$ bounds a disc $\tilde \Delta$.
Since $\gamma$ is simple and contractible in $M$, the loop $\tilde \gamma$
is disjoint from its images by the automorphisms $T$ of the cover $X\to M$.
For each $T\neq \id$ we have $T(\tilde \gamma)\subset \tilde \Delta$
or $T(\tilde \gamma)\cap \tilde \Delta=\emptyset$.
Since $T$ has no fixed point, there is no inclusion
$T(\tilde \Delta)\subset\tilde  \Delta$ or $T(\tilde \Delta)\supset \tilde \Delta$.
Since $X$ is the plane, this gives $T(\tilde \Delta)\cap \tilde \Delta=\emptyset$:
the projection of $\tilde\Delta$ in $M$ is then a disc which is bounded by $\gamma$.

When $M$ is the sphere, the lemma holds immediately by the Schoenflies theorem.

When $M$ is the projective plane, its universal cover $X$ is the sphere.
There exists two lifts $\tilde \gamma, \tilde \gamma'$ of $\gamma$:
they are disjoint and exchanged by the antipodal map $T$.
Let $\tilde \Delta\subset X$ be the disc bounded by $\tilde \gamma$
which is disjoint from $\tilde \gamma'$.
The image $T(\tilde \Delta)$ does not contain $\tilde \Delta$ since $T$
has no fixed point: it is the disc bounded by $\tilde \gamma'$
which is disjoint from $\tilde \gamma$, hence from $\tilde \Delta$.
Again $\tilde \Delta$ induces a topological disc on $M$ bounded by $\gamma$.
\end{proof}

\begin{coro}\label{c.essential-2-sided}
In every surface which is not the sphere, the plane or the projective plane, there exists an essential 2-sided curve.
\end{coro}
\begin{proof}
If there exists no essential curve, then we are on the sphere or the plane. Assume there is an essential 1-sided curve. By definition this curve is included in a M\"obius band, whose boundary is a 2-sided curve. If this curve is not essential then according to Lemma~\ref{lemm.bounding-disk} it bounds a disk, and we are in the projective plane.
\end{proof}

\begin{lemm}\label{l.bound-annulus}
Let $M$ be a (non necessarily orientable) surface.
Two oriented essential simple closed curves $\gamma,\gamma'$
that are disjoint and homotopic bound an annulus. More precisely, there exists an embedding $i: \bbS^1 \times [0,1]$ such that $i(t,0) = \gamma(t)$ and $i(t,1) = \gamma'(t)$ for every $t \in \bbS^1$.
\end{lemm}
\begin{proof}
There is no essential loop on the sphere.
In the projective space, all essential loops are homotopic.
Moreover the intersection number in $\mathbb{Z}/2\mathbb{Z}$ of any two such loops is well defined
and non zero. Hence there is no couple of disjoint essential loops in this case.

Let us now assume that $M$ is a surface different from the sphere and the projective plane.
Its universal cover $X$ is the plane. The simple curves $\gamma, \gamma'$
may be lifted as proper lines $\tilde \gamma, \tilde \gamma'$, which bound a strip $S$
(using the Schoenflies theorem). There exists an automorphism $\tau$ of the cover
such that $\gamma=\tilde \gamma/\tau$.
One considers the intermediate covering
$Y:=X/\tau\to M$. The curve $\gamma$ lifts as a loop $\bar \gamma$.
Note that $\tau$ is orientable (otherwise as above $\gamma$ meets any loop $\gamma'$
that is homotopic to $\gamma$) and $Y$ is an annulus.
A homotopy between $\gamma$ and $\gamma'$ may be lifted:
this gives a homotopy in $Y$ between $\bar \gamma$ and a loop
$\bar \gamma'$ which lifts $\gamma'$. The two loops $\bar \gamma$
and $\bar \gamma'$ are essential and disjoint in the annulus $Y$, hence (using the Schoenflies theorem)
bound an annulus $A=S/\tau$.

Let us assume that there exists an automorphism $T$ of the covering such that
$T(\tilde \gamma)$ meets (hence is contained) in $S$.
The projection of $T(\tilde \gamma)$ in $Y$ is compact, hence is a closed curve.
It is simple, hence $T(\tilde \gamma)$ is invariant by $\tau$.
We have shown that the two automorphisms $\tau,T$ commute,
so that $M$ is a quotient of the torus and the fundamental group of $M$ is abelian.
All the lifts of $\gamma$ is $ Y$ are essential loops, hence there exists
an automorphism $T$ of $Y\to M$ such that the annulus bounded
by $\bar \gamma'$ and $T(\bar \gamma)$ is
disjoint from all the other lifts of $\gamma$ and $\gamma'$.
One deduces that the annulus bounded by $\bar \gamma'$ and $T(\bar \gamma)$
projects in $M$ as an annulus as required.

Let us now assume that $T(\tilde \gamma\cup \tilde \gamma')$ is disjoint from $S$
for any automorphism $T$. One deduces that $S$ is disjoint from all its images
$T(S)$ such that $T$ is not in the group generated by $\tau$.
This proves that the projection of $S$ in $M$ is an annulus.
\end{proof}

\begin{lemm}\label{l.fibration}
Let $M$ be a surface, $x_{0}$ a point of $M$. Then the map $\homeo(M) \to M$ given by $h \mapsto h(x_{0})$ is a fiber bundle. In particular, if the trajectory of $x_{0}$ under an isotopy $(f_{t})$ is a closed contractible loop, then there is an isotopy $(f'_{t})$, from the identity to $f$, for which the trajectory of $x_{0}$ is constant; furthermore the isotopy $(f'_{t})$ is homotopic to the isotopy $(f_{t})$ and coincides with $(f_{t})$ outside some compact subset of $M$.
\end{lemm}
\begin{proof}
Since $\homeo(M)$ is a group, it is enough to prove that the map $h\mapsto h(x_{0})$ admits local sections.
Working in a chart, we see that it is enough to prove the statement with $M$ equal to the open unit disk $\bbD^2$ and $\homeo(M)$ replaced by the space of compactly supported homeomorphisms of the disk. Then the statement boils down to the following property: one can associate to each point $x$ of the unit open disk a homeomorphism $h_{x}$ which is the identity on the boundary and satisfies $h_{x}(0) = x$, in such a way  that $h_{x}$ depends continuously on $x$. This is easily achieved by an explicit construction: for instance, one can send linearly each segment $[0z]$, joining $0$ to  a point $z$ on the boundary of the disc, to the segment $[xz]$.

Then the second part follows from standard homotopy lifting properties (see for example~\cite{hatcher}, section 4.2). More precisely, a fiber bundle has the homotopy lifting property with respect to the segment $[0,1]$ (\cite{hatcher}, proposition 4.48). By hypothesis, the trajectory $t \mapsto f_{t}(x_{0})$ is a contractible loop, 
thus there is a homotopy $H: [0,1] \times [0,1] \to M$ from this loop to the constant loop. This homotopy lifts to a homotopy $\tilde H$ in $\homeo(M)$ with $\tilde H (t,0) = f_{t}$ for every $t \in [0,1]$. Define the new homotopy by restricting $\tilde H$ to the three other sides of the square. 

To see that the new homotopy may be chosen to coincide with $(f_{t})$ ouside a compact subset of $M$, we note that our local section was compactly supported. Thus the restriction of the map $f \mapsto f(x_{0})$ to the space of compactly supported homeomorphisms of $M$ is also a fiber bundle. This provides a compactly supported isotopy $(g_{t})$, from the identity to the identity and homotopic to the constant (identity) isotopy, under which the trajectory of $x_{0}$ coincides with the trajectory under $(f_{t})$. Now the isotopy $(f'_{t}) =(g_{t}^{-1} f_{t})$ suits our needs.
\end{proof}


%
%

\bibliographystyle{alpha}
\bibliography{biblio-isotopies}

\end{document}

%% file: isotopie-globale.pdf_tex
\begingroup%
  \makeatletter%
  \providecommand\color[2][]{%
    \errmessage{(Inkscape) Color is used for the text in Inkscape, but the package 'color.sty' is not loaded}%
    \renewcommand\color[2][]{}%
  }%
  \providecommand\transparent[1]{%
    \errmessage{(Inkscape) Transparency is used (non-zero) for the text in Inkscape, but the package 'transparent.sty' is not loaded}%
    \renewcommand\transparent[1]{}%
  }%
  \providecommand\rotatebox[2]{#2}%
  \ifx\svgwidth\undefined%
    \setlength{\unitlength}{513.6bp}%
    \ifx\svgscale\undefined%
      \relax%
    \else%
      \setlength{\unitlength}{\unitlength * \real{\svgscale}}%
    \fi%
  \else%
    \setlength{\unitlength}{\svgwidth}%
  \fi%
  \global\let\svgwidth\undefined%
  \global\let\svgscale\undefined%
  \makeatother%
  \begin{picture}(1,1)%
    \put(0,0){\includegraphics[width=\unitlength]{isotopie-globale.pdf}}%
    \put(0.24542445,0.95952152){\color[rgb]{0,0,0}\makebox(0,0)[lb]{\smash{$\alpha_0$}}}%
    \put(0.38780179,0.96633616){\color[rgb]{0,0,0}\makebox(0,0)[lb]{\smash{$f(\alpha_0)$}}}%
    \put(0.02711254,0.70981358){\color[rgb]{0,0,0}\makebox(0,0)[lb]{\smash{$F$}}}%
    \put(0.55878294,0.74243734){\color[rgb]{0,0,0}\makebox(0,0)[lb]{\smash{$\alpha_1$}}}%
    \put(0.59646055,0.95984313){\color[rgb]{0,0,0}\makebox(0,0)[lb]{\smash{$f_1(\alpha_1)$}}}%
    \put(0.13197296,0.41106664){\color[rgb]{0,0,0}\makebox(0,0)[lb]{\smash{$\alpha_2$}}}%
    \put(0.01016676,0.33926953){\color[rgb]{0,0,0}\makebox(0,0)[lb]{\smash{$f_2(\alpha_2)$}}}%
    \put(0.3647324,0.42900171){\color[rgb]{0,0,0}\makebox(0,0)[lb]{\smash{$\alpha_3$}}}%
    \put(0.38416915,0.33970808){\color[rgb]{0,0,0}\makebox(0,0)[lb]{\smash{$f_3(\alpha_3)$}}}%
    \put(0.5539604,0.33043278){\color[rgb]{0,0,0}\makebox(0,0)[lb]{\smash{$\alpha_4$}}}%
    \put(0.80403342,0.09861323){\color[rgb]{0,0,0}\makebox(0,0)[lb]{\smash{$\alpha_5$}}}%
    \put(0.78392091,0.19611385){\color[rgb]{0,0,0}\makebox(0,0)[lb]{\smash{$f_5(\alpha_5)$}}}%
    \put(0.5618366,0.26243491){\color[rgb]{0,0,0}\makebox(0,0)[lb]{\smash{$f_4(\alpha_4)$}}}%
  \end{picture}%
\endgroup%

%% file: sc-local-triviality.pdf_tex
\begingroup%
  \makeatletter%
  \providecommand\color[2][]{%
    \errmessage{(Inkscape) Color is used for the text in Inkscape, but the package 'color.sty' is not loaded}%
    \renewcommand\color[2][]{}%
  }%
  \providecommand\transparent[1]{%
    \errmessage{(Inkscape) Transparency is used (non-zero) for the text in Inkscape, but the package 'transparent.sty' is not loaded}%
    \renewcommand\transparent[1]{}%
  }%
  \providecommand\rotatebox[2]{#2}%
  \ifx\svgwidth\undefined%
    \setlength{\unitlength}{253.57324219bp}%
    \ifx\svgscale\undefined%
      \relax%
    \else%
      \setlength{\unitlength}{\unitlength * \real{\svgscale}}%
    \fi%
  \else%
    \setlength{\unitlength}{\svgwidth}%
  \fi%
  \global\let\svgwidth\undefined%
  \global\let\svgscale\undefined%
  \makeatother%
  \begin{picture}(1,0.87033766)%
    \put(0,0){\includegraphics[width=\unitlength,page=1]{sc-local-triviality.pdf}}%
    \put(0.64675594,0.02388902){\color[rgb]{0,0,0}\makebox(0,0)[lb]{\smash{}}}%
    \put(0.85182489,0.56022326){\color[rgb]{0,0,0}\makebox(0,0)[lb]{\smash{}}}%
    \put(0.6309814,0.00811449){\color[rgb]{0,0,0}\makebox(0,0)[lb]{\smash{$w$}}}%
    \put(0.44168698,0.00811449){\color[rgb]{0,0,0}\makebox(0,0)[lb]{\smash{$z$}}}%
    \put(0.69069929,0.39909761){\color[rgb]{0,0,0}\makebox(0,0)[lb]{\smash{$C_w$}}}%
    \put(0.3391525,0.40022438){\color[rgb]{0,0,0}\makebox(0,0)[lb]{\smash{$C_z$}}}%
    \put(0,0){\includegraphics[width=\unitlength,page=2]{sc-local-triviality.pdf}}%
    \put(0.55210873,0.68641952){\color[rgb]{0,0,0}\makebox(0,0)[lb]{\smash{$\theta(z,.)$}}}%
    \put(0.07887268,0.82839034){\color[rgb]{0,0,0}\makebox(0,0)[lb]{\smash{$\cal P$}}}%
    \put(0.8991485,0.05543809){\color[rgb]{0,0,0}\makebox(0,0)[lb]{\smash{$S$}}}%
  \end{picture}%
\endgroup%

%% file: sc-local-triviality-B.pdf_tex
\begingroup%
  \makeatletter%
  \providecommand\color[2][]{%
    \errmessage{(Inkscape) Color is used for the text in Inkscape, but the package 'color.sty' is not loaded}%
    \renewcommand\color[2][]{}%
  }%
  \providecommand\transparent[1]{%
    \errmessage{(Inkscape) Transparency is used (non-zero) for the text in Inkscape, but the package 'transparent.sty' is not loaded}%
    \renewcommand\transparent[1]{}%
  }%
  \providecommand\rotatebox[2]{#2}%
  \ifx\svgwidth\undefined%
    \setlength{\unitlength}{253.57324219bp}%
    \ifx\svgscale\undefined%
      \relax%
    \else%
      \setlength{\unitlength}{\unitlength * \real{\svgscale}}%
    \fi%
  \else%
    \setlength{\unitlength}{\svgwidth}%
  \fi%
  \global\let\svgwidth\undefined%
  \global\let\svgscale\undefined%
  \makeatother%
  \begin{picture}(1,0.87524409)%
    \put(0,0){\includegraphics[width=\unitlength,page=1]{sc-local-triviality-B.pdf}}%
    \put(0.64675594,0.02879546){\color[rgb]{0,0,0}\makebox(0,0)[lb]{\smash{}}}%
    \put(0.85182489,0.5651297){\color[rgb]{0,0,0}\makebox(0,0)[lb]{\smash{}}}%
    \put(0.6309814,0.01302092){\color[rgb]{0,0,0}\makebox(0,0)[lb]{\smash{$w$}}}%
    \put(0.44168698,0.01302092){\color[rgb]{0,0,0}\makebox(0,0)[lb]{\smash{$z_0$}}}%
    \put(0,0){\includegraphics[width=\unitlength,page=2]{sc-local-triviality-B.pdf}}%
    \put(0.50140487,0.68005844){\color[rgb]{0,0,0}\makebox(0,0)[lb]{\smash{$\theta(z_t,.)$}}}%
    \put(0.07887268,0.83329677){\color[rgb]{0,0,0}\makebox(0,0)[lb]{\smash{$\cal P$}}}%
    \put(0.8991485,0.06034453){\color[rgb]{0,0,0}\makebox(0,0)[lb]{\smash{$S$}}}%
    \put(0,0){\includegraphics[width=\unitlength,page=3]{sc-local-triviality-B.pdf}}%
    \put(0.34816654,0.01302092){\color[rgb]{0,0,0}\makebox(0,0)[lb]{\smash{$z_t$}}}%
    \put(0,0){\includegraphics[width=\unitlength,page=4]{sc-local-triviality-B.pdf}}%
    \put(0.25577281,0.01302092){\color[rgb]{0,0,0}\makebox(0,0)[lb]{\smash{$z_1$}}}%
    \put(0,0){\includegraphics[width=\unitlength,page=5]{sc-local-triviality-B.pdf}}%
    \put(0.66253047,0.28118807){\color[rgb]{0,0,0}\makebox(0,0)[lb]{\smash{$\gamma'_0$}}}%
    \put(0.67830501,0.39160981){\color[rgb]{0,0,0}\makebox(0,0)[lb]{\smash{$\gamma'_t$}}}%
    \put(0.47323605,0.2969626){\color[rgb]{0,0,0}\makebox(0,0)[lb]{\smash{$\gamma_0$}}}%
    \put(0.37858884,0.44682069){\color[rgb]{0,0,0}\makebox(0,0)[lb]{\smash{$\gamma_t$}}}%
    \put(0.28394163,0.59667877){\color[rgb]{0,0,0}\makebox(0,0)[lb]{\smash{$\gamma_0$}}}%
    \put(0.67830506,0.5245666){\color[rgb]{0,0,0}\makebox(0,0)[lb]{\smash{$\gamma'_1$}}}%
  \end{picture}%
\endgroup%

%% file: sc-local-triviality-C.pdf_tex
\begingroup%
  \makeatletter%
  \providecommand\color[2][]{%
    \errmessage{(Inkscape) Color is used for the text in Inkscape, but the package 'color.sty' is not loaded}%
    \renewcommand\color[2][]{}%
  }%
  \providecommand\transparent[1]{%
    \errmessage{(Inkscape) Transparency is used (non-zero) for the text in Inkscape, but the package 'transparent.sty' is not loaded}%
    \renewcommand\transparent[1]{}%
  }%
  \providecommand\rotatebox[2]{#2}%
  \ifx\svgwidth\undefined%
    \setlength{\unitlength}{101.31528625bp}%
    \ifx\svgscale\undefined%
      \relax%
    \else%
      \setlength{\unitlength}{\unitlength * \real{\svgscale}}%
    \fi%
  \else%
    \setlength{\unitlength}{\svgwidth}%
  \fi%
  \global\let\svgwidth\undefined%
  \global\let\svgscale\undefined%
  \makeatother%
  \begin{picture}(1,1.20945494)%
    \put(0,0){\includegraphics[width=\unitlength,page=1]{sc-local-triviality-C.pdf}}%
    \put(0.37274242,0.01862225){\color[rgb]{0,0,0}\makebox(0,0)[lb]{\smash{$\varphi_{z,0}$}}}%
    \put(0.63782727,0.49803101){\color[rgb]{0,0,0}\makebox(0,0)[lb]{\smash{\ZZ}}}%
    \put(0.32762159,1.14946276){\color[rgb]{0,0,0}\makebox(0,0)[lb]{\smash{$\psi_{z,1}$}}}%
    \put(-0.00655442,1.07473139){\color[rgb]{0,0,0}\makebox(0,0)[lb]{\smash{$f(z)$}}}%
    \put(0.57296602,1.10293192){\color[rgb]{0,0,0}\makebox(0,0)[lb]{\smash{$f(w)$}}}%
    \put(0.15136842,0.0454127){\color[rgb]{0,0,0}\makebox(0,0)[lb]{\smash{$w$}}}%
    \put(0.56450589,0.14270457){\color[rgb]{0,0,0}\makebox(0,0)[lb]{\smash{$z$}}}%
  \end{picture}%
\endgroup%

%% file: properhmotopy-not-planar1.pdf_tex
\begingroup%
  \makeatletter%
  \providecommand\color[2][]{%
    \errmessage{(Inkscape) Color is used for the text in Inkscape, but the package 'color.sty' is not loaded}%
    \renewcommand\color[2][]{}%
  }%
  \providecommand\transparent[1]{%
    \errmessage{(Inkscape) Transparency is used (non-zero) for the text in Inkscape, but the package 'transparent.sty' is not loaded}%
    \renewcommand\transparent[1]{}%
  }%
  \providecommand\rotatebox[2]{#2}%
  \ifx\svgwidth\undefined%
    \setlength{\unitlength}{139.85120216bp}%
    \ifx\svgscale\undefined%
      \relax%
    \else%
      \setlength{\unitlength}{\unitlength * \real{\svgscale}}%
    \fi%
  \else%
    \setlength{\unitlength}{\svgwidth}%
  \fi%
  \global\let\svgwidth\undefined%
  \global\let\svgscale\undefined%
  \makeatother%
  \begin{picture}(1,0.4369284)%
    \put(0,0){\includegraphics[width=\unitlength,page=1]{properhmotopy-not-planar1.pdf}}%
    \put(0.68267045,0.17793181){\color[rgb]{0,0,0}\makebox(0,0)[lb]{\smash{$x$}}}%
    \put(0.48450067,0.2147056){\color[rgb]{0,0,0}\makebox(0,0)[lb]{\smash{$\alpha$}}}%
    \put(0.26079349,0.39346703){\color[rgb]{0,0,0}\makebox(0,0)[lb]{\smash{$\gamma$}}}%
    \put(0.89003379,0.14115802){\color[rgb]{0,0,0}\makebox(0,0)[lb]{\smash{$U$}}}%
    \put(0.03708632,0.25658684){\color[rgb]{0,0,0}\makebox(0,0)[lb]{\smash{$V$}}}%
    \put(0,0){\includegraphics[width=\unitlength,page=2]{properhmotopy-not-planar1.pdf}}%
  \end{picture}%
\endgroup%

%% file: properhmotopy-planar.pdf_tex
\begingroup%
  \makeatletter%
  \providecommand\color[2][]{%
    \errmessage{(Inkscape) Color is used for the text in Inkscape, but the package 'color.sty' is not loaded}%
    \renewcommand\color[2][]{}%
  }%
  \providecommand\transparent[1]{%
    \errmessage{(Inkscape) Transparency is used (non-zero) for the text in Inkscape, but the package 'transparent.sty' is not loaded}%
    \renewcommand\transparent[1]{}%
  }%
  \providecommand\rotatebox[2]{#2}%
  \ifx\svgwidth\undefined%
    \setlength{\unitlength}{146.0790219bp}%
    \ifx\svgscale\undefined%
      \relax%
    \else%
      \setlength{\unitlength}{\unitlength * \real{\svgscale}}%
    \fi%
  \else%
    \setlength{\unitlength}{\svgwidth}%
  \fi%
  \global\let\svgwidth\undefined%
  \global\let\svgscale\undefined%
  \makeatother%
  \begin{picture}(1,0.25005065)%
    \put(0,0){\includegraphics[width=\unitlength,page=1]{properhmotopy-planar.pdf}}%
    \put(0.63513885,0.13432323){\color[rgb]{0,0,0}\makebox(0,0)[lb]{\smash{$x$}}}%
    \put(0.76422757,0.13823497){\color[rgb]{0,0,0}\makebox(0,0)[lb]{\smash{$\alpha$}}}%
    \put(0.87864702,0.10107314){\color[rgb]{0,0,0}\makebox(0,0)[lb]{\smash{$U$}}}%
    \put(0,0){\includegraphics[width=\unitlength,page=2]{properhmotopy-planar.pdf}}%
    \put(0.23222574,0.14801453){\color[rgb]{0,0,0}\makebox(0,0)[lb]{\smash{$x'$}}}%
  \end{picture}%
\endgroup%

%% file: compact-approx.pdf_tex
\begingroup%
  \makeatletter%
  \providecommand\color[2][]{%
    \errmessage{(Inkscape) Color is used for the text in Inkscape, but the package 'color.sty' is not loaded}%
    \renewcommand\color[2][]{}%
  }%
  \providecommand\transparent[1]{%
    \errmessage{(Inkscape) Transparency is used (non-zero) for the text in Inkscape, but the package 'transparent.sty' is not loaded}%
    \renewcommand\transparent[1]{}%
  }%
  \providecommand\rotatebox[2]{#2}%
  \ifx\svgwidth\undefined%
    \setlength{\unitlength}{357.12970718bp}%
    \ifx\svgscale\undefined%
      \relax%
    \else%
      \setlength{\unitlength}{\unitlength * \real{\svgscale}}%
    \fi%
  \else%
    \setlength{\unitlength}{\svgwidth}%
  \fi%
  \global\let\svgwidth\undefined%
  \global\let\svgscale\undefined%
  \makeatother%
  \begin{picture}(1,0.47925868)%
    \put(0,0){\includegraphics[width=\unitlength,page=1]{compact-approx.pdf}}%
    \put(0.43640932,0.24757671){\color[rgb]{0,0,0}\makebox(0,0)[lb]{\smash{$\gamma$}}}%
    \put(0,0){\includegraphics[width=\unitlength,page=2]{compact-approx.pdf}}%
    \put(0.53801307,0.15557329){\color[rgb]{0,0,0}\makebox(0,0)[lb]{\smash{$F$}}}%
    \put(0.91882709,0.45798446){\color[rgb]{0,0,0}\makebox(0,0)[lb]{\smash{$F_n$}}}%
    \put(0.25800276,0.4019824){\color[rgb]{0,0,0}\makebox(0,0)[lb]{\smash{$M$}}}%
    \put(0,0){\includegraphics[width=\unitlength,page=3]{compact-approx.pdf}}%
    \put(0.0227941,0.32357951){\color[rgb]{0,0,0}\makebox(0,0)[lb]{\smash{$S$}}}%
    \put(0,0){\includegraphics[width=\unitlength,page=4]{compact-approx.pdf}}%
  \end{picture}%
\endgroup%

%% file: pushing-homotopy.pdf_tex
\begingroup%
  \makeatletter%
  \providecommand\color[2][]{%
    \errmessage{(Inkscape) Color is used for the text in Inkscape, but the package 'color.sty' is not loaded}%
    \renewcommand\color[2][]{}%
  }%
  \providecommand\transparent[1]{%
    \errmessage{(Inkscape) Transparency is used (non-zero) for the text in Inkscape, but the package 'transparent.sty' is not loaded}%
    \renewcommand\transparent[1]{}%
  }%
  \providecommand\rotatebox[2]{#2}%
  \ifx\svgwidth\undefined%
    \setlength{\unitlength}{255.94768914bp}%
    \ifx\svgscale\undefined%
      \relax%
    \else%
      \setlength{\unitlength}{\unitlength * \real{\svgscale}}%
    \fi%
  \else%
    \setlength{\unitlength}{\svgwidth}%
  \fi%
  \global\let\svgwidth\undefined%
  \global\let\svgscale\undefined%
  \makeatother%
  \begin{picture}(1,0.44097723)%
    \put(0,0){\includegraphics[width=\unitlength]{pushing-homotopy.pdf}}%
    \put(0.01210868,0.11468109){\color[rgb]{0,0,0}\makebox(0,0)[lb]{\smash{$\gamma$}}}%
    \put(0.09024964,0.05216831){\color[rgb]{0,0,0}\makebox(0,0)[lb]{\smash{$F$}}}%
    \put(0.43406989,0.05216831){\color[rgb]{0,0,0}\makebox(0,0)[lb]{\smash{$F_n$}}}%
    \put(0.66849279,0.36473218){\color[rgb]{0,0,0}\makebox(0,0)[lb]{\smash{$f(\gamma)$}}}%
    \put(0.91854388,0.28659121){\color[rgb]{0,0,0}\makebox(0,0)[lb]{\smash{$s$}}}%
    \put(0.91854388,0.20845025){\color[rgb]{0,0,0}\makebox(0,0)[lb]{\smash{$t$}}}%
    \put(0.91854388,0.06779651){\color[rgb]{0,0,0}\makebox(0,0)[lb]{\smash{$u$}}}%
    \put(0.77900645,0.3301269){\color[rgb]{0,0,0}\makebox(0,0)[lb]{\smash{variables:}}}%
    \put(0.77874326,0.16185984){\color[rgb]{0,0,0}\makebox(0,0)[lb]{\smash{$I_n.\gamma$}}}%
    \put(0.77953259,0.22342862){\color[rgb]{0,0,0}\makebox(0,0)[lb]{\smash{$H_{\gamma,n}$}}}%
    \put(0.91854388,0.14593747){\color[rgb]{0,0,0}\makebox(0,0)[lb]{\smash{$t$}}}%
  \end{picture}%
\endgroup%

%% file: pushing-homotopy-lift.pdf_tex
\begingroup%
  \makeatletter%
  \providecommand\color[2][]{%
    \errmessage{(Inkscape) Color is used for the text in Inkscape, but the package 'color.sty' is not loaded}%
    \renewcommand\color[2][]{}%
  }%
  \providecommand\transparent[1]{%
    \errmessage{(Inkscape) Transparency is used (non-zero) for the text in Inkscape, but the package 'transparent.sty' is not loaded}%
    \renewcommand\transparent[1]{}%
  }%
  \providecommand\rotatebox[2]{#2}%
  \ifx\svgwidth\undefined%
    \setlength{\unitlength}{288.29784551bp}%
    \ifx\svgscale\undefined%
      \relax%
    \else%
      \setlength{\unitlength}{\unitlength * \real{\svgscale}}%
    \fi%
  \else%
    \setlength{\unitlength}{\svgwidth}%
  \fi%
  \global\let\svgwidth\undefined%
  \global\let\svgscale\undefined%
  \makeatother%
  \begin{picture}(1,0.43184505)%
    \put(0,0){\includegraphics[width=\unitlength]{pushing-homotopy-lift.pdf}}%
    \put(0.79084877,0.25142488){\color[rgb]{0,0,0}\makebox(0,0)[lb]{\smash{$\tilde F'_p$}}}%
    \put(0.31395205,0.03346835){\color[rgb]{0,0,0}\makebox(0,0)[lb]{\smash{$\tilde M$}}}%
    \put(0.31535361,0.37824727){\color[rgb]{0,0,0}\makebox(0,0)[lb]{\smash{$\partial \tilde M$}}}%
    \put(0.90184508,0.09880495){\color[rgb]{0,0,0}\makebox(0,0)[lb]{\smash{$R_u$}}}%
  \end{picture}%
\endgroup%

%% file: preuve-semi-continuite.pdf_tex
\begingroup%
  \makeatletter%
  \providecommand\color[2][]{%
    \errmessage{(Inkscape) Color is used for the text in Inkscape, but the package 'color.sty' is not loaded}%
    \renewcommand\color[2][]{}%
  }%
  \providecommand\transparent[1]{%
    \errmessage{(Inkscape) Transparency is used (non-zero) for the text in Inkscape, but the package 'transparent.sty' is not loaded}%
    \renewcommand\transparent[1]{}%
  }%
  \providecommand\rotatebox[2]{#2}%
  \ifx\svgwidth\undefined%
    \setlength{\unitlength}{548.89137965bp}%
    \ifx\svgscale\undefined%
      \relax%
    \else%
      \setlength{\unitlength}{\unitlength * \real{\svgscale}}%
    \fi%
  \else%
    \setlength{\unitlength}{\svgwidth}%
  \fi%
  \global\let\svgwidth\undefined%
  \global\let\svgscale\undefined%
  \makeatother%
  \begin{picture}(1,0.70241954)%
    \put(0,0){\includegraphics[width=\unitlength]{preuve-semi-continuite.pdf}}%
    \put(0.51169956,0.43116548){\color[rgb]{0,0,0}\makebox(0,0)[lb]{\smash{$x$}}}%
    \put(0.81170229,0.04292661){\color[rgb]{0,0,0}\makebox(0,0)[lb]{\smash{$z$}}}%
    \put(0.71353973,0.24476669){\color[rgb]{0,0,0}\makebox(0,0)[lb]{\smash{$f(z)$}}}%
    \put(0.86023218,0.12564799){\color[rgb]{0,0,0}\makebox(0,0)[lb]{\smash{$I_p.z$}}}%
    \put(0.35177164,0.60212293){\color[rgb]{0,0,0}\makebox(0,0)[lb]{\smash{$F$}}}%
    \put(0.93082105,0.65947639){\color[rgb]{0,0,0}\makebox(0,0)[lb]{\smash{$V$}}}%
    \put(0.17640237,0.42895958){\color[rgb]{0,0,0}\makebox(0,0)[lb]{\smash{$Z$}}}%
    \put(0.20177023,0.25248736){\color[rgb]{0,0,0}\makebox(0,0)[lb]{\smash{$I_p.{z'}$}}}%
    \put(0.49294939,0.19844277){\color[rgb]{0,0,0}\makebox(0,0)[lb]{\smash{$\gamma_{z'}$}}}%
    \put(0.42897819,0.1157214){\color[rgb]{0,0,0}\makebox(0,0)[lb]{\smash{$z'$}}}%
    \put(0.46868444,0.30873786){\color[rgb]{0,0,0}\makebox(0,0)[lb]{\smash{$f(z')$}}}%
    \put(0.33081556,0.3010172){\color[rgb]{0,0,0}\makebox(0,0)[lb]{\smash{$x'$}}}%
    \put(0.56795006,0.06057383){\color[rgb]{0,0,0}\makebox(0,0)[lb]{\smash{$\alpha$}}}%
    \put(0.7190544,0.3396205){\color[rgb]{0,0,0}\makebox(0,0)[lb]{\smash{$f(\alpha)$}}}%
  \end{picture}%
\endgroup%

%% file: quasi-transverse.pdf_tex
\begingroup%
  \makeatletter%
  \providecommand\color[2][]{%
    \errmessage{(Inkscape) Color is used for the text in Inkscape, but the package 'color.sty' is not loaded}%
    \renewcommand\color[2][]{}%
  }%
  \providecommand\transparent[1]{%
    \errmessage{(Inkscape) Transparency is used (non-zero) for the text in Inkscape, but the package 'transparent.sty' is not loaded}%
    \renewcommand\transparent[1]{}%
  }%
  \providecommand\rotatebox[2]{#2}%
  \ifx\svgwidth\undefined%
    \setlength{\unitlength}{435.625bp}%
    \ifx\svgscale\undefined%
      \relax%
    \else%
      \setlength{\unitlength}{\unitlength * \real{\svgscale}}%
    \fi%
  \else%
    \setlength{\unitlength}{\svgwidth}%
  \fi%
  \global\let\svgwidth\undefined%
  \global\let\svgscale\undefined%
  \makeatother%
  \begin{picture}(1,1.21829896)%
    \put(0,0){\includegraphics[width=\unitlength]{quasi-transverse.pdf}}%
    \put(0.26066619,0.79008098){\color[rgb]{0,0,0}\makebox(0,0)[lb]{\smash{$\alpha'$}}}%
    \put(0.18924537,0.96724172){\color[rgb]{0,0,0}\makebox(0,0)[lb]{\smash{$\alpha$}}}%
    \put(0.14565086,0.67599315){\color[rgb]{0,0,0}\makebox(0,0)[lb]{\smash{$D$}}}%
    \put(0.71989531,0.73820453){\color[rgb]{0,0,0}\makebox(0,0)[lb]{\smash{$\bar\alpha_1$}}}%
    \put(0.72443508,0.97278857){\color[rgb]{0,0,0}\makebox(0,0)[lb]{\smash{$\bar\alpha_2$}}}%
    \put(0.83183891,0.87632877){\color[rgb]{0,0,0}\makebox(0,0)[lb]{\smash{$\bar\alpha_3$}}}%
    \put(0.72072493,1.1013395){\color[rgb]{0,0,0}\makebox(0,0)[lb]{\smash{$\bar\alpha_4$}}}%
    \put(0.69114144,0.6667837){\color[rgb]{0,0,0}\makebox(0,0)[lb]{\smash{Restriction to $D$}}}%
    \put(0.07996925,0.00468234){\color[rgb]{0,0,0}\makebox(0,0)[lb]{\smash{Schoenflies coordinates}}}%
    \put(0.73181374,0.00323888){\color[rgb]{0,0,0}\makebox(0,0)[lb]{\smash{Back to $D$}}}%
    \put(0.05121542,0.21059698){\color[rgb]{0,0,0}\makebox(0,0)[lb]{\smash{$\bar\alpha_1$}}}%
    \put(0.04555226,0.31810763){\color[rgb]{0,0,0}\makebox(0,0)[lb]{\smash{$\bar\alpha_2$}}}%
    \put(0.31156606,0.1261109){\color[rgb]{0,0,0}\makebox(0,0)[lb]{\smash{$\bar\alpha_3$}}}%
    \put(0.09934978,0.45407898){\color[rgb]{0,0,0}\makebox(0,0)[lb]{\smash{$\bar\alpha_4$}}}%
    \put(0.79780891,0.7957122){\color[rgb]{0,0,0}\makebox(0,0)[lb]{\smash{$\bar\alpha'$}}}%
    \put(0.27846353,0.22358254){\color[rgb]{0,0,0}\makebox(0,0)[lb]{\smash{$\bar\alpha'$}}}%
    \put(0.14582482,0.37384466){\color[rgb]{0,0,0}\makebox(0,0)[lb]{\smash{$\bar\alpha''=f(\bar\alpha')$}}}%
    \put(0.71525758,0.79756725){\color[rgb]{0,0,0}\makebox(0,0)[lb]{\smash{$\check\alpha$}}}%
    \put(0.21724565,0.12619048){\color[rgb]{0,0,0}\makebox(0,0)[lb]{\smash{$\check\alpha$}}}%
    \put(0.84497411,0.24254495){\color[rgb]{0,0,0}\makebox(0,0)[lb]{\smash{$\alpha$}}}%
    \put(0.80879982,0.15906609){\color[rgb]{0,0,0}\makebox(0,0)[lb]{\smash{$f(\alpha')$}}}%
  \end{picture}%
\endgroup%

%% file: trouver-un-bigone-minimal.pdf_tex
\begingroup%
  \makeatletter%
  \providecommand\color[2][]{%
    \errmessage{(Inkscape) Color is used for the text in Inkscape, but the package 'color.sty' is not loaded}%
    \renewcommand\color[2][]{}%
  }%
  \providecommand\transparent[1]{%
    \errmessage{(Inkscape) Transparency is used (non-zero) for the text in Inkscape, but the package 'transparent.sty' is not loaded}%
    \renewcommand\transparent[1]{}%
  }%
  \providecommand\rotatebox[2]{#2}%
  \ifx\svgwidth\undefined%
    \setlength{\unitlength}{365.37931362bp}%
    \ifx\svgscale\undefined%
      \relax%
    \else%
      \setlength{\unitlength}{\unitlength * \real{\svgscale}}%
    \fi%
  \else%
    \setlength{\unitlength}{\svgwidth}%
  \fi%
  \global\let\svgwidth\undefined%
  \global\let\svgscale\undefined%
  \makeatother%
  \begin{picture}(1,0.38512607)%
    \put(0,0){\includegraphics[width=\unitlength]{trouver-un-bigone-minimal.pdf}}%
    \put(0.10016987,0.03947614){\color[rgb]{0,0,0}\makebox(0,0)[lb]{\smash{$\alpha$}}}%
    \put(0.0993879,0.31316432){\color[rgb]{0,0,0}\makebox(0,0)[lb]{\smash{$\alpha'$}}}%
    \put(0.73512651,0.22949395){\color[rgb]{0,0,0}\makebox(0,0)[lb]{\smash{$B(\hat \alpha)$}}}%
    \put(0.66083968,0.18726776){\color[rgb]{0,0,0}\makebox(0,0)[lb]{\smash{$\hat \alpha$}}}%
    \put(0.80863123,0.29361516){\color[rgb]{0,0,0}\makebox(0,0)[lb]{\smash{$\hap$}}}%
  \end{picture}%
\endgroup%

%% file: bigones-successifs.pdf_tex
\begingroup%
  \makeatletter%
  \providecommand\color[2][]{%
    \errmessage{(Inkscape) Color is used for the text in Inkscape, but the package 'color.sty' is not loaded}%
    \renewcommand\color[2][]{}%
  }%
  \providecommand\transparent[1]{%
    \errmessage{(Inkscape) Transparency is used (non-zero) for the text in Inkscape, but the package 'transparent.sty' is not loaded}%
    \renewcommand\transparent[1]{}%
  }%
  \providecommand\rotatebox[2]{#2}%
  \ifx\svgwidth\undefined%
    \setlength{\unitlength}{352.7999926bp}%
    \ifx\svgscale\undefined%
      \relax%
    \else%
      \setlength{\unitlength}{\unitlength * \real{\svgscale}}%
    \fi%
  \else%
    \setlength{\unitlength}{\svgwidth}%
  \fi%
  \global\let\svgwidth\undefined%
  \global\let\svgscale\undefined%
  \makeatother%
  \begin{picture}(1,1)%
    \put(0,0){\includegraphics[width=\unitlength]{bigones-successifs.pdf}}%
    \put(0.20197602,0.54942976){\color[rgb]{0,0,0}\makebox(0,0)[lb]{\smash{$B_1$}}}%
    \put(0.3590865,0.53404264){\color[rgb]{0,0,0}\makebox(0,0)[lb]{\smash{$B_2$}}}%
    \put(0.45869775,0.58344333){\color[rgb]{0,0,0}\makebox(0,0)[lb]{\smash{$B_3$}}}%
    \put(0.74781339,0.51946534){\color[rgb]{0,0,0}\makebox(0,0)[lb]{\smash{$\tilde \alpha$}}}%
    \put(0.74295431,0.38017152){\color[rgb]{0,0,0}\makebox(0,0)[lb]{\smash{$\tilde \alpha'$}}}%
    \put(0.64577264,0.26517312){\color[rgb]{0,0,0}\makebox(0,0)[lb]{\smash{$\pi^{-1}(\alpha)$}}}%
    \put(0.4400713,0.33158061){\color[rgb]{0,0,0}\makebox(0,0)[lb]{\smash{$\pi^{-1}(\alpha')$}}}%
  \end{picture}%
\endgroup%

%% file: removing-a-bigon.pdf_tex
\begingroup%
  \makeatletter%
  \providecommand\color[2][]{%
    \errmessage{(Inkscape) Color is used for the text in Inkscape, but the package 'color.sty' is not loaded}%
    \renewcommand\color[2][]{}%
  }%
  \providecommand\transparent[1]{%
    \errmessage{(Inkscape) Transparency is used (non-zero) for the text in Inkscape, but the package 'transparent.sty' is not loaded}%
    \renewcommand\transparent[1]{}%
  }%
  \providecommand\rotatebox[2]{#2}%
  \ifx\svgwidth\undefined%
    \setlength{\unitlength}{478.49012126bp}%
    \ifx\svgscale\undefined%
      \relax%
    \else%
      \setlength{\unitlength}{\unitlength * \real{\svgscale}}%
    \fi%
  \else%
    \setlength{\unitlength}{\svgwidth}%
  \fi%
  \global\let\svgwidth\undefined%
  \global\let\svgscale\undefined%
  \makeatother%
  \begin{picture}(1,0.22270868)%
    \put(0,0){\includegraphics[width=\unitlength]{removing-a-bigon.pdf}}%
    \put(0.01995116,0.07558666){\color[rgb]{0,0,0}\makebox(0,0)[lb]{\smash{$x$}}}%
    \put(0.35780621,0.07412792){\color[rgb]{0,0,0}\makebox(0,0)[lb]{\smash{$y$}}}%
    \put(0.40166994,0.14103234){\color[rgb]{0,0,0}\makebox(0,0)[lb]{\smash{$\alpha$}}}%
    \put(0.40927102,0.06603283){\color[rgb]{0,0,0}\makebox(0,0)[lb]{\smash{$\alpha'$}}}%
    \put(0.95503539,0.1854561){\color[rgb]{0,0,0}\makebox(0,0)[lb]{\smash{$V$}}}%
    \put(0.16932364,0.11581212){\color[rgb]{0,0,0}\makebox(0,0)[lb]{\smash{$B$}}}%
    \put(0.67140512,0.16575127){\color[rgb]{0,0,0}\makebox(0,0)[lb]{\smash{$\hat \alpha'$}}}%
    \put(0.60603077,0.12002194){\color[rgb]{0,0,0}\makebox(0,0)[lb]{\smash{$\hat \alpha$}}}%
    \put(0.59515238,0.03500985){\color[rgb]{0,0,0}\makebox(0,0)[lb]{\smash{$f(\hat \alpha')$}}}%
  \end{picture}%
\endgroup%

%% file: bringing-back-1.pdf_tex
\begingroup%
  \makeatletter%
  \providecommand\color[2][]{%
    \errmessage{(Inkscape) Color is used for the text in Inkscape, but the package 'color.sty' is not loaded}%
    \renewcommand\color[2][]{}%
  }%
  \providecommand\transparent[1]{%
    \errmessage{(Inkscape) Transparency is used (non-zero) for the text in Inkscape, but the package 'transparent.sty' is not loaded}%
    \renewcommand\transparent[1]{}%
  }%
  \providecommand\rotatebox[2]{#2}%
  \ifx\svgwidth\undefined%
    \setlength{\unitlength}{533.79012032bp}%
    \ifx\svgscale\undefined%
      \relax%
    \else%
      \setlength{\unitlength}{\unitlength * \real{\svgscale}}%
    \fi%
  \else%
    \setlength{\unitlength}{\svgwidth}%
  \fi%
  \global\let\svgwidth\undefined%
  \global\let\svgscale\undefined%
  \makeatother%
  \begin{picture}(1,0.46691274)%
    \put(0,0){\includegraphics[width=\unitlength]{bringing-back-1.pdf}}%
    \put(0.15638009,0.00577479){\color[rgb]{0,0,0}\makebox(0,0)[lb]{\smash{$F$}}}%
    \put(0.60385399,0.20810148){\color[rgb]{0,0,0}\makebox(0,0)[lb]{\smash{$\beta^\lambda$}}}%
    \put(0.23452744,0.20488995){\color[rgb]{0,0,0}\makebox(0,0)[lb]{\smash{$z_\infty^\lambda$}}}%
    \put(0.48930922,0.20381944){\color[rgb]{0,0,0}\makebox(0,0)[lb]{\smash{$(z_k^\lambda)$}}}%
    \put(0.55996297,0.31729367){\color[rgb]{0,0,0}\makebox(0,0)[lb]{\smash{$I$}}}%
    \put(0.81367428,0.45324866){\color[rgb]{0,0,0}\makebox(0,0)[lb]{\smash{$f(\beta^\lambda)$}}}%
  \end{picture}%
\endgroup%

%% file: bringing-back-2.pdf_tex
\begingroup%
  \makeatletter%
  \providecommand\color[2][]{%
    \errmessage{(Inkscape) Color is used for the text in Inkscape, but the package 'color.sty' is not loaded}%
    \renewcommand\color[2][]{}%
  }%
  \providecommand\transparent[1]{%
    \errmessage{(Inkscape) Transparency is used (non-zero) for the text in Inkscape, but the package 'transparent.sty' is not loaded}%
    \renewcommand\transparent[1]{}%
  }%
  \providecommand\rotatebox[2]{#2}%
  \ifx\svgwidth\undefined%
    \setlength{\unitlength}{533.79012032bp}%
    \ifx\svgscale\undefined%
      \relax%
    \else%
      \setlength{\unitlength}{\unitlength * \real{\svgscale}}%
    \fi%
  \else%
    \setlength{\unitlength}{\svgwidth}%
  \fi%
  \global\let\svgwidth\undefined%
  \global\let\svgscale\undefined%
  \makeatother%
  \begin{picture}(1,0.39972311)%
    \put(0,0){\includegraphics[width=\unitlength]{bringing-back-2.pdf}}%
    \put(0.86077682,0.05608881){\color[rgb]{0,0,0}\makebox(0,0)[lb]{\smash{$Z_\infty$}}}%
    \put(0.8639883,0.01540938){\color[rgb]{0,0,0}\makebox(0,0)[lb]{\smash{$F_0$}}}%
    \put(0.49894384,0.27661424){\color[rgb]{0,0,0}\makebox(0,0)[lb]{\smash{$(z_k^\lambda)$}}}%
    \put(0.43043108,0.05180681){\color[rgb]{0,0,0}\makebox(0,0)[lb]{\smash{$(c_k^\lambda)$}}}%
    \put(0.23773898,0.28089629){\color[rgb]{0,0,0}\makebox(0,0)[lb]{\smash{$z_\infty^\lambda$}}}%
    \put(0.15396851,0.00850187){\color[rgb]{0,0,0}\makebox(0,0)[lb]{\smash{$F$}}}%
  \end{picture}%
\endgroup%

%% file: end-isotopy2.pdf_tex
\begingroup%
  \makeatletter%
  \providecommand\color[2][]{%
    \errmessage{(Inkscape) Color is used for the text in Inkscape, but the package 'color.sty' is not loaded}%
    \renewcommand\color[2][]{}%
  }%
  \providecommand\transparent[1]{%
    \errmessage{(Inkscape) Transparency is used (non-zero) for the text in Inkscape, but the package 'transparent.sty' is not loaded}%
    \renewcommand\transparent[1]{}%
  }%
  \providecommand\rotatebox[2]{#2}%
  \ifx\svgwidth\undefined%
    \setlength{\unitlength}{344.73885277bp}%
    \ifx\svgscale\undefined%
      \relax%
    \else%
      \setlength{\unitlength}{\unitlength * \real{\svgscale}}%
    \fi%
  \else%
    \setlength{\unitlength}{\svgwidth}%
  \fi%
  \global\let\svgwidth\undefined%
  \global\let\svgscale\undefined%
  \makeatother%
  \begin{picture}(1,0.29634649)%
    \put(0,0){\includegraphics[width=\unitlength]{end-isotopy2.pdf}}%
    \put(-0.00231153,0.02438091){\color[rgb]{0,0,0}\makebox(0,0)[lb]{\smash{$F$}}}%
    \put(0.41356166,0.14992959){\color[rgb]{0,0,0}\makebox(0,0)[lb]{\smash{$\alpha_i$}}}%
    \put(0.45474093,0.26852907){\color[rgb]{0,0,0}\makebox(0,0)[lb]{\smash{$f(\alpha_i)$}}}%
    \put(0.57033262,0.11439247){\color[rgb]{0,0,0}\makebox(0,0)[lb]{\smash{$V_i$}}}%
  \end{picture}%
\endgroup%

%% file: fin-isotopie-A.pdf_tex
\begingroup%
  \makeatletter%
  \providecommand\color[2][]{%
    \errmessage{(Inkscape) Color is used for the text in Inkscape, but the package 'color.sty' is not loaded}%
    \renewcommand\color[2][]{}%
  }%
  \providecommand\transparent[1]{%
    \errmessage{(Inkscape) Transparency is used (non-zero) for the text in Inkscape, but the package 'transparent.sty' is not loaded}%
    \renewcommand\transparent[1]{}%
  }%
  \providecommand\rotatebox[2]{#2}%
  \ifx\svgwidth\undefined%
    \setlength{\unitlength}{580.60742188bp}%
    \ifx\svgscale\undefined%
      \relax%
    \else%
      \setlength{\unitlength}{\unitlength * \real{\svgscale}}%
    \fi%
  \else%
    \setlength{\unitlength}{\svgwidth}%
  \fi%
  \global\let\svgwidth\undefined%
  \global\let\svgscale\undefined%
  \makeatother%
  \begin{picture}(1,0.42248093)%
    \put(0,0){\includegraphics[width=\unitlength,page=1]{fin-isotopie-A.pdf}}%
    \put(0.73323197,0.40677809){\color[rgb]{0,0,0}\makebox(0,0)[lb]{\smash{$W_0$}}}%
    \put(0.43815845,0.38868396){\color[rgb]{0,0,0}\makebox(0,0)[lb]{\smash{$W_2$}}}%
    \put(0.26179479,0.36377361){\color[rgb]{0,0,0}\makebox(0,0)[lb]{\smash{$W_4$}}}%
    \put(0.10334005,0.35688427){\color[rgb]{0,0,0}\makebox(0,0)[lb]{\smash{$W_6$}}}%
    \put(0.83291382,0.28381534){\color[rgb]{0,0,0}\makebox(0,0)[lb]{\smash{$\alpha_0$}}}%
    \put(0.88183509,0.35688427){\color[rgb]{0,0,0}\makebox(0,0)[lb]{\smash{$\alpha'_0$}}}%
    \put(0.4891429,0.28451127){\color[rgb]{0,0,0}\makebox(0,0)[lb]{\smash{$\alpha_2$}}}%
    \put(0.55114693,0.34999493){\color[rgb]{0,0,0}\makebox(0,0)[lb]{\smash{$\alpha'_2$}}}%
    \put(0,0){\includegraphics[width=\unitlength,page=2]{fin-isotopie-A.pdf}}%
    \put(0.73200901,0.02805832){\color[rgb]{0,0,0}\makebox(0,0)[lb]{\smash{$W'_1$}}}%
    \put(0.43955028,0.03271148){\color[rgb]{0,0,0}\makebox(0,0)[lb]{\smash{$W'_3$}}}%
    \put(0.22047284,0.02680635){\color[rgb]{0,0,0}\makebox(0,0)[lb]{\smash{$W'_5$}}}%
    \put(0.0473556,0.05289336){\color[rgb]{0,0,0}\makebox(0,0)[lb]{\smash{$W'_7$}}}%
    \put(0.66241615,0.11644118){\color[rgb]{0,0,0}\makebox(0,0)[lb]{\smash{$\alpha_1$}}}%
    \put(0.39796347,0.11504932){\color[rgb]{0,0,0}\makebox(0,0)[lb]{\smash{$\alpha_3$}}}%
    \put(0.59301269,0.00487434){\color[rgb]{0,0,0}\makebox(0,0)[lb]{\smash{$\alpha''_1$}}}%
    \put(0.3501338,0.03103137){\color[rgb]{0,0,0}\makebox(0,0)[lb]{\smash{$\alpha''_3$}}}%
    \put(0,0){\includegraphics[width=\unitlength,page=3]{fin-isotopie-A.pdf}}%
  \end{picture}%
\endgroup%

%% file: F-union-alpha.pdf_tex
\begingroup%
  \makeatletter%
  \providecommand\color[2][]{%
    \errmessage{(Inkscape) Color is used for the text in Inkscape, but the package 'color.sty' is not loaded}%
    \renewcommand\color[2][]{}%
  }%
  \providecommand\transparent[1]{%
    \errmessage{(Inkscape) Transparency is used (non-zero) for the text in Inkscape, but the package 'transparent.sty' is not loaded}%
    \renewcommand\transparent[1]{}%
  }%
  \providecommand\rotatebox[2]{#2}%
  \ifx\svgwidth\undefined%
    \setlength{\unitlength}{497.02548157bp}%
    \ifx\svgscale\undefined%
      \relax%
    \else%
      \setlength{\unitlength}{\unitlength * \real{\svgscale}}%
    \fi%
  \else%
    \setlength{\unitlength}{\svgwidth}%
  \fi%
  \global\let\svgwidth\undefined%
  \global\let\svgscale\undefined%
  \makeatother%
  \begin{picture}(1,0.49650053)%
    \put(0,0){\includegraphics[width=\unitlength,page=1]{F-union-alpha.pdf}}%
    \put(-0.00200411,0.04150582){\color[rgb]{0,0,0}\makebox(0,0)[lb]{\smash{$\pi^{-1}(\alpha\setminus F')$}}}%
    \put(0,0){\includegraphics[width=\unitlength,page=2]{F-union-alpha.pdf}}%
    \put(0.37871884,0.32239247){\color[rgb]{0,0,0}\makebox(0,0)[lb]{\smash{$\tilde \gamma$}}}%
    \put(0.49236549,0.40774898){\color[rgb]{0,0,0}\makebox(0,0)[lb]{\smash{$\tilde f(\tilde \gamma)$}}}%
  \end{picture}%
\endgroup%

%% file: unlinked-continua.pdf_tex
\begingroup%
  \makeatletter%
  \providecommand\color[2][]{%
    \errmessage{(Inkscape) Color is used for the text in Inkscape, but the package 'color.sty' is not loaded}%
    \renewcommand\color[2][]{}%
  }%
  \providecommand\transparent[1]{%
    \errmessage{(Inkscape) Transparency is used (non-zero) for the text in Inkscape, but the package 'transparent.sty' is not loaded}%
    \renewcommand\transparent[1]{}%
  }%
  \providecommand\rotatebox[2]{#2}%
  \ifx\svgwidth\undefined%
    \setlength{\unitlength}{385.30596385bp}%
    \ifx\svgscale\undefined%
      \relax%
    \else%
      \setlength{\unitlength}{\unitlength * \real{\svgscale}}%
    \fi%
  \else%
    \setlength{\unitlength}{\svgwidth}%
  \fi%
  \global\let\svgwidth\undefined%
  \global\let\svgscale\undefined%
  \makeatother%
  \begin{picture}(1,0.43829199)%
    \put(0,0){\includegraphics[width=\unitlength,page=1]{unlinked-continua.pdf}}%
    \put(0.60091479,0.3268159){\color[rgb]{0,0,0}\makebox(0,0)[lb]{\smash{$D$}}}%
    \put(0.22234282,0.07198766){\color[rgb]{0,0,0}\makebox(0,0)[lb]{\smash{$z_0$}}}%
    \put(0.48556042,0.27333342){\color[rgb]{0,0,0}\makebox(0,0)[lb]{\smash{$z_1$}}}%
    \put(0.90922542,0.30689106){\color[rgb]{0,0,0}\makebox(0,0)[lb]{\smash{$F_0$}}}%
    \put(0,0){\includegraphics[width=\unitlength,page=2]{unlinked-continua.pdf}}%
    \put(0.5224325,0.13608033){\color[rgb]{0,0,0}\makebox(0,0)[lb]{\smash{$I.z$}}}%
    \put(0.40635097,0.19396563){\color[rgb]{0,0,0}\makebox(0,0)[lb]{\smash{$z$}}}%
  \end{picture}%
\endgroup%